\documentclass[11pt,reqno]{amsart}
\usepackage{xifthen,setspace,bm}
 \usepackage{subcaption}
\usepackage{aligned-overset}
\usepackage[utf8]{inputenc}
\usepackage[T1]{fontenc} 
\usepackage{bm}
\usepackage{amsmath,amsthm,amsfonts,amssymb,mathrsfs,stmaryrd,bigints,graphicx}
\usepackage{dsfont}
\usepackage{mathtools}
\usepackage[usenames]{color}
\usepackage{thmtools, thm-restate}

\usepackage[colorlinks=true,linkcolor=blue]{hyperref}
\DeclareFontFamily{U}{matha}{\hyphenchar\font45}
\DeclareFontShape{U}{matha}{m}{n}{
      <5> <6> <7> <8> <9> <10> gen * matha
      <10.95> matha10 <12> <14.4> <17.28> <20.74> <24.88> matha12
      }{}
\DeclareSymbolFont{matha}{U}{matha}{m}{n}
\DeclareMathSymbol{\varleftrightarrow}{3}{matha}{"D8}
\DeclareMathSymbol{\nvarleftrightarrow}{3}{matha}{"DC}

\newcommand{\sfZ}{\mathsf{Z}}
\theoremstyle{plain}
\newtheorem*{informal-thm}{Informal theorem}
\theoremstyle{plain}
\newtheorem{informal-ppn}{Informal Proposition}

\usepackage{epstopdf} 
\usepackage{amssymb}
\usepackage{setspace}
\usepackage{enumerate}
\usepackage{bigstrut}
\usepackage{esvect}
\usepackage{multirow}
\usepackage{mathtools,xparse}
\usepackage{mathrsfs}
\usepackage{hyperref}
\usepackage{xcolor}
\newtheorem*{theorem*}{Theorem}
\newtheorem*{lemma*}{Lemma}
\newtheorem{theorem}{Theorem}[section]
\newtheorem{corollary}[theorem]{Corollary}
\newtheorem{lemma}[theorem]{Lemma}
\newtheorem{proposition}[theorem]{Proposition}

\theoremstyle{definition}

\newtheorem{remark}[theorem]{Remark}

\newcommand{\nin}{\noindent}

\allowdisplaybreaks
\usepackage{chngcntr}

\renewcommand{\P}{\mathbb{P}}

\newcommand\ba{\boldsymbol{a}}

\newcommand{\E}{{\mathbb{E}}}

\newcommand{\Z}{\mathbb{Z}}

\newcommand{\R}{\mathbb{R}}

\newcommand{\e}{\varepsilon}

\newcommand{\N}{\mathbb{N}}

	\renewcommand{\P}{\mathbb{P}}

\newcommand{\cC}{\mathcal{C}}

\newcommand{\cG}{\mathcal{G}}

\newcommand{\cL}{\mathcal{L}}

\newcommand{\cT}{\mathcal{T}}

\newcommand{\cZ}{\mathcal{Z}}

\newcommand{\wt}{\widetilde}

\newcommand{\f}{\frac}

\renewcommand{\setminus}{\backslash}

\newcommand\dd{\mathrm{d}}
\DeclareMathOperator{\Var}{Var}

\def\ba{\begin{align}}
\def\ea{\end{align}}
\def\bs{\begin{split}}
\def\es{\end{split}}

\begin{document}

\title{Logarithmic intermittency of the critical 2D SHF}
\author{Shirshendu Ganguly,   Kyeongsik Nam}

\begin{abstract}
While the solution to the $1+1$ dimensional stochastic heat equation with multiplicative noise is closely related to the exponential of a Brownian motion, owing to the sub-critical nature of the one-dimensional KPZ equation, the two-dimensional picture exhibits an additional weak-to-strong disorder transition as the disorder strength varies. In \cite{shfcritical}, the critical two-dimensional stochastic heat flow (SHF) was constructed as the scaling limit of the partition function of $2+1$ dimensional directed polymers under the logarithmic intermediate-disorder scaling at criticality. The SHF is not function-valued but is instead a random measure. Like many naturally occurring random measures, it is expected to exhibit rich intermittency, and \cite{shfsingularity} established that it is almost surely singular with respect to Lebesgue measure. More recently, \cite{loglog} showed that the logarithm of the SHF averaged over small balls is asymptotically Gaussian, with both its mean and variance diverging as the ball radius tends to zero. In this paper we prove a sharp result quantifying the singularity of the support of the SHF as well as its intermittency. In particular, we show that, almost surely, for all small $\e>0$, up to a vanishing error, all the mass of the point-to-plane SHF in any domain is concentrated on
{  ${1}/{\big(\e^2\log^{1/2+o(1)}(1/\e)\big)}$ balls of radius $\e$, each containing $\e^2{\log^{1/2+o(1)}(1/\e)}$ mass,}
thus precisely establishing its logarithmic fractal behavior. A key ingredient in the proof is a refined large-deviations theory, which allows access to conditional distributions, by taking advantage of the Gaussian-like behavior of the SHF at quasi-critical scales.  A further useful observation that features prominently is that conditioning a Brownian motion on its endpoint being unusually large essentially induces a shift in the mean of its increments, and consequently, at small enough scales, their distributions do not alter significantly. 
\end{abstract}

\address{Department of Statistics, UC Berkeley} 
\email{sganguly@berkeley.edu}

\address{Department of Mathematical Sciences, KAIST, South Korea}
\email{ksnam@kaist.ac.kr}

\maketitle
\vspace{-20pt}
 
\begin{figure}[!ht]
\centering
  \includegraphics[width=0.3\linewidth]{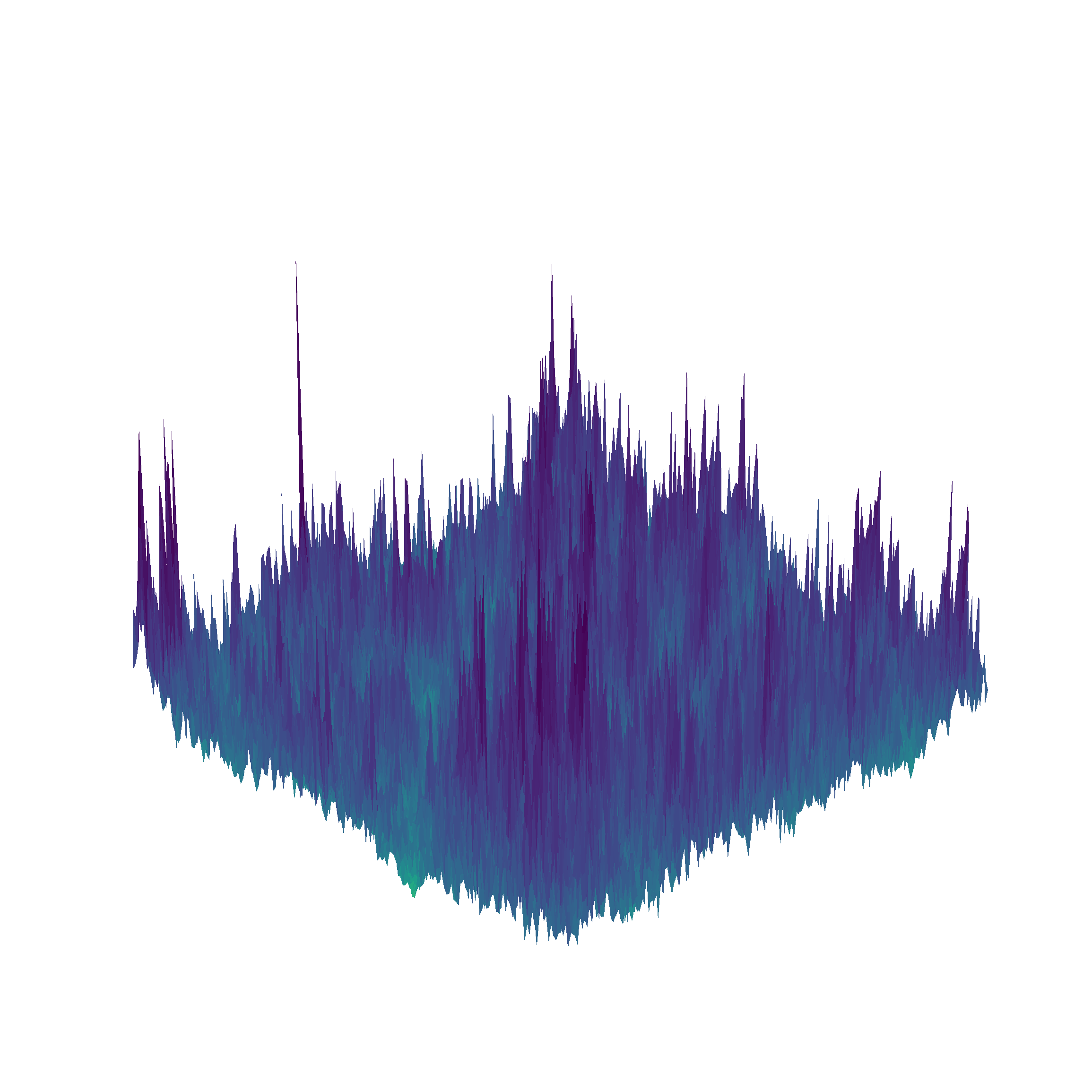}
  %\hfill
  \includegraphics[width=0.3\linewidth]{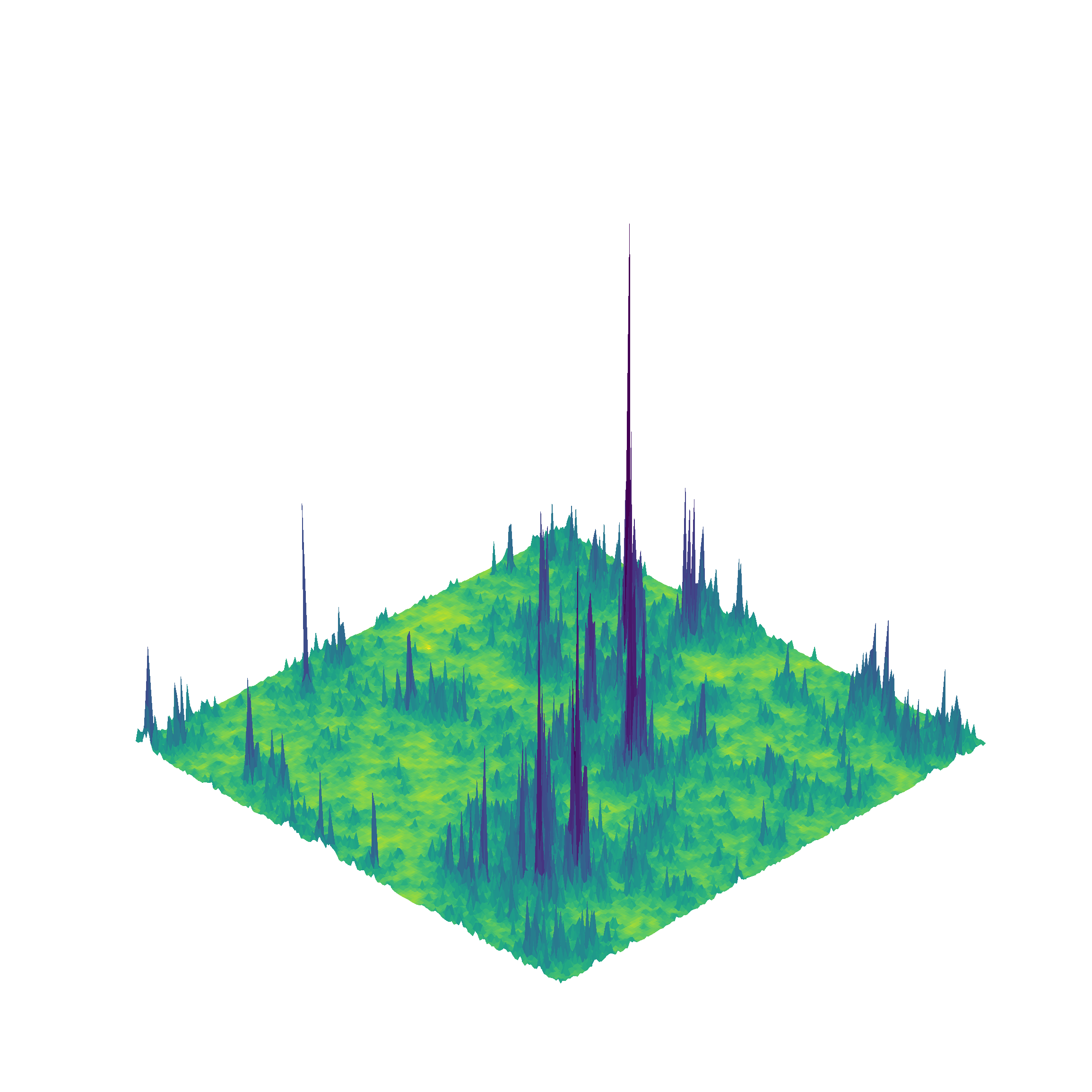}
  %\hfill
  \includegraphics[width=0.3\linewidth]{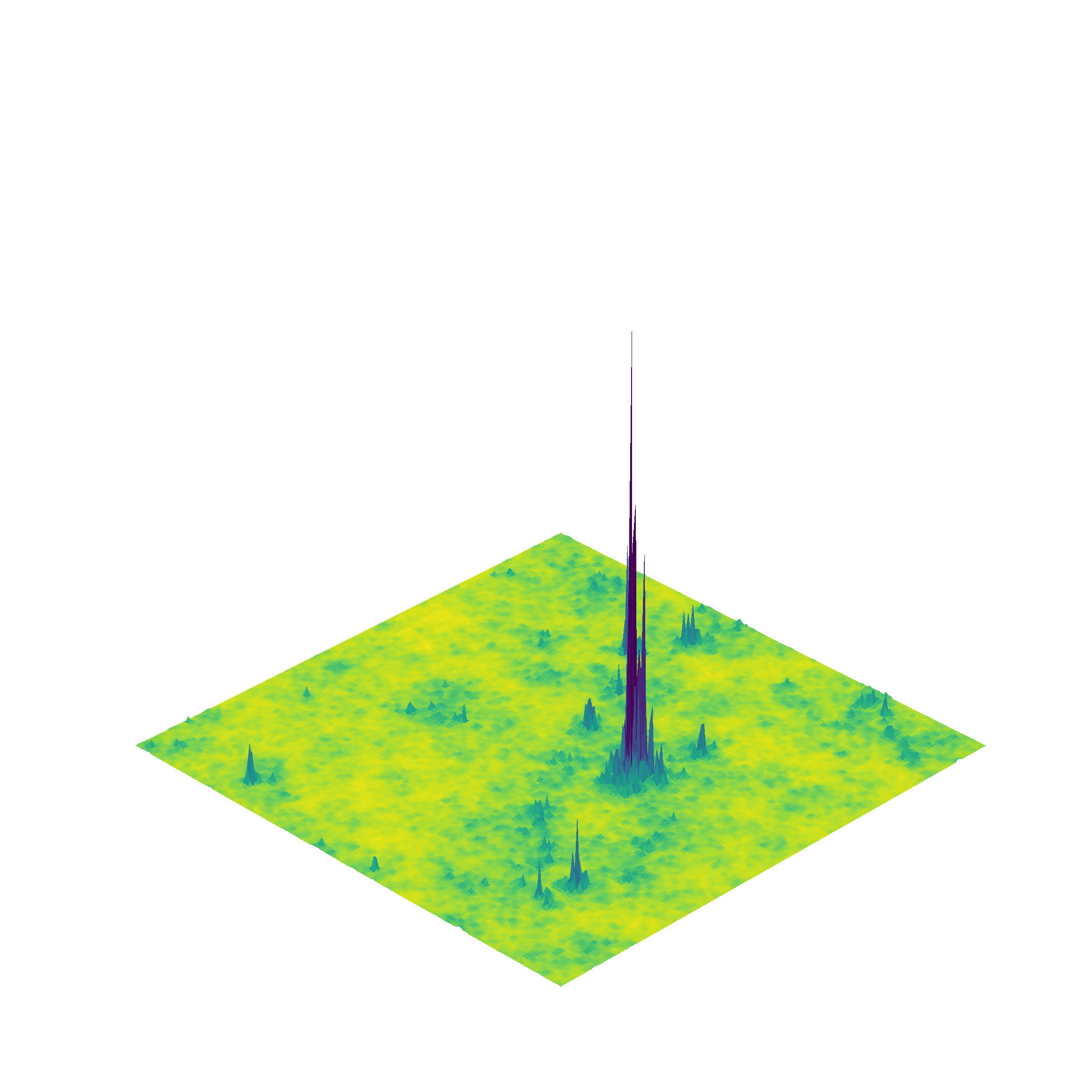}
\caption{
A simulation of the point-to-plane partition function for the
lattice directed polymer on $\mathbb{Z}^{2+1}$
(two spatial dimensions and one time dimension) of length $N=1000$, with the inverse temperature parameter $\beta_N$
in the critical window given by
$e^{\beta^2_N}-1 = \frac{1}{R_N} \Big(1+\frac{\vartheta+o(1)}{\log N} \Big)$.
Here, $\vartheta=-6,0,6$ from left to right, respectively.  $R_N$ denotes the expected number of collisions or overlap of two independent simple random walks up to time $N$, which grows logarithmically in $N$ (the precise asymptotics is recorded later around \eqref{choiceb}). As $\vartheta$ decreases, the surface
becomes progressively flatter and converges to the
random-walk heat kernel as $\vartheta\to-\infty$, whereas larger
positive values of $\vartheta$ produce increasingly singular
surfaces.
}
\label{fig:simulation}
\end{figure}

\clearpage

\setcounter{tocdepth}{2}{
  \tableofcontents
}

\section{Introduction and main results}
\nin
The following is the stochastic heat equation (SHE) in dimension $d$:
\begin{align}\label{SHE}
    \partial_t Z = \frac{1}{2}\Delta Z+\beta  \xi Z, \qquad t>0,   x\in \R^d,
\end{align}
where $\xi$ is space-time white noise.
Formally, the logarithm of the solution is the Cole-Hopf solution to the KPZ equation
\begin{align}\label{KPZ}
    \partial_t h = \frac{1}{2}\Delta h+ \frac{1}{2}|\nabla h|^2+ \beta  \xi, \qquad t>0,   x\in \R^d.
\end{align}
The roughness of the noise term coupled with the non-linearity makes the KPZ equation ill-posed. 
As is well known, dimension one has seen spectacular activities in the study of the above equations both through pre-limiting models and the lens of integrable probability \cite{amir2011probability,corwin2012kardar}, as well as through the theory of regularity structures \cite{hairer2014theory} and paracontrolled calculus \cite{gubinelli2015paracontrolled} which put forth a comprehensive theory of such singular stochastic PDEs (SPDEs) in dimension $1.$ A particular feature of the latter is that it is sub-critical, i.e., the non-linear term is asymptotically vanishing at small scales. This is the opposite of the super-critical case $d\ge 3$ where the effect of non-linearity blows up at small scales.
All of this makes the critical case of dimension $2$ particularly interesting as the effect of non-linearity stays the same at all scales.
While making sense of \eqref{KPZ} in $d=2$ and higher dimensions still remains a major open problem, the planar SHE \eqref{SHE} has led to a rather deep picture through a series of rigorous works \cite{bertini1998two, shfsubcritical0, shfcriticalmoment, chatterjee, shfsubcritical, gu,  shfuppermoment, shfcritical}. A natural way of making sense of \eqref{SHE} is by regularizing the noise. One then attempts to pass to a limit which usually involves a renormalization procedure. It turns out that if the mollification scale is $\e,$ then the renormalization associated with the intermediate disorder scaling involves taking $\beta$ to be ${\hat \beta}/{\sqrt{\log  \left({1}/{\e}\right)}}$ where $\hat \beta$ is $O(1).$
%\red{

As is well known, a formal solution to the SHE is given by the continuum counterpart of the partition function of the discrete directed polymer in a random environment (DPRE). 
We next turn to defining the latter to help base our upcoming discussion on an easier to visualize discrete lattice object.
Let
\begin{align}\label{polymer456}
    Z_{M,N}^{\beta}(x,y) := E\Big[  \exp\Big(\sum^{N-1}_{n=M+1}(\beta\omega(n,S_n)-\lambda(\beta))\Big) \mathbf 1_{\{S_{N}=y\}}   |  S_M=x\Big],
\end{align}
where $(S_n)_{n\geq 0}$ denotes a two-dimensional simple random walk, whose law 
and expectation 
are denoted 
by $P$ and 
$E$ respectively. and $(\omega(n,x))_{n\in \N, x\in \Z^2}$ is a family %%%%%
of i.i.d. random variables %%%%%%
with mean $0$ and variance $1$ %%%%%
having a finite log-moment generating function $\lambda(\beta):=\log\E\big[ e^{\beta\omega} \big]<\infty$ ($\forall \beta\in \R$) (we can assume that the variables are i.i.d. standard Gaussians). This is the %%%%%
discrete equivalent of a space-time white %%%%%
noise. Summarizing, $Z_{M,N}^{\beta}(x,y)$ is the partition function of the directed polymer model where each space-time lattice point $(n,x)$ is equipped with an i.i.d. variable $\omega(n,x).$ The directed polymer is obtained by tilting  a trajectory $\{S_{n}\}_{M\le n\le N}$, by the tilt factor $$ \exp\Big(\sum^{N-1}_{n=M+1}(\beta\omega(n,S_n)-\lambda(\beta))\Big),$$ where, following usual statistical mechanics terminology, the expression in the exponent is called the energy. For the reader less familiar with this topic, let us quickly mention that the parameter $\beta$ (inverse temperature) tunes the interaction between the base measure, i.e., the random walk, and the disorder. Thus, the smaller $\beta$ is, the closer the polymer measure is to the base measure while for larger values of $\beta$, the model approaches the zero temperature behavior of last passage percolation where the entire measure degenerates to a single path, the one with the {maximum} energy, often termed as the \emph{geodesic}. 
This is known as the weak-to-strong disorder transition.

While in dimensions three and higher, one indeed witnesses a transition as above as $\beta$ crosses a critical value $\beta_c$, in one and two dimensions it is expected that the entire interval $\beta >0$ is in the strong disorder regime and an interesting interpolating behavior is only witnessed when one zooms in on a critical window around $\beta=0.$ This is known as the intermediate disorder regime. 
For dimension $1$ and polymers of length $N,$ the intermediate disorder behavior is witnessed by choosing $\beta=\hat \beta N^{-1/4}$. It was shown in the foundational work \cite{alberts2014intermediate} that for any $\hat \beta>0,$ the discrete partition function $Z^{\hat \beta N^{-1/4}}_{0,N}(0, x\sqrt N)$ converges to the solution $Z^{\hat \beta}$ of the SHE \eqref{SHE} with $\hat \beta$ in place of $\beta$, regardless of the distribution of the disorder variables. This established the universality of the directed polymer in the intermediate disorder regime in dimension one.

In dimension $2,$ the correct scaling is logarithmic owing to the logarithmic blow up of the Green function. That is, one should instead take $\beta={\hat \beta} \frac{1}{\sqrt{\log N}}.$ 
Further, more interestingly, unlike dimension $1,$ inside the intermediate disorder regime, one witnesses a weak-to-strong disorder phase transition at $\hat \beta=\hat \beta_c$. The critical value is of course model dependent. While {it turns out that $\displaystyle \hat \beta_c = \sqrt{\pi}$ for the directed polymer,  in the continuous model obtained by a mollification of the SHE in \eqref{SHE}, one has $\displaystyle \hat \beta_c = \sqrt{2\pi}$. \cite[Appendix]{shfnotgmc} includes further details of comparison of the critical windows of these approximations which the interested reader is encouraged to refer to. 

The main object of investigation in this article is the solution to the SHE exactly at this critical point constructed in the important work \cite{shfcritical}. However, unlike in one dimension, the solution is not function-valued but rather measure-valued. While we will introduce this object formally shortly, let us first review briefly related developments prior to \cite{shfcritical}.
 Motivated by the works on the delta-Bose gas in \cite{albeverio}, in the pioneering paper \cite{bertini1998two} the first computations on the $2d$ SHE
in the critical window were carried out. The critical scaling was rediscovered in \cite{shfsubcritical0} who also discovered the phase transition. Curiously, this observation seems to be missing from \cite{bertini1998two}.
Subsequently, the sub-critical model, i.e., when {$\hat \beta <\hat \beta_c,$} has been studied extensively (see \cite{chatterjee, shfsubcritical0, gu, shfsubcritical}). The super-critical case still remains essentially out of reach though recent progress has been {made in a regime approaching  super-criticality \cite{berger2025strong}}.
 
Turning to the critical point, in \cite{shfcritical}, it was shown that the sequence of scaled fields (with the special choice of $\beta_N$ described shortly) 
\begin{align}\label{preSHF}
    \cZ_{N;  s,t}^{\beta_N}(\dd x, \dd y):=
    \f{N}{4}
    Z _{[  Ns ],[ Nt ]}^{\beta_N}
 (  \llbracket  \sqrt{N}x    \rrbracket
    ,
    \llbracket   \sqrt{N}y   \rrbracket
) 
    \dd x \dd y, \qquad 0\le s<t<\infty  ,
\end{align}
thought of as measures on $\R^2 \times \R^2$, converges and the limiting object was termed as the critical $2d$ stochastic heat flow (SHF).
Above, $[\cdot]$ maps a real %%%%%%
number to %%%%%%
its nearest even %%%%%%
integer neighbour, %%%%%%
$\llbracket\cdot\rrbracket$ maps $\R^2$ %%%%%%
points to 
their nearest even %%%%%%
integer point on %%%%%%
$\Z^2_\text{even}:=\{ (z_1,z_2)\in\Z^2:z_1+z_2\in 2\Z \}$, %%%%%%
and $\dd x  \dd y$ denotes the Lebesgue %%%%%%
measure on $\R^2 \times \R^2$. %%%%%%
The critical $\beta_N$ is chosen  %%%%%%
 such that
\begin{align}\label{choiceb}
\sigma_N^2:= e^{\lambda(2\beta_N)-2\lambda(\beta_N)}-1 = \frac{1}{R_N} \Big(1+\frac{\vartheta+o(1)}{\log N} \Big),
\end{align}
{where $\vartheta \in \R$ and $R_N$ denotes the expected number of collisions or overlap of two independent simple random walks up to time $N$. One also has the asymptotics: $R_N = \frac{\log N}{\pi} + c_0 + o(1)$ for some absolute constant $c_0$} (see \cite[Appendix A.1]{shfnotgmc} for details). Summarizing, the following is their main result (see also the recent survey \cite{caravenna2024critical}).

\begin{theorem*} \label{shfdefthm}
    Let $\beta_N$ be as in \eqref{choiceb} for %%%%%%
some fixed $\vartheta \in \R$ and %%%%%%
$\big( \cZ_{N;  s,t}^{\beta_N}(\dd x, \dd y) \big)_{0\le s<t<\infty}$ 
    %%%%%%
be defined as in \eqref{preSHF}. %%%%%%
Then as $N\rightarrow\infty$, the process of %%%%%%
random measures %%%%%%
    $(\cZ_{N; s,t}^{\beta_N}(\dd x,\dd y))_{0\le s\le t<\infty}$ %%%%%%
converges in %%%%%%
finite dimensional %%%%%%
distributions to %%%%%%
a unique limit
    \begin{align}\label{SHF5687}
        \mathscr{Z}^\vartheta
        =
        (\mathscr{Z}_{s,t}^{\vartheta}(\dd x,\dd y))_{0\le s\le t<\infty},
    \end{align}
    named the Critical %%%%%%
2d Stochastic Heat Flow.
\end{theorem*}

Subsequently, an axiomatic framework constructing the SHF and proving its uniqueness was put forth in \cite{tsai}.\\

\nin
\textbf{Intermittency in random measures.} Naturally occurring random measures often  exhibit rich fractal or intermittent behavior, typically being singular with respect to the natural measure on the base space. Sharp tail and decay of correlation estimates have facilitated the study of intermittent behavior in  $1+1$ dimensional polymer models and the associated solution to the stochastic heat equation. While the primary focus of this article will be the planar setting which we turn to shortly, let us nonetheless point the reader to {\cite{book,inter,inter2,ganguly,lil}} which review some of the developments.

A particularly canonical example of this is the Gaussian multiplicative chaos (GMC) obtained by exponentiating the Gaussian free field (GFF). Constructed rigorously in the seminal work \cite{kahane1985chaos}, it was subsequently shown by Duplantier and Sheffield in \cite[Proposition 3.4]{lqz} that the GMC measure is supported on the exceptional set of points where the GFF ``value'' (made sense of by taking local averages) is unusually high. Such points are termed as thick points, see e.g.  \cite{hu2010thick}. It is worth pointing out that the topological support is indeed the whole space since thick points form a dense set, but nonetheless, the measure is carried by the thick points in the sense that the complement of thick points has zero measure. We will continue to use the word support in this slightly loose sense throughout the rest of the article with the hope that it evokes the right picture. The interested reader is also referred to \cite{gmc} for a beautiful survey on the developments around the construction of GMCs. 
The preceding discussion already allows us to articulate the main goal of the paper as follows.

\begin{center}
\textit{ Prove a counterpart to the Duplantier-Sheffield result for the SHF.}
\end{center}

While, the SHF was shown not to be an exponential of a Gaussian process by establishing certain moment inequalities in \cite{shfnotgmc}, it still exhibits rather intriguing intermittent behavior, as is evident from Figure \ref{fig:simulation}. 
In this vein, several results have recently been established. Before getting to our main result we briefly review these. \\

\nin
\textbf{Growth of moments.}
A particular manifestation of intermittency is the rapid growth of its moments, a problem that has seen some important developments.  
In fact, the {tightness of the family $(\cZ_{N; 0,1}^{\beta_N}(\dd x,\dd y))_{N \in \mathbb N}$ follows immediately from a first-moment analysis, while the non-triviality of sub-sequential limits was first established by controlling the second and third moments of observables of the form
\begin{align}\label{flatdata} 
\mathscr{Z}^\vartheta_{t}(\varphi):=\int_{\R^2}  \varphi(x)   \mathscr{Z}_t^\vartheta(\dd x, \mathbf 1),
\text{ where}\quad
\mathscr{Z}_t^\vartheta(\dd x, \mathbf 1)
:=\int_{y\in \R^2} \mathscr{Z}_{0,t}^\vartheta(\dd x, \dd y),
\end{align}
and $\varphi$ is any smooth non-negative test function,
 in \cite{bertini1998two, shfcriticalmoment}.} It was predicted in the late nineties in {\cite{conjecture}} that such observables should in fact exhibit a double exponentially growing behavior of their moments, i.e., the $h^{th}$ moment should grow as $\exp(\exp (ch))$ for some constant $c = c(\vartheta)>0$.} 
In \cite{shfuppermoment}, using functional analytic tools an upper bound of $\exp(\exp(c h^2))$ was proven (see also \cite{chen1,chen2,chen3}).  On the other hand, a lower bound obtained as a consequence of the Gaussian correlation inequality \cite{gaussian1, gaussian2} is $\exp\left(c {h\choose 2}\right)$
{for some constant $c>0$,} 
and can be found, for instance, in \cite{shfnotgmc}. This stood as the state of the art until in recent work \cite{ganguly2025sharp} by the authors, a double exponential lower bound was established using a novel connection to the discrete Gaussian free field on Feynman diagrams and the connection to spanning trees via Kirchhoff's matrix-tree theorem \cite{matrixtree}. \\

\nin
\textbf{Probabilistic developments.}
While the preceding results pertain to the algebraic problem of moments, there also have been developments on the probabilistic side. For instance, \cite{shfsingularity} considered the SHF mass of a small ball of radius $\e,$ say $B=B(0,\e),$ i.e., the ball of radius $\e$ around the origin.  To be more concrete let
\begin{align} 
X_\e:= \frac{\int_{x\in B}\mathscr{Z}_1^\vartheta(\dd x, \mathbf 1)}{\pi \e^2}.
\end{align}
Note that $X_\e$ is the mass SHF assigns to $B$, i.e., the partition function of the polymer with one endpoint uniformly chosen inside $B$ and another completely free, relative to its Lebesgue measure. 
This normalization ensures that $X_\e$ has mean {of order $O(1)$} for all $\e>0$. 
In \cite{shfsingularity} it was shown that $X_\e \Rightarrow 0$ ($\Rightarrow$ denotes convergence in distribution) as $\e \to 0.$
This was done by setting up a comparison, using monotonicity of fractional moments, to a slightly sub-critical model on a renormalized space-time lattice, say, corresponding to  ${\hat \beta}/{\sqrt{\log\left(\frac{1}{\e}\right)}} $ where $\hat \beta$ is strictly less than $\hat \beta_c$. This allows one to reduce the analysis to the easier sub-critical regime where it is known already from \cite{shfsubcritical}, that as $\e \to 0$, $X_\e \Rightarrow W$ where $W$ is a log-Normal variable of mean one corresponding to a Gaussian with variance approximately  $\log \left(\frac{1}{1-\left({\hat \beta}/{\hat \beta_c}\right)^2}\right)$. Since such a log-Normal variable converges to zero in probability as $\hat \beta$ converges to $\hat\beta_c$ using the above comparison one can deduce the same for $X_\e$ in the critical case. This along with stationarity shows that the Critical $2d$ SHF is almost surely singular with respect to Lebesgue measure.

As we will see in the forthcoming section, a recent result of \cite{loglog} which is closely related to the theme of this paper proves again a log-CLT result but importantly directly in the critical regime where the variance of the associated Gaussian now blows up unlike in the sub-critical regime where it converged to $\log\left(\frac{1}{1-\left({\hat \beta}/{\hat \beta_c}\right)^2}\right).$

Given the above evidence of the singularity of the SHF measure and the intermittent structure revealed through moment growth as well as simulations (Figure \ref{fig:simulation}) several natural questions arise. \\

\nin
$(1)$
If most of the mass is carried by a sparse collection of 
exceptional peaks, how high are the peaks? \\

\nin
$(2)$
How large is the set of points exhibiting unusually high values, i.e., more precisely what is its dimension?\\

\nin
$(3)$
What is the spatial organization of these peaks? \\

\nin
$(4)$ What is the corresponding extreme value theory, the value of the maximum and its fluctuations?\\

\nin
As indicated, by now there is a thorough understanding of all the above questions for the planar GFF and the associated Liouville quantum {gravity measure (LQG) \cite{lqz, gmc, bramson2016convergence, biskup2016extreme}}. This paper aims to initiate a parallel program for the SHF and answers the first two questions in a strong sense. \\

\nin
We now turn to precisely defining all the objects in play and the statement of our main result.

\subsection{Main result}
Recall the SHF from \eqref{SHF5687}. Letting $E:=\mathbb R^2\times\mathbb R^2,$ 
and \(\mathcal M_+ = \mathcal M_+(E)\), the space of positive locally finite Borel
measures on \(E\), equipped with the Borel \(\sigma\)-field induced by the
vague topology:
\[
        \mathcal B _{\mathrm{vag}}(\mathcal M_+)
        :=
        \sigma\left(
        \mu\mapsto \int_E \varphi\,d\mu:
        \varphi\in C_c(E)
        \right).
\]
Then for any time interval $[s,t],$ each SHF increment $\mathscr Z^\vartheta_{s,t}$ 
is viewed as a random element in $(\mathcal M_+,\mathcal B _{\mathrm{vag}}(\mathcal M_+)).$ 
Throughout the paper,
for a kernel \(A = A(dx,dy)\) and test functions \(f,g\) on $\mathbb R^2$, following \cite{loglog}, we use the notation
\[
        f\blacktriangleleft A\blacktriangleright g
        :=
        \int_{\mathbb R^2\times\mathbb R^2}
        f(x)A(dx,dy)g(y),
\] 
and similarly define $
        f\blacktriangleleft A$ and $A\blacktriangleright g$ to be the natural measures on $\R^2.$ \\

 We are now in a position to state our main result.
\begin{theorem} 
\label{main1}
Let $\vartheta\in \mathbb R$ and  $
     \mathsf    Q:=[0,1]^2.$ Further, let  \(\xi\in (0,1/10)\) be a fixed constant which should be thought of as arbitrarily small. Then the following  statements hold almost surely, simultaneously for all $0<\e <1/2$.   
\begin{enumerate}
\item[\rm{(i)}]
There exists a collection
\(\mathcal B_\e\) of balls of radius \(\varepsilon\) such that  
\[
        |\mathcal B_\e|
        \le
        \varepsilon^{-2}(\log \varepsilon^{-1})^{-1/2+\xi},
\]
and 
\begin{align} \label{900}
    (\mathscr Z^\vartheta_{0,1}\blacktriangleright 1 )\left(
        \mathsf Q\setminus\bigcup_{B\in\mathcal B_\e}B
        \right)
        =o_{\e}(1).
\end{align}

\item[\rm{(ii)}] {Denoting by $\cC_\e$ a set which is a union of at
most
\[
        \varepsilon^{-2}(\log \varepsilon^{-1})^{-1/2-\xi}
\]
balls of radius \(\varepsilon\), we have 
\begin{align} \label{901}
   \sup_{\cC_\e}   \,\,(\mathscr Z^\vartheta_{0,1}\blacktriangleright 1)({\mathcal C}_\e\cap \mathsf Q)=o_{\e}(1),
      % \qquad\text{almost surely}.
\end{align}
where the supremum is taken over all such possible sets  $\cC_\e$.}
\end{enumerate}
\end{theorem}

A few remarks are in order. 

\begin{remark}  \label{re13}

By the strict local positivity result for the  SHF
\cite{positive}, the point-to-line measure
\(\mathscr Z^\vartheta_{0,1}\blacktriangleright1\) assigns positive mass to
every nonempty open ball, almost surely.
Since \(\mathsf Q\) has
nonempty interior, it follows that
\[
        \bigl(\mathscr Z^\vartheta_{0,1}\blacktriangleright1\bigr)(\mathsf Q)>0
        \qquad\text{almost surely}.
\]
Consequently,  in words, Theorem~\ref{main1} shows that the number of \(\varepsilon\)-balls needed to capture all of this strictly positive random mass, up to a vanishing (in $\e$)
error, is ``exactly''
\begin{equation}\label{covering68954}
        \varepsilon^{-2}
        (\log \varepsilon^{-1})^{-1/2+o(1)}.
\end{equation}
\end{remark}

\begin{remark}\label{raredisc2345}Theorem \ref{main1} has strong connections (through the lens of large deviations) to the recent result 
\cite{loglog} stating that the logarithm of the locally averaged, at scale $\e$, SHF  has Gaussian
fluctuations on the scale \(\sqrt{\log\log\varepsilon^{-1}}\). To state things precisely, let us denote the spatially averaged point-to-line mass at
microscopic scale \(\varepsilon\) by
 \begin{align} \label{ori}
        \mathcal Z_{\varepsilon }
        :=
        p(\varepsilon^2)\blacktriangleleft \mathscr{Z}^\vartheta_{0,1}\blacktriangleright 1 .
        \end{align}
Above $p(\e^2)$ denotes the heat kernel at time $\e^2$ (the precise definition appears later in \eqref{eq:heat-kernel-density-definition}).
Then, 
\begin{equation}\label{lognorm456}
        \frac{
        \log\mathcal Z_{\varepsilon }
        +
        \frac{1+ o_{\e}(1)}{2}\log\log\varepsilon^{-1}
        }{
        \sqrt{\log\log\varepsilon^{-1}}
        }
        \Longrightarrow
        \mathcal N(0,1),
        \qquad
         \varepsilon \rightarrow 0,
\end{equation}
where  $  \mathcal N(\mu, \sigma^2)$ is a Gaussian variable with mean $\mu$ and variance $\sigma^2$, and $o_\e(1)$ converges to $0$ as $\e \to 0$. 

Given this, Theorem \ref{main1} may be seen as the large deviation counterpart. The pointwise typical
smoothed density $\mathcal Z_{\varepsilon }$ is, in light of \eqref{lognorm456}, approximately a log-Normal variable $\exp (\sigma \mathsf{Z}-\frac{\sigma^2}{2})$ with $\sigma^2=\log\log\varepsilon^{-1}$ where $\mathsf{Z}\sim\mathcal  N(0,1).$ Note that the centering ensures that the log-Normal variable has mean one. However, observe that $\mathcal Z_{\varepsilon }$ is typically  of order $(\log\varepsilon^{-1})^{-1/2},$ 
and hence a typical \(\varepsilon\)-ball has SHF mass of order $\varepsilon^2(\log\varepsilon^{-1})^{-1/2}.$ Thus, the mean of the log-Normal variable is  obtained from the rare event that $ \mathsf{Z}\approx \sqrt{\log\log \varepsilon^{-1}}$ {which occurs with probability $1/(\log \varepsilon^{-1})^{1/2 + o(1)}.$}
This suggests that the SHF measure is supported on the sparse set of $\e$-balls where the above rare event is witnessed.  
The theorem makes this heuristic precise.
\end{remark}

\begin{remark} 
Theorem~\ref{main1} is a Minkowski-dimension type statement but with a crucial logarithmic correction. Such a refinement beyond the usual notion of Minkowski-dimension is necessary since the sparsity of the SHF support is only logarithmic in nature (the SHF is almost a function; in fact it is known to be in the negative H\"older space $\cC^{-\delta}$ for any $\delta>0$, see \cite{shfsingularity} for the formal statement as well as the precise definition of the space $\cC^{-\delta}$).
Indeed, the leading factor in the covering number is still
\(\varepsilon^{-2}\), and the characteristic information is conveyed only by the
logarithmic correction $(\log\varepsilon^{-1})^{-1/2+o(1)}.$ 
Thus, the ordinary notion of dimension which is designed to  detect only powers of
\(\varepsilon\) is not fine enough to capture such logarithmic sparsity. Another possible approach is to work with the notion of the Hausdorff measure associated with the following gauge function incorporating a logarithmic factor 
\[
        h_*(r)
        :=
        r^2(\log r^{-1})^{1/2}.
\]
However, this will require further refined understanding and, given the already substantial length of the paper, we do not pursue this direction here. {It is also worth mentioning that the above form of the covering number in \eqref{covering68954}, where the logarithmic correction to the main polynomial term is essential to capture the fractal property, is strongly reminiscent of the case of the trace of a planar Brownian motion \cite{taylor1964exact,le1990wiener}. In this case, the polynomial exponent is again $2$,} thereby revealing no information, but with a crucial pre-factor of $(\log \e^{-1})^{-1}$ (in place of the $(\log \e^{-1})^{-1/2}$ factor in our case), guided by the logarithmic nature of Brownian local time in two dimensions.
 Finally, fully logarithmic gauge functions, without any polynomial pre-factor, have appeared in the context of the critical Gaussian multiplicative chaos \cite{barral2015basic,barral2014critical}. \end{remark} 

\subsection{Idea of the proof}\label{iop}
While the proof is long and involves several steps, in this section we attempt to provide a broad overview of the basic intuition driving the proof and some of the central ingredients that will be developed to implement the proof strategy. The starting point is the recent result \eqref{lognorm456} in \cite{loglog} who proved that small ball averages of the SHF behave like a log-Normal variable in the sense that its logarithm under appropriate centering and scaling converges to a standard Gaussian variable as the ball radius converges to zero. 
 To begin discussing the key ideas of this paper it would be helpful to review the approach from \cite{loglog} which is what we start with along with a general discussion on why one might expect \eqref{lognorm456} to be true.

{Throughout the discussion, we will use $p(t):=p(t,\cdot)$ to denote the heat kernel at time $t$ on $\R^{2}$ (the formal definition appears in \eqref{eq:heat-kernel-density-definition} later).}

The SHF has a natural semigroup structure in time and the Gaussian structure emerges from the small time behavior. Namely, for  short time intervals, the SHF behaves sub-critically  or quasi-critically. In these regimes it is known that the SHF is in the Edwards-Wilkinson phase. In particular, in \cite{shfsingularity} it was shown that as $\e \to 0,$ letting $G$ denote $p(\varepsilon^2)\blacktriangleleft \mathscr{Z}^\vartheta_{0,\e^2}\blacktriangleright 1,$
\begin{equation}\label{quasicritical}
\sqrt{\log(\e^{-1})}\log G    \Longrightarrow
  c \mathsf Z,
\end{equation}
where $\mathsf Z$ is a standard Gaussian variable {and $c>0$ is a  constant} which can be computed explicitly. 
This along with the fact that $\E G=1$ suggests that $G$ is approximately the following log-Normal variable: $$G \approx 1+ \frac{c}{\sqrt{\log(\e^{-1})}}\sfZ\approx  \exp\Big(\frac{c}{\sqrt{\log(\e^{-1})}}\sfZ-\frac{c^2}{2\log(\e^{-1})}\Big)$$
(we will not attempt to precisely define $\approx$ in this informal discussion). 
In fact, this continues to hold for all quasi-critical times. We will let $t_i:=b^{-2i}\e^2$ for any fixed constant $b>0$ (this choice of the parametrization will be notationally convenient throughout the rest of the paper), and 
\begin{equation}\label{blockdef9860}
G_i=p({t_{i-1}})\blacktriangleleft Z_i\blacktriangleright 1
\end{equation}
where, for brevity, $Z_i:=\mathscr{Z}^\vartheta_{t_{i-1},t_{i}}$.
It then holds that {for a small enough constant $b>0$,}
{\begin{equation}
G_i \approx  \exp\Big(\frac{1}{\sqrt{N_{\e,b}-i}}\sfZ-\frac{1}{2(N_{\e,b}-i)}\Big),
\end{equation}
where $N_{\e,b}= \left\lfloor
        \frac{\log\varepsilon^{-1}}{\log b^{-1}}
        \right\rfloor,$ essentially as long as $t_i \ll 1$.
See Figure \ref{fig:decomposition}. Note that the $G_i$s are independent.  To use this, let us define a threshold $M$ to be specified later.  Then, noticing that a product of independent log-Normal variables is simply another log-Normal variable with the variance of the corresponding Gaussian just being the sum of the variances of the individual Gaussians, we have 
        \begin{equation}\label{clt21}
\frac{\sum_{i=1}^M\log G_i        +
        \frac{1}{2}\log \left(N_{\e,b}/[N_{\e,b}-M]\right)
        }{
        \sqrt{\log \left(N_{\e,b}/[N_{\e,b}-M]\right)}
        }
        \Longrightarrow
        \mathcal N(0,1),
\end{equation}
where we used that $$\sum_{i=1}^{M}\frac{1}{N_{\e,b}-i}\approx \log \big(N_{\e,b}/[N_{\e,b}-M]\big).$$
However, the object of interest for us is the full SHF flow counterpart and not the product of the individual blocks.   Let $\cZ_i:=p({\e^2})\blacktriangleleft \mathscr{Z}^\vartheta_{t_0,t_{i}}\blacktriangleright 1$.   
Then, to transfer the above result one needs to show \begin{align}\label{tensor12}
\cZ_i \approx \prod_{j=1}^i G_j=:\cG_i.
\end{align}
This is indeed plausible since in the quasi-critical regime $\mathscr{Z}^\vartheta_{t_{i-1},t_{i}}$ is not as singular and the effect of convolution is essentially the same as taking products.
 \begin{figure}[]
\centering
  \includegraphics[width=.4\linewidth]{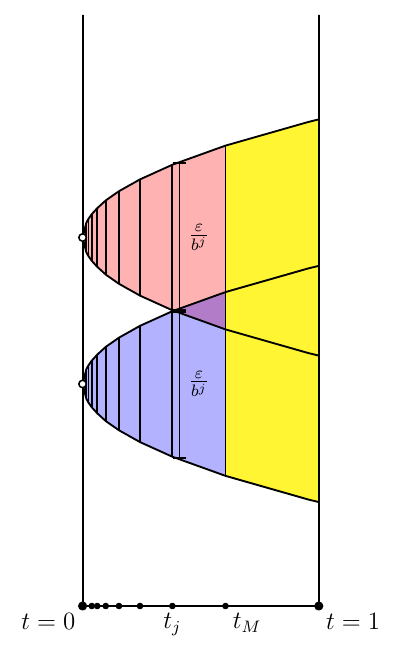}
\caption{
An illustration of the decomposition into blocks $G_i$ in \eqref{blockdef9860}. The points marked on the time axis correspond to the times $t_i=\e^2 b^{-2i}.$ While the block 
$G_i=p_{t_i-1}\blacktriangleleft \mathscr{Z}^\vartheta_{t_{i-1},t_{i}}\blacktriangleright 1$ actually depends 
on the entire noise in the time strip $[t_{i-1},t_i]$, since the heat kernel $p_{t_i-1}$ is primarily concentrated on the spatial window of width $\sqrt{t_{i-1}}$ and the expectation of the SHF $\mathscr{Z}^\vartheta_{t_{i-1},t_{i}}$ is again the heat kernel, the effective amount of noise that $G_i$ depends on takes the shape of a diffusive trapezoid whose two sides have lengths $\sqrt{t_{i-1}}$ and $\sqrt{t_{i}}$ and the width is $t_{i}-t_{i-1}.$  In the figure, blocks centered at two points spatially separated by $\e b^{-j}$ are shown. Consequently the blocks start having non-trivial overlap from index $j$ onwards. Finally, the threshold index $M$ is indicated where our analysis will treat the blocks between $t_{M}$ and $1$ as edge blocks (marked in yellow) and those to the left of $t_M$ as bulk blocks. Most of the analysis will focus on the latter due to them being strongly in the quasi-critical phase and consequently in the Gaussian regime.
}
\label{fig:decomposition}
\end{figure}
Then, one can choose $M$ to ensure two things.  First, \eqref{tensor12} stays valid. Second, $$\log [N_{\e,b}-M] \ll \log N_{\e,b},$$ so that, by combining these estimates, one can conclude 
\begin{equation}\label{clt211}
\frac{\log \cZ_M        +
        \frac{1+ o_{\e}(1) }{2}\log N_{\e,b}
        }{
        \sqrt{\log N_{\e,b}}
        }
        \Longrightarrow
        \mathcal N(0,1).
\end{equation}
Finally, one can show that if $M$ is close enough to $N_{\e,b},$ then 
\begin{equation}\label{compare2356}
\cZ_M \approx \cZ_{N_{\e,b}},
\end{equation}
and hence one can conclude that \eqref{clt211} holds even after replacing  $\cZ_M$ by $\cZ_{N_{\e,b}}.$
The above is the broad strategy implemented in \cite{loglog}. We now discuss in a bit more detail the technical content of the latter which will allow us to present the new ideas and techniques developed in this paper.\\

\nin
\textbf{Proof of \eqref{tensor12}}
In \cite{loglog}, it is shown that for $M\approx N_{\e,b}- \log N_{\e,b}$
\begin{equation}
\mathbb E\left[
\left|
\frac{\mathcal Z_{M}}{\cG_M}-
1
\right|^2
\right]\longrightarrow 0.
\end{equation}
To be completely precise, the above is only shown on a high probability event where $\cG_M$ is not too small, but we will ignore this technicality in this somewhat informal discussion of the ideas.

The proof of the above relies on an expansion of the ratio $\cZ_M/\cG_M$. This is obtained by using the following decomposition $$Z_i=(Z_i\blacktriangleright1)p_i+D_i$$  and expanding  the full partition function
\[
     \cZ_M=p(\varepsilon^2)\blacktriangleleft Z_1\bullet\cdots\bullet Z_{M} \blacktriangleright1.
\] 
Above $\bullet$ is a notion of a convolution made precise later in \eqref{conv34} (this is not the usual notion of convolutions as the objects being dealt with are singular and one has to resort to a definition involving mollification followed by passing to a limit). 

 Objects appearing in the expansion then are indexed by intervals $I$ with the corresponding contribution being
\begin{equation}\label{mainobject}
        D_I
        =
        p_{i_I-1}
        \blacktriangleleft
        D_{i_I}\bullet D_{i_I+1}\bullet\cdots\bullet D_{j_I-1}
        \bullet Z_{j_I}
        \blacktriangleright1 .
\end{equation}

The final expression one obtains is 
\begin{equation}\label{finalexp}
\cZ_M=\prod_{i=1}^M G_i + \sum_{\wt I\in \mathfrak I}\prod_{I\in \wt I}D_I \prod_{i\notin \wt I}G_i,
\end{equation}
where $\wt I$ is a collection of disjoint intervals which are subsets of $[1,M]$ and $ \mathfrak I$ is the set of all such collections. 
Dividing through by $\cG_M$ one obtains:
\begin{equation}\label{final34}
\frac{\cZ_M}{\cG_M}=1+ \sum_{\wt I\in \mathfrak I}\prod_{I\in \wt I}\frac{D_I}{G_I} ,
\end{equation}
where $G_I=\prod_{i\in I}G_i.$
The task is now reduced to showing the smallness of the $D_I$s and this is what \cite{loglog} shows. {Very recently, the above strategy coupled with a change-of-measure argument based on exponential tilting was employed in \cite{fractional} to obtain sharp estimates for the moments of $ \mathcal Z_{\varepsilon }
$ of fractional order \(h\in(0,1)\).}\\

Given this discussion, we now attempt to offer a quick overview of the key underlying ideas developed in this paper. Recall that for a large value of $\sigma^2$, the log-Normal $X\sim \exp (\sigma \sfZ-\frac{\sigma^2}{2})$ (where $\sfZ$ is a standard Gaussian distribution) has expectation $1$ but is typically of size $\exp (\sigma-\frac{\sigma^{2}}{2}) \approx 0.$ Thus the primary contribution to the mean is from a large deviation event where $\log X \approx \frac{1}{2}\sigma^2$ which occurs with probability approximately $\exp (-\frac{\sigma^2}{2}).$\\

\nin
\textbf{Upper-tail large deviations}
This motivates us to study the large deviation behavior for $\cZ:=\cZ_{N_{\e,b}}.$ In particular, we will consider the following upper-tail event 
\begin{align} \label{utdef}
        \left\{
        \cZ      \ge
      N_{\varepsilon,b}^{-\frac{1}{2} + \alpha}
        \right\}
\end{align}
where $\alpha>0$. Since, as indicated above, by \eqref{clt211}, the typical value of $\cZ$ is $N_{\varepsilon,b}^{-\frac{1}{2}}$, $\alpha$ parametrizes the various large deviation events at the polynomial (in $N_{\varepsilon,b}$) scale.\\

Continuing to draw a parallel to a log-Normal variable, one then expects, 
\begin{align}\label{firstappear00293}
        \mathbb P\left(\cZ
        \ge  N_{\varepsilon,b}^{-\frac{1}{2}+\alpha} 
        \right)
=
        N_{\varepsilon,b}^{-\frac{\alpha^2}{2}+o(1)}.
\end{align}
However, we will not directly prove the above statement (we expand on this further later in Remark \ref{largedevz}) but a proxy for this via analyzing $\cG_M.$ 
For \(\alpha>0\), define the bulk upper-tail large deviation event
\begin{align} \label{utdef87864}
\mathrm{UT}_{\varepsilon}(\alpha)
        :=
        \left\{
        \prod_{i=1}^MG_{i}        \ge
      N_{\varepsilon,b}^{-\frac{1}{2} + \alpha}
        \right\},
\end{align}
where $M=N_{\e,b}-N_{\e,b}^{o(1)}$. Here $o(1)$ is used to denote that the gap $N_{\e,b}-M$ will be chosen to be a small power of $N_{\e,b}$.  This will be crucial for us and the  small power will be amplified and made effective by invoking arbitrarily high-moment concentration of the $G_i$s near $1$ when $i\le M$. This technical point will become clearer later through Lemma \ref{goodevent} and the discussion preceding it.\\

The argument in this paper hinges on the following two key statements stated informally for the purposes of the current discussion.
\begin{informal-ppn}[Sharp large deviations for $\cG_M$]
 \begin{align}  \label{dec9878}
        \mathbb P\left(\mathrm{UT}_{\varepsilon}(\alpha)        \right)
=
        N_{\varepsilon,b}^{-\frac{\alpha^2}{2}+o(1)}.
\end{align}\end{informal-ppn}

\begin{informal-ppn}[Conditional and unconditional comparison of $\cG_M$ and $\cZ$] For any $\alpha \ge 0$ (where the $\alpha=0$ case essentially corresponds to no conditioning)
\begin{equation} \label{cond8787}
\mathbb E\left[
\left|
\frac{\mathcal Z}{\cG_M}\right|^2 \ 
\Big\vert  \  \mathrm{UT}_{\varepsilon}(\alpha) 
\right]= N_{\e,b}^{o(1)}.
\end{equation}
Note that above the numerator is $\cZ=\cZ_{N_{\e,b}}$, the object of interest, and not $\cZ_M$, while the denominator is $\cG_M$. Hence, with the already indicated choice of $M$, one cannot expect to have a better bound than $N_{\e,b}^{o(1)}.$
\end{informal-ppn}

Given the above, one can deduce the following mass concentration result, from which Theorem \ref{main1} can be concluded without too much difficulty.

\begin{informal-ppn} 
Let \(\zeta\in(0,1/10)\) be an arbitrary constant. 
Then there exists $\upsilon_\zeta>0$ such that for all  sufficiently small $\varepsilon>0,$
\begin{equation}
        \mathbb E\left[
        \mathcal Z \mathbf 1_{ \mathcal Z < (\log \varepsilon^{-1})^{1/2-\zeta}}
        \right]
        +
        \mathbb E\left[
        \mathcal Z \mathbf 1_{ \mathcal Z > (\log \varepsilon^{-1})^{1/2+\zeta}} 
        \right] \le (\log \varepsilon^{-1})^{-\upsilon_\zeta}.
\label{eq:Z-mass-concentration-final-Geta 00}
\end{equation}
\end{informal-ppn}

While the proof of \eqref{eq:Z-mass-concentration-final-Geta 00} using \eqref{dec9878} and \eqref{cond8787} is not completely straightforward, the majority of the new ideas developed in the paper go into the proofs of the first two propositions. In the remainder of the section, we offer a glimpse of some of the key novelties. \\

\nin
\textbf{Proof of \eqref{dec9878}}. Letting $Y_i=\log G_i$, the upper-tail deviation for products of the $G_i$s is equivalent to that of the sum of the $Y_i$s. The analysis now proceeds by obtaining rather refined estimates on the log-Laplace transform $\lambda_i(\theta)$ of $Y_i$. That is, $\lambda_i(\theta)=\log \E(\exp(\theta Y_i))=\log \E(G_i^{\theta}).$ Now, here one encounters a usual difficulty. Namely, for polymer models, one has access to only integer moments via their connections to random walk local time, see e.g., \cite{shrinking} whereas the above demands estimating $\E(G_i^{\theta})$ for any $\theta>0$. 

What comes to our aid is that for each $i\le M,$ owing to quasi-criticality, $G_i$ is in fact concentrated near $1$ and thus Taylor expanding $G^{\theta}_{i}$ allows us to transfer sharp second moment estimates along with control on higher integer moments to sharp enough estimates of $\lambda_i(\theta)$ for a general $\theta.$ While the remaining details to go from this to \eqref{dec9878} are involved and beyond the scope of the present discussion, let us briefly mention that the key step involves suitably exponentially tilting the measure along with a quantitative central limit theorem for the sum of the $Y_i$s under the tilted measure. \\

\nin
\textbf{Proof of \eqref{cond8787}}. Starting with the expansion in \eqref{final34}, the proof proceeds by estimating the $L^2-$norm of $\prod_{I\in \wt I}\frac{D_I}{G_I}$ where ${\wt I\in \mathfrak I}$ \emph{conditionally} on $\mathrm{UT}_{\varepsilon}(\alpha).$ Unconditional estimates already developed in \cite{loglog} show that the norm decays exponentially in the length $|\wt I| :=\sum_{I\in \wt I}|I|.$ Thus such estimates continue to hold even in the conditional setting provided that $|\wt I|\gg \log N_{\e,b},$ since the conditioning event, by \eqref{dec9878}, has probability polynomial in $N_{\e,b}.$

However, the main idea here goes into controlling the terms where $|\wt I|\ll \log N_{\e,b}$ and the above naive approach is not effective. The key observation is the following.  
Suppose $X_1,X_2, \ldots, X_M$ are independent random variables with $X_i \sim \mathcal N(0,\frac{1}{N_{\e,b}-i}).$ Then conditioning $\sum_{i}X_{i}\ge \alpha \log N_{\e,b}$ for some $\alpha>0$ has roughly the effect of shifting the mean of the $i^{th}$ variable by $\alpha/(N_{\e,b}-i).$  Assuming $\alpha=1,$ for notational simplicity,  for a small subset $U \subset \{1,2,\ldots,M\}$, the Radon-Nikodym derivative of the conditional distribution with respect to the unconditional distribution, projected onto the variables indexed by $U$ typically looks like 
\begin{align}
{\mathrm{RN}}_U:=\exp\left(\sum_{i\in U} \Big[{X_i}-\frac{1}{2(N_{\e,b}-i)}\Big]\right)\approx 1+\sum_{i\in U}{X_i}.
\end{align}
Above, we used that the Radon-Nikodym derivative of $\mathcal N(\frac{1}{j}, \frac{1}{j})$ with respect to $\mathcal N(0,\frac{1}{j})$ at $x$ is $\exp(x-\frac{1}{2j}).$
Thus,
 $$\E\left[\big({\mathrm{RN}}_U-1\big)^2\right]\approx \sum_{i\in U}\Var (X_i)=\sum_{i\in U}\frac{1}{N_{\e,b}-i}\le \frac{|U|}{N_{\e,b}-M},$$ and hence as long as $|U|\ll \log N_{\e,b}$ and $ N_{\e,b}-M\gg \log N_{\e,b}$, one can conclude that $\E\left[\big({\mathrm{RN}}_U-1\big)^2\right]$ is small.

In the application, we show that the above intuition applies even to the random variables $\log G_i$ which are approximately Gaussian. The tool that facilitates this is a rather sharp understanding of the large deviation probabilities which allows us to characterize conditional distributions. 
Most importantly, this control of the conditional density allows us to transfer unconditional estimates on $\prod_{I\in \wt I}D_{I}$ to conditional ones when $|\wt I|$ is small.

While several other issues remain, we will refrain from elaborating on them in this informal discussion and instead proceed to the main body of the paper.

\subsection{Acknowledgements}
SG thanks Francesco Caravenna, Rongfeng Sun and Nikos Zygouras for various useful discussions regarding the SHF and  
Francesco Caravenna again for sharing his code which was used to generate the simulations of the polymer partition function displayed in Figure \ref{fig:simulation}. 
The authors also thank Zoe Himwich and Vilas Winstein for  help with the figures.
SG was supported by a Miller Research Professorship at the Miller Institute for Basic Research in Science and NSF Career grant-1945172. KN was supported by
Samsung Science and Technology Foundation under Project Number SSTF-BA2202-02. This project was initiated during  KN’s visit to UC Berkeley in the Spring of 2026.

\section{Key intermediate statements}
The proof of Theorem \ref{main1} relies on some key intermediate statements as indicated in Section \ref{iop}. These are recorded here, but we first introduce the necessary preliminary setup.

\subsection{Critical 2D SHF} \label{sec2.1}
We start with the formal definition of the critical two-dimensional stochastic heat flow (SHF) following \cite{tsai}. It will be helpful to set up some notation first.  \\

\nin
\textbf{Heat kernel.} For \(0\le s<t\), let \(P_{s,t}\) denote the 2D heat kernel from time \(s\) to
time \(t\).  More precisely,
\begin{equation}
        P_{s,t}(dx,dy)
        :=
        p(t-s,y-x)\,dxdy,
        \qquad x,y\in\mathbb R^2,
\label{eq:heat-kernel-Pst-definition}
\end{equation}
where $p(t,x)$ denotes the 2D heat kernel
\begin{equation}
        p(t,x)
        :=
        \frac{1}{2\pi t}
        \exp\left\{
        -\frac{|x|^2}{2t}
        \right\},
        \qquad t>0,\ x\in\mathbb R^2.
\label{eq:heat-kernel-density-definition}
\end{equation}
Often, for brevity at the cost of a slight abuse of notation, for every fixed $t>0$, we will use $p(t)$ to denote $p(t,\cdot)$.\\

\nin
\textbf{Random measures, vague topology.}
Let \(\mathcal M_+ = \mathcal M_+(\mathbb R^2  \times \mathbb R^2)\) be  the space of positive locally finite measures
on  $\mathbb R^2  \times \mathbb R^2$, equipped with the $\sigma$-algebra generated by a vague topology:
\[
        \mathcal B _{\mathrm{vag}}(\mathcal M_+)
        :=
        \sigma\left(
        \mu\mapsto \int_{\mathbb R^2  \times \mathbb R^2} \varphi\,d\mu:
        \varphi\in C_c(\mathbb R^2  \times \mathbb R^2)
        \right).
\]

\nin
With the above preparation we can now define the SHF. 
For a coupling parameter \(\vartheta\in\mathbb R\), the
critical $2d$ stochastic heat flow (SHF) is a stochastic process 
\begin{align}\label{SHF23}
\mathscr Z^{\vartheta}&=\{\mathscr Z^{\vartheta}_{s,t}\}_{0<s\le t<\infty}
\end{align}
such that
\begin{enumerate}
\item 
$\mathscr Z^{\vartheta}$ is an $\mathcal M_+$-valued, {continuous process} (in the vague topology) in $(s,t)$ where $s\leq t\in\R$.
\item  (Chapman-Kolmogorov) For  $s< t< u$, $\mathscr Z^{\vartheta}_{s,t}\bullet  \mathscr Z^{\vartheta}_{t,u}=\mathscr Z^{\vartheta}_{s,u}$. Here $\bullet$ denotes the convolution operation introduced in \cite{clark2024continuum} extending the usual convolution of kernels to the present singular setting by introducing a mollification procedure followed by passing to the limit. Namely,
\begin{align}\label{conv34}
\mathscr Z^{\vartheta}_{s,t}\bullet  \mathscr Z^{\vartheta}_{t,u}:=\lim_{\delta\to 0}Z^{\vartheta}_{s,t}\bullet_{\delta}\mathscr Z^{\vartheta}_{t,u},
\end{align}
where 
\begin{align}
\mathscr Z^{\vartheta}_{s,t}\bullet_{\delta}\mathscr Z^{\vartheta}_{t,u}({\rm{d}}x, {\rm{d}}x')=\int_{\R^{4}}\mathscr Z^{\vartheta}_{s,t}({\rm{d}}x, {\rm{d}}y)p(\delta^2, y-y')\mathscr Z^{\vartheta}_{t,u}({\rm{d}}y', {\rm{d}}x').
\end{align}
Above $p(\delta^2, \cdot)$ is the heat kernel from \eqref{eq:heat-kernel-density-definition}. Further, note that the existence of the limit in \eqref{conv34} is by no means obvious and indeed is a part of the statement in \cite{clark2024continuum}.

\item (Independent increments)
For  $s< t< u$, $\mathscr Z^{\vartheta}_{s,t}$ and $\mathscr Z^{\vartheta}_{t,u}$ are independent.
\item  (Moments)
For $n=1,2,3, 4$, $s<t$, and $x=(x_1,\ldots,x_n),x'=(x'_1,\ldots,x'_n)\in (\R^{2})^n$,
\begin{align}
    \E\, \bigotimes_{i=1}^n \mathscr Z^{\vartheta}_{s,t}(d x_i,d x'_i)
    =
    d x\, d x' \, Q^{n,\vartheta}(t-s,x,x'),
\end{align}
where $Q^{n,\vartheta}(t,x,x')$ is the integral kernel of the 2D delta-Bose semigroup  $Q^{n,\vartheta}(t)$. We will not, however, require the precise definition of the latter and consequently not state it and instead refer the interested reader to \cite{tsai}.
\end{enumerate}

{In the degenerate case \(s=t\), SHF reduces
to Lebesgue measure supported on the diagonal of
\(\mathbb R^2\times\mathbb R^2\), i.e. the pushforward of Lebesgue measure on \(\mathbb R^2\) under the diagonal
embedding \(x\mapsto(x,x)\).}

\subsection{Key notations}
\label{subsec:setup}

The time scale $\e^2$  (corresponding to spatial scale $\e$) will be often termed \emph{microscopic} in analogy to a single point in the lattice case. Following the discussion in Section \ref{iop} we will now decompose the interval of times from the microscopic scale
\(\varepsilon^2\) to a macroscopic time of order one into geometric scales  growing at rate $b^{-2}$ given by a small enough constant \(0<b<1\) chosen later.
These times will be denoted by  
\[
        t_i=t_{i,\varepsilon,b}:=\varepsilon^2 b^{-2i},
        \qquad
        i=0,1,\ldots,N_{\varepsilon,b},
\]
where
\begin{align}\label{ndef}
        N_{\varepsilon,b}
        :=
        \left\lfloor
        \frac{\log\varepsilon^{-1}}{\log b^{-1}}
        \right\rfloor .
\end{align}
Throughout the paper, \(b>0\) is regarded as a small but fixed constant, while
\(\varepsilon \rightarrow 0\).  Thus \(N_{\varepsilon,b}\) is comparable to \(\log\varepsilon^{-1}\), with the comparison constant
depending on \(b\).  
  Later, we will also consider a slightly more general
setting where both \(b\) and the coupling parameter \(\vartheta\) in \eqref{SHF5687} will mildly depend on
\(\varepsilon\), in a way that they converge to certain limiting values as
\(\varepsilon\downarrow0\).

Further, for \(1\le i\le N_{\varepsilon,b}\), define the $i$-th block SHF and the heat kernel by
\begin{equation}\label{keyblock}
Z_i:=\mathscr{Z}^\vartheta_{t_{i-1},t_i},\qquad  p_i:=p(t_i,\cdot),\qquad P_{i-1,i} := P_{t_{i-1},t_i}.
\end{equation}
Just to clarify, above, the first term in a random measure on $\R^2 \times \R^2$, the second is a function on $\R^2$ and the third is a {measure} on $\R^2\times \R^2.$
Note that we suppress the coupling parameter $\vartheta$  in $Z_i$ for the sake of simplicity.
Define
the partition function associated with the \(i\)-th block $[t_{i-1},t_i]:$
\begin{equation}\label{tensoredblocks}
        G_i:=p_{i-1}\blacktriangleleft Z_i\blacktriangleright 1.
\end{equation}
Note that 
\begin{align} \label{mean1}
    \mathbb E G_i  = p_{i-1}\blacktriangleleft  P_{i-1,i} \blacktriangleright 1  =1.
\end{align}  Since the SHF increments on disjoint time intervals are independent,
the variables \((G_i)_{1\le i\le N_{\varepsilon,b}}\) are independent. \\

We next define the partition  function 
\begin{align} \label{auxpar}
    \mathcal Z_{\varepsilon,b}:= 
    p(\varepsilon^2)\blacktriangleleft Z_1\bullet Z_2\bullet\cdots\bullet Z_{N_{\varepsilon,b}-1}\blacktriangleright 1  =   p(\varepsilon^2)\blacktriangleleft  \mathscr Z^\vartheta_{\varepsilon^2, t_{N_{\varepsilon,b} - 1} }\blacktriangleright 1 .
\end{align} 
Note that above we only consider the noise up to $ t_{N_{\varepsilon,b} - 1}$ instead of  $t_{N_{\varepsilon,b}}.$ This is purely a convention to ease some of notation appearing in the proofs. We will expand on this again shortly (see Remark \ref{elaborate1450}) once we have had a chance to record some of our estimates. However, a scaling argument later in
Corollary~\ref{corcorcor} will allow us to pass from ${\mathcal Z_{\varepsilon,b}}$ to the
original \([0,1]\)-object $\cZ_\e$ (defined in \eqref{ori}).
 \\

\nin
Carrying on with further notational groundwork, throughout this paper, an
interval means an interval of integers. If
\[
        I=[i_I,j_I]\subset[1,N_{\varepsilon,b}]\cap\mathbb Z,
\]
then
\[
        |I|:=j_I-i_I+1,
\] where $|I|$ denotes the cardinality of $I.$ A recurring object will be  \(\mathfrak I_{\varepsilon,b}\), the set of all
nonempty finite collections \(\widetilde I\) of disjoint intervals 
\begin{align} \label{defii}
        I\subset [1,N_{\varepsilon,b}-1],
        \qquad |I|\ge2.
\end{align} 
Again, that we are only considering indices up to $N_{\varepsilon,b}-1$ is guided by the discussion following \eqref{auxpar}.
That is, a generic element $\wt I \in \mathfrak I_{\varepsilon,b}$ is a collection $\{I_1 <I_2< \ldots< I_k\}$ for some $k,$ where $I_j=[a_j,b_j]$ and $I_j<I_{j+1}$ is used to imply $b_j< a_{j+1}.$
For  \(\widetilde I \in \mathfrak I_{\varepsilon,b}\), define its total length by
\begin{align} \label{idef1}
        |\widetilde I|
        :=
        \sum_{I\in\widetilde I}|I|,
\end{align}
and write
\begin{align}\label{idef2}
        \#\widetilde I
\end{align}
for the number of disjoint intervals in the collection, i.e., for $\widetilde I=\{I_1 <I_2< \ldots< I_k\}$, we have $\#\widetilde I=k$. For \(\widetilde I \in \mathfrak I_{\varepsilon,b}\) and an interval $I$, we write
\begin{align} 
        G_{\widetilde I}:=\prod_{I\in \widetilde  I}G_I,\qquad G_I:=\prod_{i\in I}G_i .
\end{align}    

While the main result of this section will be stated shortly, we first provide some further motivation behind our choices of the parameters as well as develop some further notation.  
Recall from Section \ref{iop}, say \eqref{clt21}, the threshold $M$. We define this precisely now.  Given
\(0<\eta<1\), we let
\begin{align} \label{ellm}
        \ell_{\varepsilon,\eta}
        :=
        \left\lfloor
        (\log\varepsilon^{-1})^\eta
        \right\rfloor,
        \qquad
        M_{\varepsilon,\eta}
        :=
        N_{\varepsilon,b}-\ell_{\varepsilon,\eta},
\end{align}
and introduce the bulk and edge (or terminal) index set respectively as follows: 
\[
[1,M_{\varepsilon,\eta}]  \ \text{ and }  \ [M_{\varepsilon,\eta}+1,N_{\varepsilon,b}-1].
\]
This separation guarantees that for every
bulk index \(i\le M_{\varepsilon,\eta}\), the distance to the terminal scale is
at least \(\ell_{\varepsilon,\eta}\) thus ensuring strong enough concentration of 
\( G_i\) around $1$ (the upper-tail event we will condition on will also only involve the bulk variables. 
For \(\widetilde I\subset[1,N_{\varepsilon,b}-1]\), we will write
\begin{equation}\label{bulk123}
       \widetilde  I^{\rm bulk}:= \widetilde  I\cap
[1,M_{\varepsilon,\eta}],
        \qquad
        G_{\widetilde I}^{\rm bulk}
        :=
        \prod_{i\in \widetilde  I^{\rm bulk}}G_i,
\end{equation}
with the convention that an empty product is equal to \(1\). Thus
\(G_{\widetilde I}^{\rm bulk}=G_{\widetilde I}\) if \(\widetilde I\subset[1,M_{\varepsilon,\eta}]\), while \(G_{\widetilde I}^{\rm bulk}=1\) if
\(I\subset[M_{\varepsilon,\eta}+1,N_{\varepsilon,b}-1 ]\).
{For an interval $I$ we will often abuse notation and also consider it to be an element $\wt I$ consisting of the single interval $I.$}

~

The next definition prescribes an event on which every single-block partition function, in the bulk regime, on account of their proximity to one  is stipulated to be bounded away from zero:
\begin{align} \label{nice}
\Omega_{\varepsilon,\eta}
        :=
        \bigcap_{i=1}^{M_{\varepsilon,\eta}}
        \left\{G_i\ge\frac12\right\}.
\end{align}
This will be shown to be a high probability event (with failure probability an arbitrarily large power of $\log \e^{-1}$). {We also remark that we will 
frequently use the shorthand}
\begin{align} \label{loglog}
\mathcal L_\varepsilon:=\log\log\varepsilon^{-1}.
\end{align} 

We next define the following bulk and terminal products 
\[
        G_{\eta;1}:=\prod_{i=1}^{M_{\varepsilon,\eta}}G_i,
        \qquad
        G_{\eta;2}:=\prod_{i=M_{\varepsilon,\eta}+1}^{N_{\varepsilon,b}-1}G_i .
\]
Here \(G_{\eta;1}\) contains the factors in the bulk region and \(G_{\eta;2}\)
contains the remaining terminal factors. Thus \(G_{\eta;1}G_{\eta;2}=G_{[1, N_{\e,b-1}]}\) is the
fully decoupled product of partition functions.
We next define the following  upper-tail large deviation event.
For $\alpha>0,$ let
\begin{align} \label{utdef11}
\mathrm{UT}_{\varepsilon,\eta}(\alpha)
        :=
        \left\{
        G_{\eta;1}        \ge
      N_{\varepsilon,b}^{-\frac{1-\eta}{2} + \alpha}
        \right\}.
\end{align}
This event depends only on the bulk variables $(Z_1,\ldots,Z_{M_{\varepsilon,\eta}}).$ \\

Given the above, we have the three following statements which are the precise versions of the informal statements appearing earlier in Section \ref{iop}.

\subsection{Intermediate propositions}
The key upper-tail large deviation result (the formal version of  \eqref{dec9878}) is as follows.

\begin{proposition}\label{largedeviation45}
For any $\alpha>0$ and $\eta \in (0,1/10)$,   there exists
\(\widetilde \nu_b\ge0\) (depending on $\alpha$) with \(\widetilde \nu_b\to0\) as \(b\downarrow0\)  such that for any small enough constant $b>0$ (depending on $\alpha$), as $\varepsilon \rightarrow 0,$
\begin{align}  \label{dec}
        N_{\varepsilon,b}^{-\frac{\alpha^2}{2(1-\eta)}-\widetilde  \nu_b+o(1)}
        \le
        \mathbb P\left(G_{\eta;1}
        \ge  N_{\varepsilon,b}^{-\frac{1-\eta}{2}+\alpha} 
        \right)
        \le
        N_{\varepsilon,b}^{-\frac{\alpha^2}{2(1-\eta)}+ \widetilde  \nu_b+o(1)}.
\end{align}
\end{proposition}

Next, formalizing \eqref{cond8787}, the following proposition is our main decoupling statement.

\begin{proposition}[Conditional decoupling of partition function]
\label{prop:terminal-block-extension}
 There exists $C>0$ such that the following holds:
Let \(\alpha>0\) and \(\eta \in (0,1/10)\) be any constants. Then for any small enough constant \(b>0\) (depending only on $\alpha$ and $\eta$),
\begin{equation}
\mathbb E\left[
\mathbf 1_{\Omega_{\varepsilon,\eta}}
\left|
\frac{\mathcal Z_{\varepsilon,b}}{G_{\eta;1}}
-
G_{\eta;2}
\right|^2
\,\middle|\,
\mathrm{UT}_{\varepsilon,\eta}(\alpha)
\right] \le 
   C 
        (\log\varepsilon^{-1})^{C
 \eta } 
\label{eq:terminal-extension-main}
\end{equation}
for all 
sufficiently small \(\varepsilon>0\).
\end{proposition}
The power $C\eta$ in the bound
\((\log\varepsilon^{-1})^{C\eta}\) essentially arises from the second moment of $G_{\eta;2}$ but will be harmless
because \(\eta\) will be chosen to be small.\\\\

Finally, we state our mass concentration statement formalizing \eqref{eq:Z-mass-concentration-final-Geta 00} which is the key consequence of the two above results used in the proof of Theorem \ref{main1}.
{Essentially repeating the discussion in Remark \ref{raredisc2345}, recall that the above results suggest that $\cZ_{\e,b}$ behaves like  the log-Normal variable $\exp(X)$ where $$X\sim \mathcal N\Big(-\frac{\log N_{\e,b}}{2}, {\log N_{\e,b}}\Big).$$  The log-Normal variable has mean $1$ but for small $\e$, is typically of size $\frac{1}{\sqrt{N_{\e,b}}}$ and hence very close to zero. Thus, the mean is driven by its tail behavior. Namely, it takes value $\approx \sqrt{N_{\e,b}}$ with probability approximately $1/\sqrt{N_{\e,b}}$ and it is this level set and nearby values that contribute to its mean. 
The key next result makes this precise.}

\begin{proposition}
\label{prop:Z-mass-concentration-from-Geta1-Geta2} 
Let \(\zeta\in(0,1/10)\) be an arbitrary constant. 
Then there exists $\upsilon_\zeta>0$ such that the following holds 
 for  any small  constant $b>0$: For all  sufficiently small $\varepsilon>0,$
\begin{equation}\label{mass45}
        \mathbb E\left[
        \mathcal Z_{\varepsilon,b}\mathbf 1_{\{\mathcal Z_{\varepsilon,b}<N_{\varepsilon,b}^{1/2-\zeta}\}}
        \right]
        +
        \mathbb E\left[
        \mathcal Z_{\varepsilon,b}\mathbf 1_{\{\mathcal Z_{\varepsilon,b}>N_{\varepsilon,b}^{1/2+\zeta}\}}
        \right] \le (\log \varepsilon^{-1})^{-\upsilon_\zeta}. 
\end{equation} 
\end{proposition}

\vspace{.1in}

We end this section with the following remark.

\begin{remark}\label{largedevz} One may wonder whether one can prove a large-deviation statement as Proposition \ref{largedeviation45} with $G_{\eta,1}$ replaced by $\cZ_{\e,b}.$ 
Towards this, note that the following is an immediate consequence of \eqref{mass45} along with the fact that $\E \cZ_{\e,b}
=1.$ 
\begin{align}  \label{dec23}
        \mathbb P\left(\cZ_{\e,b}
= N_{\varepsilon,b}^{\frac{1}{2}+o(1)} 
        \right)
=
        N_{\varepsilon,b}^{-\frac{1}{2}+o(1)}.
\end{align} 
The upper bound follows from Markov's inequality and the lower bound follows by observing that \eqref{mass45} implies that 
$$\E \left[\cZ_{\e,b}  \cdot \mathbf{1}\{ \cZ_{\e,b}
=  N_{\varepsilon,b}^{\frac{1}{2}+o(1)}\} 
        \right]=1-o(1).$$
Thus the above is the exact $\cZ_{\e,b}$ counterpart of \eqref{dec} for $\alpha=1.$ For a general value of $\alpha>0$, however, the estimates recorded in the paper do not suffice. While this will be taken up in a future project, let us nonetheless indicate that for the upper bound, at least for integer values of $\alpha$, say $k$, one can already apply Markov's inequality with $\cZ_{\e,b}^k$ to obtain that 
$$\P\left( \cZ_{\e,b} \ge N_{\varepsilon,b}^{k-\frac{1}{2}+o(1)}\right)\le N_{\e,b}^{-\frac{k^{2}}{2}+o(1)}.$$
This is expected to be sharp because for a log-Normal variable $\exp(X)$ with $X\sim  \mathcal N(-\frac{\log N_{\e,b}}{2}, {\log N_{\e,b}}),$ for any positive integer $k,$ $$\E (\exp(kX))=\exp\left({{k}\choose{2}} {\log N_{\e,b}}\right)$$ and a direct computation reveals that this is essentially driven by the event that $X= (k-1/2){\log N_{\e,b}}$ which occurs with probability $N_{\e,b}^{-\frac{k^{2}}{2}}.$ For general values of $\alpha>0$, just relying on integer moments will not yield sharp upper bounds, and one approach to address this is by proving a version of \eqref{eq:terminal-extension-main} where the $L^2$ norm on the LHS is replaced by an arbitrarily large ($\alpha$-dependent) power.

{The lower bound may be proven using the following two step strategy.  Define $$\mathcal Z_{\eta;1}:= 
    p(\varepsilon^2)\blacktriangleleft Z_1\bullet Z_2\bullet\cdots\bullet Z_{M_{\e,\eta}}\blacktriangleright 1  $$ analogous to $G_{\eta;1}.$\\
 
\nin
$(1)$
Show that $\cZ_{\e,b}$ is comparable to $\cZ_{\eta;1}.$ That is, the terminal part of the noise in the interval $[t_{M_{\e, n}},1]$ does not cause things to become too small. A similar idea appears in \cite[Section 5]{loglog}.\\

\nin
$(2)$ Show that the initial part admits an even stronger comparison with }$G_{\eta;1}$. That is,
 \begin{equation}
\mathbb E\left[
\mathbf 1_{\Omega_{\varepsilon,\eta}}
\left|
\frac{\mathcal Z_{\eta;1}}{G_{\eta;1}}
-
1
\right|^2
\,\middle|\,
\mathrm{UT}_{\varepsilon,\eta}(\alpha)
\right] =o(1).
\label{eq:terminal-extension-main99990}
\end{equation}
Note the distinction of the above from \eqref{eq:terminal-extension-main}.
While the arguments in the paper have mostly made use of the upper bound of the probabilities in \eqref{dec}, the above two steps, once implemented, will allow to transfer the lower bound of the probabilities in \eqref{dec} to $\cZ_{\e,b}$. \end{remark}

\subsection{Organization of the article} To help orient the reader to the rest of the paper we provide a roadmap summarizing each of the subsequent sections. \\

\nin
$\bullet$ In Section \ref{error7895}, we formally define the error term $D_i$ and record some moment bounds. We also present statements of concentration of the $G_i$s around $1$ as well as some norm bounds.\\
\nin
$\bullet$ In Section \ref{expansion7895}, we record the precise expansion of the partition function $\cZ_{\e,b}$ around the decoupled proxy $G_{\eta;1}$.
\\
\nin
$\bullet$ In Section \ref{sec3}, we prove Proposition \ref{prop:terminal-block-extension} conditional on two statements which together control the error terms in the expansion. They are proven subsequently in Sections \ref{sec4567} and \ref{sec6789}.\\
\nin
$\bullet$ In Section \ref{masconc4567}, we prove the mass concentration statement Proposition \ref{prop:Z-mass-concentration-from-Geta1-Geta2} and derive Theorem \ref{main1} in Section \ref{min567891}.\\
\nin
$\bullet$ Section \ref{sec5} is one of the most central sections in this paper, deriving Proposition \ref{largedeviation45} and, in the process, proving a general large-deviations result for the decoupled partition function $G_{\eta;1}.$\\
\nin
$\bullet$ In Section \ref{rnproof34}, developing another technical foundation of our arguments, the large-deviations machinery is employed to control the moments of the conditional density (conditioned on the event $\mathrm{UT}_{\varepsilon,\eta}(\alpha)$) when projected onto a small number of coordinates. \\
\nin
$\bullet$ Finally, in the Appendix, proofs of various outstanding statements are provided, including some measure-theoretic results, as well as some pertaining to convex analysis which are invoked in the large-deviations analysis.

~

{Before proceeding with the rest of the paper let us briefly comment on some notational conventions that will be adopted. A generic constant such as  \(C\) will be used to denote a positive constant whose value may and often will 
change from line to line.} 
{In addition, for two positive quantities \(x\) and \(y\), we write $x \asymp y$ 
if there exists a constant \(C\ge 1\) such that $C^{-1}y \le x \le Cy.$ 
More generally, we write  \(x \asymp_\alpha y \)  if the same comparison holds with a constant \(C=C(\alpha)\ge 1\) depending on \(\alpha\).}
  
\section{Error terms and  moments}\label{error7895}
Recall from \eqref{mainobject} the error variables  which quantitatively measure the difference between the original partition function and the 
decoupled one $\prod_i G_i$. In this section we first define them formally and record some moment bounds. 
Recall from \eqref{keyblock} that $Z_i:=\mathscr{Z}^\vartheta_{t_{i-1},t_i}$ and $p_i:=p(t_i,\cdot)$.
For \(1\le i\le N_{\varepsilon,b}\),   define 
\begin{align}\label{error}
        D_i:=Z_i-(Z_i\blacktriangleright 1)p_i .
\end{align}
In other words, \(D_i\) is the 
 error in approximating the point-to-point partition function   by the point-to-line partition
function multiplied by the heat kernel term $p(t_i,\cdot) $. This measures, in some sense, the singularity of the measure $Z_i$ (relative to the heat kernel). {Note that one might find  the above choice of multiplying $(Z_i\blacktriangleright 1)$ by $p_i$ somewhat unnatural since all of this is defined on the time interval \([t_{i-1},t_i]\)  and has time length
\(t_i-t_{i-1}\), suggesting $p({t_i-t_{i-1}})$ (see \eqref{eq:heat-kernel-density-definition}) as the right multiplier and not \(p_i\). However, for our purposes this discrepancy will be negligible because the sequence \((t_i)_i\) is chosen on an exponentially separated time scale with the growth scale being $b^{-2}$ which will also approach $\infty$.}

The above is a single block version of the following more general interval-level error term defined as follows:
For  an interval $I = [i_I,j_I] $ with \(|I|  = j_I - i_I+1 \ge2\), 
\begin{align}\label{error23}
        D_I
        :=
        p_{i_I-1}
        \blacktriangleleft
        D_{i_I}\bullet\cdots\bullet D_{j_I-1}\bullet Z_{j_I}
        \blacktriangleright 1.
\end{align}   
As we will shortly formalize in the forthcoming section, in the expansion comparing $\cZ_{\e,b}$ and the decoupled product $G_{\eta;1}$, this will be a generic term. 

But as a first order of business we record some moment bounds which will serve as key inputs throughout the paper. 

\subsection{$L^2$ bounds on $D_I$} 
 By
\cite[Corollary 3.5]{loglog}, there exists a constant \(\mathsf c<\infty\),
independent of \(b\in(0,1/2]\), such that for every interval
\(I=[i_I,j_I]\subset {[1,N_{\varepsilon,b}-1]}\) with $|I| =   j_I - i_I+1  \ge 2,$
\begin{equation}
        \mathbb E D_I^2
        \le
        \frac{\mathsf c^{|I|}}
        {(\log b^{-1})^{|I|-1}}
\frac{1}{(N_{\varepsilon,b}-j_I)^2}.
\label{eq:GT-cor35-current}
\end{equation}
Equivalently,
\begin{equation}
        \|D_I\|_{L^2}
        \le
        \frac{\mathsf c^{|I|/2}}
        {(\log b^{-1})^{(|I|-1)/2}}
\frac{1}{N_{\varepsilon,b}-j_I}.
\label{eq:DI-L2-bound-current}
\end{equation} 
This estimate separates two effects: a decay depending on the length of the
interval \(I\), through \((\log b^{-1})^{-(|I|-1)/2}\), and a decay depending
on the location of its right endpoint, through \((N_{\varepsilon,b}-j_I)^{-1}\).
Both effects will feature centrally in our estimates.

\begin{remark}\label{elaborate1450} 
Given the form of the above estimate, we now briefly explain the choice of working with $N_{\varepsilon,b}-1$ instead of $N_{\varepsilon,b}$ in \eqref{auxpar}. This is simply to work with the clean expression in the bound \eqref{eq:DI-L2-bound-current}. Note that the expression becomes $\infty$ if $j_I=N_{\varepsilon,b}$ and hence the bound is effective only up to $N_{\e,b}-1$. Although one can easily remedy this by suitably modifying the expression to accommodate the case $j_I=N_{\varepsilon,b}$, at the cost of possibly losing its clean form, working with \eqref{eq:DI-L2-bound-current} will be notationally convenient in several algebraic expressions appearing in our arguments. 
\end{remark}

\subsection{Concentration of the single-scale factors around one.} 
 For \(1\le i\le N_{\varepsilon,b}\),  define
\begin{align}\label{sing34}
    W_i:= G_i-1.
\end{align}
Then $ W_i$ is a centered random variable, since $\mathbb E G_i=1$ (see \eqref{mean1}).  
The following  second moment statement for $W_i$ is quoted from  \cite[(2.15)]{loglog}
\begin{align}\label{2moment}
    \lim_{b\to 0} \sup_{\e\leq 1/2,i \in[1,N_{\varepsilon,b}-1]}
		\big| (N_{\varepsilon,b}-i) \E   W_i ^2 - 1 \big|
	= 0.
\end{align}
{The preceding estimate is formulated in the regime where \(b>0\) is held fixed
as \(\varepsilon\rightarrow0\), and the limit \(b\rightarrow0\) is taken 
afterward. We remark that a suitably modified version of the preceding estimate remains
valid  for \(\varepsilon\)-dependent parameters
\(b=b_\varepsilon\) and \(\vartheta=\vartheta_\varepsilon\), provided that  {$b_\varepsilon \rightarrow b^\star >0$ and $\vartheta_\varepsilon\rightarrow  \vartheta^\star$} as \(\varepsilon\rightarrow0\) (see Lemma \ref{vary} in Appendix for details).} In addition, by
\cite[(B.22) in Corollary~B.3]{loglog}, there exists $c_*>0$ such that for {every $b\le 1/2$, $\varepsilon>0$ and an integer \(p\ge2\), there exists a constant $C_p>0$ such that}
\begin{equation}
        \mathbb E  | W_i |^p 
        \le  [\mathbb E   W_i ^{2p}]^{1/2}
\le       {C_p} (N_{\varepsilon,b}-i)^{- p/2 },
        \qquad \forall i\in [1,N_{\varepsilon,b} - c_*].
\label{eq:centered-G-moment-B3}
\end{equation} 
{The argument used to prove \cite[(B.22) in Corollary~B.3]{loglog} also shows
that the estimate remains valid when the coupling parameter
\(\vartheta=\vartheta_\varepsilon\) depends on \(\varepsilon\), provided that
\(\sup_{\varepsilon}\vartheta_\varepsilon<\infty\).  Indeed, the proof of
\cite[(B.22)]{loglog} relies only on the general estimate
\cite[(B.21)]{loglog}, whose hypotheses \(L\ge c\) and \(r\le 1/2\) continue
to hold  under this
bounded-above assumption.} 

In particular, the estimate \eqref{eq:centered-G-moment-B3} implies a super-polynomial, in $(\log \e)^{-1}$, failure probability bound for the event $\Omega_{\varepsilon,\eta}$ defined in \eqref{nice}. To obtain the super-polynomial bound which will indeed be important for our arguments, it was crucial that we chose our cutoff threshold in \eqref{ellm} to be a small power of $\log \e^{-1}$.

\begin{lemma}\label{goodevent}
 Let $\eta \in (0,1/10)$  and $b\in (0,1/2]$ be any constants.   For any constant $A>10$, the following bound holds for all sufficiently small $\varepsilon>0$:
\begin{align}\label{prob1234}
    \mathbb P( \Omega_{\varepsilon,\eta}^c) \le   (\log \varepsilon^{-1})^{-A}.
\end{align}
\end{lemma}
\begin{proof}
Note that for every $1\le i\le M_{\varepsilon,\eta} $, we have   $N_{\varepsilon,b}-i \ge \ell_{\varepsilon,\eta}$ (see \eqref{ellm} for the definition of $M_{\varepsilon,\eta}$). Hence
     by the moment bound \eqref{eq:centered-G-moment-B3},  for every $1\le i\le M_{\varepsilon,\eta} $ and any integer $p \ge 10A/\eta,$
\begin{align*}
    \mathbb P\left(| W_i| \ge \frac{1}{2}\right) \le 2^p \mathbb E| W_i|^p \le C 2^p \ell_{\varepsilon,\eta}^{-p/2} \le C  (\log \varepsilon^{-1})^{-5A},
\end{align*}
where we used $ \ell_{\varepsilon,\eta} = (\log \varepsilon^{-1})^\eta$.
Thus by a union bound, recalling $N_{\varepsilon,b}
        =
    \lfloor
{\log\varepsilon^{-1}}/{\log b^{-1}}
   \rfloor $, for all sufficiently small $\varepsilon>0$,
\begin{align*}
    \mathbb P( \Omega_{\varepsilon,\eta}^c) \le N_{\varepsilon,b} \cdot C (\log \varepsilon^{-1})^{-5A} \le  (\log \varepsilon^{-1})^{-A}.
\end{align*}
\end{proof}
In addition, as a consequence of the moment bound for $W_i$, we obtain a  moment estimate for a 
decoupled partition function $\prod_i G_i$.  

\begin{lemma}\label{gmoment}
For every integer \(p\ge2\), there exist 
\(C,b_0>0\) (depending on $p$) such that for any  constant $b\in (0,b_0)$ and 
\(A\subset[1,N_{\varepsilon,b}]\),
\begin{equation}
        \left\|
        \prod_{i\in A}G_i
        \right\|_{L^p}
        \le
 \left(
    \log b^{-1}
        \right)^{p^2c_*}
        \exp\left\{
        C
        \sum_{i\in A\cap[1,N_{\varepsilon,b}-c_*]}
        \frac1{N_{\varepsilon,b}-i}
        \right\},
\label{eq:single-G-product-moment-full-terminal-corrected}
\end{equation}
where $c_*>0$ is a constant from \eqref{eq:centered-G-moment-B3}.
\end{lemma}

\begin{proof} 
We first bound the $L^p$ norm of a single $G_i$.
We write
\[
        \mathbb E G_i^p
        =
        \mathbb E(1+(G_i-1))^p
        =
        1+p\mathbb E (G_i-1)
        +
        \sum_{k=2}^{p}\binom pk \mathbb E (G_i-1)^k    =
        1 
        +
        \sum_{k=2}^{p}\binom pk \mathbb E (G_i-1)^k.
\]
Hence using \eqref{eq:centered-G-moment-B3}  and recalling $G_i-1=W_i,$
\begin{align} \label{sing}
        \mathbb E G_i^p
    \le
        1+
        \sum_{k=2}^{p}\binom pk
        \frac{C_k}{N_{\varepsilon,b}-i}               \le
        1+\frac{C}{N_{\varepsilon,b}-i},  \qquad \forall i\in [1,N_{\varepsilon,b} - c_*].
\end{align} 
Above, we take $C=\sum_{k=2}^{p}\binom pk
        C_k.$
Given this we now prove \eqref{eq:single-G-product-moment-full-terminal-corrected}.
We split $   A=A_{\mathrm{reg}}\cup A_{\mathrm{end}},$ 
where
\[
        A_{\mathrm{reg}}
        :=
        A\cap[1,N_{\varepsilon,b}-c_*],
        \qquad
        A_{\mathrm{end}}
        :=
        A\cap[N_{\varepsilon,b}-c_*+1,N_{\varepsilon,b}].
\]
By independence,
\begin{equation}
        \left\|
        \prod_{i\in A}G_i
        \right\|_{L^p}
        =
        \left\|
        \prod_{i\in A_{\mathrm{reg}}}G_i
        \right\|_{L^p}
        \left\|
        \prod_{i\in A_{\mathrm{end}}}G_i
        \right\|_{L^p}.
\label{eq:product-split-reg-end}
\end{equation}
We first estimate the regular part.  By independence,
\begin{align}
        \left\|
        \prod_{i\in A_{\mathrm{reg}}}G_i
        \right\|_{L^p}
     \overset{\eqref{sing}}{\le}
        \prod_{i\in A_{\mathrm{reg}}}
        \left(
        1+\frac{C}{N_{\varepsilon,b}-i}
        \right)^{1/p} \le
        \exp\left\{
        C
        \sum_{i\in A_{\mathrm{reg}}}
        \frac1{N_{\varepsilon,b}-i}
        \right\} .
\label{eq:regular-product-G-bound}
\end{align}
It remains to bound the last \(c_*\) terminal factors.  We  will show later (see \eqref{singleend} below) that for
any small enough  constant  \(b>0\),  
{\[
        \mathbb E G_i^p
        \le  
        \left(  
 \log b^{-1}
        \right)^{{p\choose2}+1} \le  \left(
   \log b^{-1}
        \right)^{p^2} , \qquad \forall i\in [1,N_{\varepsilon,b} ].
\]} 
{We remark that this estimate is essentially the shrinking-ball moment asymptotics in \cite{shrinking} (see (3) in Remark \ref{remark74} below for details).}
Since \(|A_{\mathrm{end}}|\le c_*\), this along with \eqref{eq:regular-product-G-bound} (recall $   A_{\mathrm{reg}}
    =
        A\cap[1,N_{\varepsilon,b}-c_*]$) conclude the proof.
\end{proof}

As consequences of Lemma~\ref{gmoment}, we record the following two estimates. Let $p \ge 2$ be any integer.
First, for any \(A\subset[M_{\varepsilon,\eta}+1,N_{\varepsilon,b}]\) and a  small enough constant $b>0$, for sufficiently small $\varepsilon>0,$ {there exists  $C>0$ (depending on $b$) such that}
\begin{align}
        \left\|
        \prod_{i\in A}G_i
        \right\|_{L^p}
        &\le
        \left(
          \log b^{-1}
        \right)^{p^2c_*}
        \exp\left\{
        C
        \sum_{i=M_{\varepsilon,\eta}+1}^{N_{\varepsilon,b}-c_*}
        \frac1{N_{\varepsilon,b}-i}
        \right\}  \le 
        (\log\varepsilon^{-1})^{C\eta}.
\label{eq:G-tail-product-moment-consequence}
\end{align}
{The last inequality holds} since $  M_{\varepsilon,\eta}
        =
        N_{\varepsilon,b}-\ell_{\varepsilon,\eta}$ and $\ell_{\varepsilon,\eta} = (\log \varepsilon^{-1})^\eta,$  and thus
\[
        \sum_{i=M_{\varepsilon,\eta}+1}^{N_{\varepsilon,b}-c_*}
        \frac1{N_{\varepsilon,b}-i}
        =
        \sum_{r=c_*}^{\ell_{\varepsilon,\eta}-1}
        \frac1r
        \le
        C\log\ell_{\varepsilon,\eta}
=
        C\eta\log\log\varepsilon^{-1}.
\]
Note that the factor
\(\left(  \log b^{-1}\right)^{p^2c_*}\) has been absorbed into the term $(\log\varepsilon^{-1})^{C\eta}$ (by increasing the value of the constant $C$). This is permitted and harmless because \(b>0\) is a fixed constant before taking the limit
\(\varepsilon \rightarrow 0\).

Similarly, {for a small enough constant $b>0$, for sufficiently small $\varepsilon>0$ and any \(A \subset [1,N_{\varepsilon,b}]\),}
\begin{align}
        \left\|
        \prod_{i\in A}G_i
        \right\|_{L^p}
        &\le
        \left(
          \log b^{-1}
        \right)^{p^2c_*}
        \exp\left\{
        C
        \sum_{i=1}^{N_{\varepsilon,b}-c_*}
        \frac1{N_{\varepsilon,b}-i}
        \right\}  \le 
        N_{\varepsilon,b}^{C}  \le 
        (\log\varepsilon^{-1})^{C},
\label{eq:G-full-product-moment-consequence}
\end{align} 
{where  in the last inequality we used $N_{\varepsilon,b} \le \log \varepsilon^{-1}$ for $b\in (0,1)$ (see \eqref{ndef}).}

We conclude this subsection by recording the following moment estimate from
\cite[Corollary~B.3]{loglog}. Define the centered SHF:
\[
        \mathscr W_{t,t'}^\vartheta
        :=
        \mathscr Z_{t,t'}^\vartheta-P_{t,t'}.
\]
Then, for every \(h\in\mathbb N\), there exist constants $C$ {(depending on $h$)} and a universal
{\(c_{**}>0\)}   such that, for all \(L\ge c_{**}\) and \(r\in(0,1/2]\),
\begin{align} \label{Lbound}
    \mathbb E 
        \left\vert
         p(r^2)\blacktriangleleft \mathscr W^{-L  \log r^{-1}}_{0,1}\blacktriangleright 1 
        \right\vert^h
  \le \left[\mathbb E
        \left(
         p(r^2)\blacktriangleleft \mathscr W^{-L  \log r^{-1}}_{0,1}\blacktriangleright 1 
        \right)^{2h}
        \right]^{1/2} 
        \le
        C L ^{- h/2 },
\end{align}
{where the first inequality follows from H\"older's inequality, while the second
follows from \cite[(B.21) in Corollary~B.3]{loglog}.}
In words, the above is a quantification of the statement that if $\vartheta$ is negative enough, then SHF approximates the heat kernel ({with an exact equality when $\vartheta=-\infty$}).

\section{Expansion of the partition function around the decoupled proxy}\label{expansion7895}
The
full partition function $\cZ_{\e,b}$ admits an expansion containing the
decoupled object $G_{\eta,1}$, the product of the single-scale partition functions \(G_i\), as one of the terms while the other terms represent interactions and is made of terms involving \(D_I\) defined above in \eqref{error} and \eqref{error23}.  

\newcommand{\nt}{\neg \cT}
Recalling the definitions around \eqref{defii}, for $\wt I \subset \mathfrak I_{\e,b}$ define the tail complement of \(\widetilde I\) by
\begin{align}\label{ti}
{\neg \cT(\widetilde I)}
        :=
        [M_{\varepsilon,\eta}+1,N_{\varepsilon,b}-1]\setminus   \widetilde I,
\end{align}
i.e., \(\nt(\widetilde I)\) is the set of tail  indices which are
\emph{not} covered by  \(\widetilde I\). Note that here we are slightly abusing notation to identify $\wt I$ with a collection of disjoint intervals as well as a subset of integers obtained by taking the union of the same intervals. We set \begin{equation}
         \widehat G_{\widetilde I}^{\rm tail}
        :=
        G_{\nt(\widetilde I)}
        =
        \prod_{i\in \nt(\widetilde I)}G_i ,
\label{eq:terminal-tail-product-definition}
\end{equation}
with the convention that the empty product is equal to \(1\). 
In other words,
\( \widehat G_{\widetilde I}^{\rm tail}\) is the product of the single-scale
factors \(G_i\) in the tail region \([M_{\varepsilon,\eta}+1,N_{\varepsilon,b}-1]\), \emph{excluding} those indices
covered by  \(\widetilde I\) (the $\widehat G$ notation is devised to indicate that it constitutes of the factors \emph{not} covered by any $\widetilde I$).

A similar expansion was proven in \cite{loglog} and the form below is a slight variation. 
\begin{lemma}\label{expansion234}The following exact identity holds. 
\begin{equation}
        \frac{\mathcal Z_{\varepsilon,b}}{G_{\eta;1}}
        =
        G_{\eta;2}
        +
      \sum_{\widetilde I\in\mathfrak I_{\varepsilon,b}}
        \mathcal E_{\eta;\widetilde I} 
\label{eq:terminal-expansion-identity}
\end{equation}
(see \eqref{defii} for the definition of $  \mathfrak I_{\varepsilon,b}$), 
where  each error term \(\mathcal E_{\eta;\widetilde I}\) has the following structure:
\begin{equation}
        \mathcal E_{\eta;\widetilde I}
        =
        \left(
        \prod_{I\in\widetilde I}
        \frac{D_I}{G_I^{\rm bulk}}
        \right)
                \widehat G_{\widetilde I}^{\rm tail}.
\label{eq:terminal-error-term-structure}
\end{equation}   

\end{lemma}

 Thus the entire decoupling error is represented as a (weighted) sum of product of 
 $D_I$.  The size of such a collection is measured by its total
length $ |\widetilde I|:=\sum_{I\in\widetilde I}|I|.$
It controls
both the amount of decay of $D_I$ (see \eqref{eq:DI-L2-bound-current}) and the combinatorial
cost of summing over possible collections. 
\begin{proof}

The decoupling expansion is obtained by applying the same iterative decomposition as in
Proposition~2.1 in \cite{loglog}, with one important modification: we now normalize by the bulk
product \(G_{\eta;1}\), rather than by the full product \(\prod_i G_i\). 
More precisely,  we repeatedly decompose each block as $Z_i=(Z_i\blacktriangleright1)p_i+D_i$ and expand  the full partition function
\[
     p(\varepsilon^2)\blacktriangleleft Z_1\bullet\cdots\bullet Z_{N_{\varepsilon,b}-1} \blacktriangleright1,
\] 
and  then divide by \(G_{\eta;1}\). In the  expansion above,
each index \(i\) is treated in one of two ways: either 
\((Z_i\blacktriangleright1)p_i\) is used, or the error term \(D_i\) is used. Here, $(Z_i\blacktriangleright1)p_i$ cuts the product: the part before the cut is
closed by integrating \(Z_k\) against \(1\), while the part after the cut is
restarted from the  heat kernel \(p_k\).   Choosing
the term \(D_i\), on the other hand, does not create such a cut and the
product continues to the next scale. Thus the
associated interval-level contribution is
\begin{equation} 
        D_I
        =
        p_{i_I-1}
        \blacktriangleleft
        D_{i_I}\bullet D_{i_I+1}\bullet\cdots\bullet D_{j_I-1}
        \bullet Z_{j_I}
        \blacktriangleright1 .
\end{equation}
In particular, the last index \(j_I\) is a closing index corresponding to a $Z$ factor, not a \(D\)-one.
The leading term corresponds to the case where no error term is chosen.  It is $\prod_{i=1}^{N_{\varepsilon,b}-1}G_i
        =
        G_{\eta;1}G_{\eta;2},$ 
and hence becomes \(G_{\eta;2}\) after division by \(G_{\eta;1}\). For a nonempty collection \(\widetilde I\) of error intervals, the factors
\(G_i\) with bulk indices \(i\le M_{\varepsilon,\eta}\) and
\(i\notin\bigcup_{I\in\widetilde I}I\) are exactly cancelled by the
normalization \(G_{\eta;1}\). 
On the other hand, if \(i>M_{\varepsilon,\eta}\) and \(i\) is not covered by
any error interval, then the corresponding factor \(G_i\) is not cancelled and
remains in the terminal product $
               \widehat G_{\widetilde I}^{\rm tail}.$   
        
For completeness, let us include the following details. 
\begin{align*}p(\varepsilon^2)\blacktriangleleft Z_1\bullet\cdots\bullet Z_{N_{\varepsilon,b}-1} \blacktriangleright1&=p(\varepsilon^2)\blacktriangleleft \Big((Z_1\blacktriangleright1)p_1+D_1\Big)\bullet\cdots\bullet Z_{N_{\varepsilon,b}-1} \blacktriangleright1,\\
&=p(\varepsilon^2)\blacktriangleleft \Big((Z_1\blacktriangleright1)p_1\Big)\bullet\cdots\bullet Z_{N_{\varepsilon,b}-1} \blacktriangleright1\,\,\,+ \\&\quad \quad p(\varepsilon^2)\blacktriangleleft D_1\bullet\cdots\bullet Z_{N_{\varepsilon,b}-1} \blacktriangleright1,\\
&=\,\,\ldots \\
&= \prod_{i=1}^{N_{\varepsilon,b}-1} G_i + \sum_{\wt I\in \mathfrak I_{\e,b}}\prod_{I\in \wt I}D_I \prod_{i\notin \wt I}G_i.
\end{align*}
Above we followed our convention of denoting $\wt I$ as both a  collection of intervals as well as a set of integers obtained by taking the union of the intervals. Dividing by $G_{\eta;1}$ completes the proof.
\end{proof}

\section{Conditional negligibility of the decoupling error}   \label{sec3}

To prove Proposition \ref{prop:terminal-block-extension},  we must control the second term, i.e., the sum on the RHS in \eqref{eq:terminal-expansion-identity}.      
Recalling the expression
\begin{equation}\label{error56} \sum_{\widetilde I\in\mathfrak I_{\varepsilon,b}}
        \mathcal E_{\eta;\widetilde I} 
        =\sum_{\widetilde I\in\mathfrak I_{\varepsilon,b}}
        \left(
        \prod_{I\in\widetilde I}
        \frac{D_I}{G_I^{\rm bulk}}
        \right)
                \widehat G_{\widetilde I}^{\rm tail}.
\end{equation}
Note that we must bound the $L^2$ norm of the above sum  conditional on $\mathrm{UT}_{\varepsilon,\eta}(\alpha)$. 
The input we have is an unconditional $L^2$ bound from \cite{loglog}. As indicated in Section \ref{iop}, a key new ingredient we develop is a strong moment bound on the Radon-Nikodym derivative of the conditional distribution of the $Z_i$s given $\mathrm{UT}_{\varepsilon,\eta}(\alpha)$ with respect to the unconditional distribution. Further, note that \eqref{dec} delivers a sharp probability bound, in particular a lower bound, of $\P(\mathrm{UT}_{\varepsilon,\eta}(\alpha)).$
As already alluded to in Section \ref{iop}, given the above, there will be two styles of argument employed.  For $\wt I$ with large-total-length, i.e $|\wt I |$ from \eqref{idef1}, the exponential decay of   the \(L^2\)-norm of 
the factors \(D_I\) is strong enough to dominate the combinatorial entropy as well as the effect of conditioning. Namely, dividing the unconditional estimate by the lower bound on $\P(\mathrm{UT}_{\varepsilon,\eta}(\alpha))$ suffices. 

For $\wt I$ with small-total-length, the naive argument is not strong enough and a significantly more complicated argument is implemented. On the one hand the entropy factor is not too large which comes to our aid. But the main crucial observation is that the large deviation conditioning does not alter the distribution significantly when projected on a small number of coordinates involved in $\wt I.$ The moral reason for this was outlined in Section \ref{iop}. Quantitatively, this is captured by obtaining centered moment estimates of the conditional density (recorded in Proposition \ref{thm:direct-Lp-no-split-corrected} below).

 Motivated by this,
we split the decoupling error into a small-total-length part and a large-total-length
part.  The cutoff
\[
        \delta\mathcal L_\varepsilon
        =
        \delta\log\log\varepsilon^{-1}
\]
($\delta>0$ is a small  constant)
is chosen to match the  entropy scale of the interval collections. How it aids our analysis will be clear shortly.   
To be precise, 
for a small constant $\delta>0$, define
\begin{align}\label{short}
\mathcal R_{\eta,\le\delta}
        := \sum_{\substack{\widetilde I\in\mathfrak I_{\varepsilon,b}\\
        2\le |\widetilde I|\le \delta \mathcal L_\varepsilon}} \mathcal E_{\eta;\widetilde I} = 
        \sum_{\substack{\widetilde I\in\mathfrak I_{\varepsilon,b}\\
        2\le |\widetilde I|\le \delta \mathcal L_\varepsilon}}
      \left(
        \prod_{I\in\widetilde I}
        \frac{D_I}{G_I^{\rm bulk}}
        \right)
                \widehat G_{\widetilde I}^{\rm tail},
\end{align}
and
\begin{align} \label{large}
\mathcal R_{\eta,>\delta}
        :=\sum_{\substack{\widetilde I\in\mathfrak I_{\varepsilon,b}\\
         |\widetilde I|>\delta \mathcal L_\varepsilon}} \mathcal E_{\eta;\widetilde I} = 
        \sum_{\substack{\widetilde I\in\mathfrak I_{\varepsilon,b}\\
        |\widetilde I|>\delta \mathcal L_\varepsilon}}
        \left(
        \prod_{I\in\widetilde I}
        \frac{D_I}{G_I^{\rm bulk}}
        \right)
                \widehat G_{\widetilde I}^{\rm tail}.
\end{align}
Clearly
\begin{equation}\label{decomp12}
\sum_{\widetilde I\in\mathfrak I_{\varepsilon,b}}
        \mathcal E_{\eta;\widetilde I} = \mathcal R_{\eta,\le \delta}+\mathcal R_{\eta,>\delta}
\end{equation}

We estimate these two pieces separately. Further, we will always work on the high probability event $\Omega_{\e,\eta}$ from \eqref{prob1234}. The key small length estimate is the following.

\begin{proposition}[Conditional second moment: small-total-length]
\label{prop:second-moment-small-total-length-full}

There exists $C>0$  such that  the following holds:
Let \(\alpha>0\) and  \(\eta ,\delta\in (0,1/10) \)   be any constants. Then  for any small enough constant \(b>0\) (depending only on  $\alpha$),
\begin{equation}
        \mathbb E\left[
        \mathbf 1_{\Omega_{\varepsilon,\eta}}
        \left|
        \mathcal R_{\eta,\le\delta}  
        \right|^2
        \,\middle|\,
        \mathrm{UT}_{\varepsilon,\eta}(\alpha)
        \right]
        \le
        C 
        (\log\varepsilon^{-1})^{
 C\eta+ C \delta\log(1/\delta)} 
\label{eq:full-small-second-moment-polylog-eta-delta}
\end{equation} 
for all 
sufficiently small \(\varepsilon>0\).
\end{proposition}

The simpler large-total-length estimate is the following. 
\begin{proposition}[Conditional second moment: large-total-length]
\label{prop:first-moment-large-total-length}
There exists $C>0$  such that  the following holds:
Let \(\alpha>0\) and  \(\eta ,\delta\in (0,1/10) \)   be any constants. Then for any \(\Lambda>0\), for any small enough constant \(b>0\) (depending only on  $\alpha$, $\delta$ and $\Lambda$),
\begin{equation}
        \mathbb E\left[
        \mathbf 1_{\Omega_{\varepsilon,\eta}}
        \left|
        \mathcal R_{\eta,>\delta}  
        \right|^2
        \,\middle|\,
        \mathrm{UT}_{\varepsilon,\eta}(\alpha)
        \right]
        \le C
(\log\varepsilon^{-1})^{C\eta-\Lambda} 
\end{equation} 
for all 
sufficiently small \(\varepsilon>0\).
\end{proposition}

As already  indicated, the proof is essentially an unconditional 
one followed by dividing by the probability of the upper-tail event.\\

Combining the previous two propositions immediately gives Proposition~\ref{prop:terminal-block-extension} as follows.
\begin{proof}[Proof of Proposition~\ref{prop:terminal-block-extension}] 
Let $\delta>0$ be a constant such that $\delta \log (1/\delta) < \eta.$
Since $  \frac{\mathcal Z_{\varepsilon,b}}{G_{\eta;1}}
-
        G_{\eta;2} =\mathcal R_{\eta,\le \delta} +\mathcal R_{\eta,>\delta} $, the left hand side of \eqref{eq:terminal-extension-main} is bounded by
\[ 
        2\mathbb E\left[\mathbf 1_{\Omega_{\varepsilon,\eta}}
        |\mathcal R_{\eta,\le\delta}|^2
        \,\middle|\,
        \mathrm{UT}_{\varepsilon,\eta}(\alpha)
        \right]
        +
       2 \mathbb E\left[\mathbf 1_{\Omega_{\varepsilon,\eta}}
        |\mathcal R_{\eta,>\delta}|^2
        \,\middle|\,
        \mathrm{UT}_{\varepsilon,\eta}(\alpha)
        \right].
\]  
Let $\Lambda>0$ be any  constant. The proof is now complete by the control on the first and second terms delivered by  Propositions
 \ref{prop:second-moment-small-total-length-full}  and
 \ref{prop:first-moment-large-total-length} respectively, provided that $b>0$ is  a   small enough constant.

\end{proof}

It remains to prove Proposition \ref{prop:second-moment-small-total-length-full} and Proposition \ref{prop:first-moment-large-total-length}; this will be accomplished in the two following sections.

\section{Conditional second moment: large-total-length}
\label{sec4567}
In this section we prove Proposition \ref{prop:first-moment-large-total-length}. This will rely on the following unconditional estimate.

\subsection{Unconditional second moment: large-total-length}

 In this section we prove that the contribution of collections \(\widetilde I\)
with large-total-length has an unconditional second moment which is smaller
than any prescribed power of \(\log\varepsilon^{-1}\).   The key idea is that, above the scale  $ |\widetilde I|>\delta\log\log\varepsilon^{-1}$, the decay from \eqref{eq:DI-L2-bound-current} dominates the
combinatorial cost of summing over all admissible collections, once 
\(b>0\) is sufficiently small.

\begin{proposition}[Unconditional second moment: large-total-length]

\label{Unconditional second moment: large-total-length}
There exists $C>0$  such that  the following holds:
Let  \(\eta ,\delta\in (0,1/10) \)   be any constants. Then for any \(\Lambda>0\), for any small enough constant \(b>0\) (depending only on  $\delta$ and $\Lambda$),
\begin{equation}
        \mathbb E\left[
        \mathbf 1_{\Omega_{\varepsilon,\eta}}
        \left|
\mathcal R_{\eta,>\delta}  
        \right|^2
        \right]
        \le
      C(\log\varepsilon^{-1})^{C\eta-\Lambda}
\label{eq:large-length-unconditional-second-moment-target}
\end{equation}
for all sufficiently small \(\varepsilon>0\) {(depending on $b,\eta,\delta$).}\end{proposition}

\begin{proof} 
Recalling that $\mathbf 1_{\Omega_{\varepsilon,\eta}}
\mathcal R_{\eta,>\delta} = \displaystyle{\sum_{\substack{\widetilde I\in\mathfrak I_{\varepsilon,b}\\
         |\widetilde I|>\delta \mathcal L_\varepsilon}}  \mathbf 1_{\Omega_{\varepsilon,\eta}} \mathcal E_{\eta;\widetilde I}},$
applying the $L^2$-norm  Minkowski's inequality (triangle inequality),
\begin{equation}
\begin{aligned}
&\left\|
 \mathbf 1_{\Omega_{\varepsilon,\eta}}
        \left|
\mathcal R_{\eta,>\delta} \right|
\right\|_{L^2}                             \le
    \sum_{\substack{\widetilde I\in\mathfrak I_{\varepsilon,b}\\
        |\widetilde I|>\delta \mathcal L_\varepsilon}}
        \left\|
        \mathcal E_{\eta;\widetilde I}
        \mathbf 1_{\Omega_{\varepsilon,\eta}}
        \right\|_{L^2}         . 
\end{aligned}
\label{eq:minkowski-reduction-long-L2}
\end{equation}
We now estimate each summand in \(L^2\).  Recall that on the event 
\({\Omega_{\varepsilon,\eta}}\), each $G_{i}$ for each bulk index $i,$ i.e., $i\in [1, M_{\e,\eta}]$ is at least a $1/2$ and hence we have {$G_I^{\rm bulk}\ge 2^{-|I^{\rm Bulk}|}\ge 2^{-|I|}$ (see \eqref{bulk123} for the definition of $I^{\rm Bulk}$).}  Thus, recalling \eqref{eq:terminal-error-term-structure}
\begin{equation}
\begin{aligned}
        \left|
        \mathcal E_{\eta;\widetilde I}
        \right|
        \mathbf 1_{\Omega_{\varepsilon,\eta}}
        &\le
        \left(
        \prod_{I\in\widetilde I}
        2^{|I|}|D_I|
        \right)
                \widehat G_{\widetilde I}^{\rm tail}           =
        2^{|\widetilde I|}
        \left(
        \prod_{I\in\widetilde I}|D_I|
        \right)
                \widehat G_{\widetilde I}^{\rm tail}.
\end{aligned}
\label{eq:second-moment-good-event-bound}
\end{equation}
The random variables $  \{D_I:I\in\widetilde I\}$ and $          \widehat G_{\widetilde I}^{\rm tail}$
are  mutually independent  due to  disjointness.  Hence, 
\begin{equation}
\begin{aligned}
        \left\|
        \mathcal E_{\eta;\widetilde I}
        \mathbf 1_{\Omega_{\varepsilon,\eta}}
        \right\|_{L^2}
        &\le
        2^{|\widetilde I|}
        \left\|
        \left(
        \prod_{I\in\widetilde I}|D_I|
        \right)
                \widehat G_{\widetilde I}^{\rm tail}
        \right\|_{L^2}    =
        2^{|\widetilde I|}
        \left(
        \prod_{I\in\widetilde I}\|D_I\|_{L^2}
        \right)
        \left\|        \widehat G_{\widetilde I}^{\rm tail}\right\|_{L^2} .
\end{aligned}
\label{eq:second-moment-factorization}
\end{equation}
We first bound the tail factor.  
By  \eqref{eq:G-tail-product-moment-consequence}   with \(p=2\) and
{\(A= \neg \cT(\widetilde I)  \)}, 
for all sufficiently small \(\varepsilon>0\),
\begin{align} \label{245}
        \left\|        \widehat G_{\widetilde I}^{\rm tail}\right\|_{L^2}
        \le   (\log\varepsilon^{-1})^{C\eta}.  
\end{align} 
We next estimate the first factor.
By \eqref{eq:DI-L2-bound-current}, 
\[
        2^{|\widetilde I|}
        \prod_{I\in\widetilde I} \|D_I\|_{L^2}
        \le
        \prod_{I\in\widetilde I}
        \left[
        \frac{
        2^{|I|}\mathsf c^{|I|/2}
        }{
        (\log b^{-1})^{(|I|-1)/2}
        } \cdot 
        \frac{1}{N_{\varepsilon,b}-j_I}
        \right].
\]
where we are denoting any $I\in \wt I$ as $[i_I, j_I]$.
Since every interval $I$ in \(\widetilde I\) has length at least \(2\),  using $(|I|-1)/2 \ge |I|/4,$ and that $b\ge 0$ is small enough so that $\log b^{-1}\ge 1,$ we have
\[
        \prod_{I\in\widetilde I}
        \frac{
        2^{|I|}\mathsf c^{|I|/2}
        }{
        (\log b^{-1})^{(|I|-1)/2}
        }
        \le
        \frac{(2\mathsf c^{1/2})^{|\widetilde I|}}{(\log b^{-1})^{|\widetilde I| /4}}
        =
        e^{-c_b |\widetilde I| },
\]
where we set
\begin{equation}
        c_b
        :=
        \frac14\log\log b^{-1}
        -
        \log\bigl(2\mathsf c^{1/2}\bigr).
\label{eq:cb-first-moment}
\end{equation}
Hence applying this and \eqref{245} to \eqref{eq:second-moment-factorization},  the RHS of  \eqref{eq:minkowski-reduction-long-L2}  is bounded by 
\begin{equation} 
(\log\varepsilon^{-1})^{C\eta}
\sum_{\substack{\widetilde I\in\mathfrak I_{\varepsilon,b}\\
|\widetilde I|>\delta \mathcal L_\varepsilon}}\left[
e^{-c_b|\widetilde I|}
\prod_{I\in\widetilde I}
\frac{1}{N_{\varepsilon,b}-j_I}\right].
\label{eq:first-moment-reduced-sum-alpha}
\end{equation}

Let
\[
        m:=|\widetilde I|,
        \qquad
        q:=\#\widetilde I 
\]
(see \eqref{idef1} and \eqref{idef2} for the definitions of $|\widetilde I|$ and  $\#\widetilde I $),
 and  write the (disjoint)  intervals in
\(\widetilde I\) as
\[
        I_1,\ldots,I_q,
        \qquad
        I_r=[i_r,j_r],
        \qquad
        j_1<\cdots<j_q,
\]
and then set \(u_r:=|I_r| \ge 2 \). Then $\sum_{r=1}^q u_r = m$, and the data
\begin{align}\label{11correspond}
        \bigl(m,q,(u_1,\ldots,u_q),(j_1,\ldots,j_q)\bigr)  
\end{align}
determine \(\widetilde I\) uniquely.  Indeed, once \(u_r\) and \(j_r\) are
given, the \(r\)-th interval is necessarily $I_r = [j_r-u_r+1,\ldots,j_r].$ 
Conversely, every \(\widetilde I\in\mathfrak I_{\varepsilon,b}\)
gives such a tuple, subject to the admissibility conditions that $u_r\ge 2$ for all $r,$ the above
intervals lie in \(\{1,\ldots,N_{\varepsilon,b}\}\), are mutually disjoint.
Thus there is a one-to-one correspondence between
\(\mathfrak I_{\varepsilon,b}\) and the corresponding set of
admissible tuples.  In the following upper bound, we drop the admissibility
constraints on the endpoints \(j_1,\ldots,j_q\), which only enlarges the sum
since all summands are nonnegative. In other words,
\begin{equation}
\begin{aligned}
\sum_{\substack{\widetilde I\in\mathfrak I_{\varepsilon,b}\\
|\widetilde I|>\delta \mathcal L_\varepsilon}}\left[
e^{-c_b|\widetilde I|}
\prod_{I\in\widetilde I}
\frac{1}{N_{\varepsilon,b}-j_I}\right]  \le
\sum_{m>\delta \mathcal L_\varepsilon}
\sum_{q=1}^{m}
\sum_{\substack{u_1,\ldots,u_q\ge2\\
u_1+\cdots+u_q=m}}
\sum_{\substack{1\le j_1<\cdots<j_q< N_{\varepsilon,b}}}
e^{-c_bm}
\prod_{r=1}^{q}\frac{1}{N_{\varepsilon,b}-j_r}.
\end{aligned}
\label{eq:first-moment-mq-profile-endpoint}
\end{equation}
For fixed \(q\), 
\[
\sum_{\substack{1\le j_1<\cdots<j_q< N_{\varepsilon,b}}}
\prod_{r=1}^{q}
\frac{1}{N_{\varepsilon,b}-j_r}
\le \frac{1}{q!}  \left(\sum_{j=1}^{N_{\varepsilon,b}-1}  \frac{1}{N_{\varepsilon,b}-j}\right)^q \le 
\frac{(1 + \log N_{\varepsilon,b})^q}{q!}.
\]
{Above, we used the trivial bound $\sum_{i=1}^{k-1}\frac{1}{i}\le \log(1+k)$ for integers $k \ge 3$.}
For fixed \(m\) and \(q\),
\begin{align} \label{combinatorics}
        \#\{(u_1,\ldots,u_q):u_r\ge2,\ u_1+\cdots+u_q=m\}
        =  \binom{ (m-2q) +(q-1)}{q-1}\le 2^{m-q+1}\le  2^m.
\end{align}
Thus the RHS of \eqref{eq:first-moment-mq-profile-endpoint} is bounded by
\begin{equation} 
\sum_{m>\delta \mathcal L_\varepsilon}
e^{-c_bm}2^m
\sum_{q=1}^{m}\frac{(1 + \log N_{\varepsilon,b})^q}{q!}.
\label{eq:first-moment-before-q-sum-alpha}
\end{equation}
The \(q\)-sum is bounded by \(e^{1 + \log N_{\varepsilon,b}}=e N_{\varepsilon,b} \le C \log\varepsilon^{-1}\) {for a small enough constant $b>0$}  (see the definition of  $N_{\varepsilon,b} $ in \eqref{ndef}).
Also recalling the definition of $c_b$ in  \eqref{eq:cb-first-moment}, for  small enough  $b>0$, we have  \(c_b>\log2\). Thus recalling $\mathcal L_\varepsilon = \log \log\varepsilon^{-1},$
\[
        \sum_{m>\delta \mathcal L_\varepsilon}
     e^{-c_bm}2^m
        \le
        \frac{(2e^{-c_b})^{\delta \mathcal L_\varepsilon}}
        {1-2e^{-c_b}}  \le 
10(\log\varepsilon^{-1})^{-\delta(c_b-\log2)}.
\] 
Hence  for sufficiently small constant $b>0$,  we deduce that  the quantity  
\eqref{eq:first-moment-reduced-sum-alpha} is at most  $(\log\varepsilon^{-1})^{C\eta -\Lambda/2}.$ 
Putting things together yields 
$\left\|
 \mathbf 1_{\Omega_{\varepsilon,\eta}}
        \left|
\mathcal R_{\eta,>\delta} \right|
\right\|_{L^2}\le {C}(\log\varepsilon^{-1})^{C\eta -\Lambda/2}$. 
Taking the square of the $L^2$-norm, we conclude the proof.
\end{proof}

The conditional bound Proposition \ref{prop:first-moment-large-total-length} is now an immediate consequence of Proposition \ref{Unconditional second moment: large-total-length}. \begin{proof}[Proof of Proposition \ref{prop:first-moment-large-total-length}]
For a given $\alpha>0$, let  $(\widetilde \nu_b)_{b>0}$ be such that Proposition \ref{largedeviation45}, i.e., \eqref{dec} holds.
Let $b>0$ be a  small enough constant  such that $\widetilde \nu_b < 1 $ and Proposition \ref{Unconditional second moment: large-total-length} holds.
 By  
Proposition \ref{Unconditional second moment: large-total-length}   with $\Lambda + \frac{\alpha^2}{2(1-\eta)} + 2$ in place of $\Lambda$, 
\[
\begin{aligned}
& \mathbb E\left[
        \mathbf 1_{\Omega_{\varepsilon,\eta}}
        \left|
        \mathcal R_{\eta,>\delta}  
        \right|^2
        \,\middle|\,
        \mathrm{UT}_{\varepsilon,\eta}(\alpha)
        \right]   \le
        \frac{
        (\log\varepsilon^{-1})^{C\eta -\Lambda - \frac{\alpha^2}{2(1-\eta)}-2}
        }{
        \mathbb P(\mathrm{UT}_{\varepsilon,\eta}(\alpha))
        }                                          \overset{\eqref{dec}}{\le} 
        (\log\varepsilon^{-1})^{C\eta - \Lambda},
\end{aligned}
\] 
yielding the desired conclusion. 
\end{proof}

\section{Conditional second moment: small-total-length
}\label{sec6789}
In this section we prove Proposition \ref{prop:second-moment-small-total-length-full}.
This section is long and technical and hence to guide the reader we start with a roadmap.\\

\nin
$\bullet$ In Section \ref{basic234} we record some basic ingredients such as moments of block partition functions such as $p_{r-1}\blacktriangleleft     \mathscr Z^\vartheta_{t_{r-1},t_s} \blacktriangleright 1$.\\

\nin
$\bullet$  In Section \ref{cancel654} we record a crude estimate to control products of $D_I$. It does not use cancellation and instead, each \(D_i\) is dominated by a sum of two positive kernels  $Z_i$ and $(Z_i\blacktriangleright1)p_i $,
and the resulting products are factorized into independent blocks.  This bound will be applied only to  control $D_I$ with short intervals $I$. \\

\nin
$\bullet$ In Section \ref{rn535} we state the central estimate on conditional densities (under large deviations) projected on a small number of coordinates. The statement whose proof is deferred to Section \ref{rnproof34} will be a key ingredient in the proof of Lemma \ref{prop:second-moment-small-total-length-full}. \\

\nin
$\bullet$ Given the above inputs, we finish the proof of Lemma \ref{prop:second-moment-small-total-length-full} in Section \ref{centralsec}.

~

\subsection{Basic ingredients}\label{basic234}

The first input records a monotonicity property of integer moments with respect to the SHF parameter. 
\begin{lemma}[Moment monotonicity]
\label{lem:moment-monotonicity-theta}
Let $h \ge 2$ be a positive integer and \(\phi\ge0\) be a nonnegative test
function. Then {for any fixed time  $t>0$}, the map 
\[
        \vartheta
        \longmapsto
        \mathbb E\left[
        \left(
        Z_t^\vartheta(\phi)
        \right)^h
        \right]
\]
is nondecreasing.  
Consequently, for every \(\rho\in(0,1)\),
\[
        \vartheta_1\le\vartheta_2
        \quad\Longrightarrow\quad
        \mathbb E\left[
        \left(
        p(\rho^2)\blacktriangleleft Z_{0,1}^{\vartheta_1}\blacktriangleright 1
        \right)^h
        \right]
        \le
        \mathbb E\left[
        \left(
        p(\rho^2)\blacktriangleleft Z_{0,1}^{\vartheta_2}\blacktriangleright 1
        \right)^h
        \right].
\]
\end{lemma}

Since this is a direct consequence of the  collision-diagram
representation for integer moments, we defer its proof to the appendix.

~

The next 
input is a shrinking-test-function moment bound at macroscopic time.  In
the lemma below, the parameter \(\vartheta\) is allowed to vary in a
bounded-above range which is a slight generalization of \cite{shrinking}, and will be convenient throughout the paper.

\begin{lemma}[Small ball moments]\label{shrinking}
Let $h \ge 2$ be a positive integer and  \(\kappa,c_0>0\).
  Then for every
\(\vartheta\le c_0\) and sufficiently small \(\rho>0\) (depending on $h,\kappa$, and $c_0$),
\begin{equation}
        \mathbb E\left[
        \left(
        p(\rho^2)\blacktriangleleft \mathscr{Z}^\vartheta_{0,1}\blacktriangleright 1
        \right)^h
        \right]
        \le
 (\log\rho^{-1})^{{h \choose 2}+\kappa}.
\label{eq:fixed-time-fourth-moment-input}
\end{equation}
\end{lemma}

\begin{proof}
    It is a direct consequence of  \cite{shrinking} and the monotonicity result in Lemma \ref{lem:moment-monotonicity-theta}.
\end{proof}

Finally, we need a uniform bound for blocks consisting of a bounded
number of consecutive scales.

\begin{lemma}[Block Moments]
\label{lem:bounded-consecutive-fourth}
 Let {$\vartheta\in \mathbb R$}, $h \ge 2$ and $\Gamma \ge 1$ be positive integers.  For an
interval 
\[
        Q=[r,s]\subset[1,N_{\varepsilon,b}],
        \qquad
        |Q|=s-r+1,
\]
set  
\[
        Z_Q
        :=
        p_{r-1}\blacktriangleleft  Z_r\bullet Z_{r+1}\bullet\cdots\bullet Z_s \blacktriangleright 1
        =
       p_{r-1}\blacktriangleleft     \mathscr Z^\vartheta_{t_{r-1},t_s} \blacktriangleright 1.
\]

\begin{enumerate}
\item[\rm{(i)}]
There exist constants \(C = C(h )>0\) and \(\tau = \tau(\Gamma ) >0\)  such that the following holds.  For every \(b\in(0,1/10)\), if
\[
        Q=[r,s]\subset[1,N_{\varepsilon,b}-\tau],
        \qquad
        |Q|\le \Gamma,
\]
then
\begin{equation}
        \mathbb E\left[
        Z_Q^h
        \right]
        \le
        C .
\label{eq:bounded-consecutive-nonterminal}
\end{equation}

\item[\rm{(ii)}]
For every \(\kappa>0\), there exists a constant
\(b_0=b_0(h,\kappa ) \in (0,1)\)  such that the following holds.  For every
\(b\in(0,b_0)\), if
\[
        Q=[r,s]\subset[1,N_{\varepsilon,b}],
        \qquad
        |Q|\le \Gamma,
\]
then
\begin{equation}
        \mathbb E\left[  
        Z_Q^h
        \right]
        \le
        \left(\Gamma\log b^{-1}
        \right)^{{h\choose2}+\kappa}.
\label{eq:bounded-consecutive-terminal-b-dependent}
\end{equation}
In particular, for every fixed \(b\in(0,b_0)\), the above moment is bounded
uniformly in \(\varepsilon\) and uniformly over all intervals
\(Q\subset[1,N_{\varepsilon,b}]\) with \(|Q|\le \Gamma\).
\end{enumerate}
\end{lemma}

\begin{remark} \label{remark74}
We mention some aspects of the above lemma.

\begin{enumerate}
    \item 
It is crucial that the Part \({\rm (i)}\) estimate is uniform in \(b\), but it requires the block \(Q\) to stay
a fixed number of scales (depending on $\Gamma)$ away from the terminal endpoint.   Part \({\rm (ii)}\) allows \(Q\) to touch the terminal
scale, and this produces the polynomial loss in \(\log b^{-1}\).

\item   Taking a single-element interval  $Q$ in  Part \({\rm (ii)}\) with $\Gamma=1$ and recalling $G_i = p_{i-1} \blacktriangleleft  Z_i \blacktriangleright 1$, we deduce that for every $\kappa>0$,  for any small enough constant $b >0$,
\begin{align} \label{singleend}
     \mathbb E\left[  
        G_i^h
        \right]
        \le
        \left( \log b^{-1}
        \right)^{{h\choose2}+\kappa},\qquad \forall i\in [1,N_{\varepsilon,b}].
\end{align}

\item {The estimate \eqref{eq:bounded-consecutive-terminal-b-dependent} is consistent with the shrinking-ball (of radius $\rho$) moment asymptotics in \cite{shrinking}, which derives \(h\)-th moments of order
\((\log \rho^{-1})^{\binom h2}\).  Indeed, by the scaling relation of the SHF,
$Z_Q$ corresponds to a unit-time SHF smoothed at an effective
radius \(\rho\asymp b^k\) where $|Q|=k$. Note that by hypothesis, $k\le  \Gamma.$}

\item  {The proof below shows that the  bounds \eqref{eq:bounded-consecutive-nonterminal}  and \eqref{eq:bounded-consecutive-terminal-b-dependent} remain valid for}
\(\varepsilon\)-dependent parameters \(b=b_\varepsilon\) and
\(\vartheta=\vartheta_\varepsilon\), provided that $\vartheta_\varepsilon$ is bounded from above. {Indeed, the SHF scale invariance (see
\eqref{eq:SHF-scaling-measure-current} below) remains exact, and the key estimates
\eqref{312} and \eqref{1423} in the proof below remain valid uniformly in \(\varepsilon\).  No lower bound on \(\vartheta_\varepsilon\) is needed, since
decreasing the coupling parameter only decreases the relevant nonnegative
integer moments.}

\end{enumerate}

\end{remark}

\begin{proof}

\textbf{Step 1. Setup.}
Let   \(\Delta_Q:=t_s-t_{r-1}\). By time-translation invariance,
\(\mathscr{Z}^\vartheta_{t_{r-1},t_s}\) has the same law as \(\mathscr{Z}^\vartheta_{0,\Delta_Q}\). We use
the scaling invariance of the critical SHF \cite{shfcritical} in the following form: for
\(\lambda>0\),
\begin{equation}
        \mathscr{Z}^\vartheta_{0,\lambda}\bigl(d(\sqrt\lambda\,x),d(\sqrt\lambda\,y)\bigr)
        \stackrel{\textup{d}}{=}
        \lambda\,\mathscr Z^{\vartheta+\log\lambda}_{0,1}(dx,dy).
\label{eq:SHF-scaling-measure-current}
\end{equation}
Here the LHS denotes the pullback of the measure under
\((x,y)\mapsto(\sqrt\lambda x,\sqrt\lambda y)\), so the Jacobian is already
included in the notation, i.e.
 for every function \(F\),
\begin{equation}
        \int F(x,y)\,
        \mathscr{Z}^\vartheta_{0,\lambda}
        \bigl(d(\sqrt\lambda x),d(\sqrt\lambda y)\bigr)
        \stackrel{\textup{d}}{=}
        \lambda
        \int F(x,y)\,
        \mathscr Z^{\vartheta+\log\lambda}_{0,1}(dx,dy).
\label{eq:scaling-tested-against-F}
\end{equation}
Setting $ F(x,y):=f(\sqrt\lambda x)g(\sqrt\lambda y)$ with nonnegative test functions \(f,g\),
\begin{equation}
\begin{aligned}
&\int f(\sqrt\lambda x)g(\sqrt\lambda y)\,
        \mathscr{Z}^\vartheta_{0,\lambda}
        \bigl(d(\sqrt\lambda x),d(\sqrt\lambda y)\bigr)  =
        \int f(u)g(v)\,
        \mathscr{Z}^\vartheta_{0,\lambda}(du,dv)                        =
        f\blacktriangleleft \mathscr{Z}^\vartheta_{0,\lambda}\blacktriangleright g .
\end{aligned}
\label{eq:left-side-pullback-test}
\end{equation} 
Hence by \eqref{eq:scaling-tested-against-F},
\begin{equation}
        f\blacktriangleleft \mathscr{Z}^\vartheta_{0,\lambda}\blacktriangleright g
        \stackrel{\textup{d}}{=}
        \lambda\int_{\mathbb R^4}
        f(\sqrt\lambda x)g(\sqrt\lambda y)
        \mathscr Z^{\vartheta+\log\lambda}_{0,1}(dx,dy).
\label{eq:SHF-scaling-test-current}
\end{equation}   
Applying this with  $ f=p(t_{r-1},\cdot)$, $g\equiv 1$ and $\lambda=\Delta_Q$, 
\[
\begin{aligned}
        p_{r-1}\blacktriangleleft \mathscr{Z}^\vartheta_{0,\Delta_Q}\blacktriangleright 1
        &\stackrel{\textup{d}}{=}
        \Delta_Q
        \int_{\mathbb R^4}
        p(t_{r-1},\sqrt{\Delta_Q}x)
        \mathscr Z^{\vartheta+\log\Delta_Q}_{0,1}(dx,dy) \\ 
        &=
        \int_{\mathbb R^4}
        p\left(\frac{t_{r-1}}{\Delta_Q},x\right)
        \mathscr Z^{\vartheta+\log\Delta_Q}_{0,1}(dx,dy),
\end{aligned}
\]
where we used  $  p(t_{r-1},\sqrt{\Delta_Q}x)
        =
        \Delta_Q^{-1}
        p(\frac{t_{r-1}}{\Delta_Q},x)$ (a consequence of diffusive invariance of Brownian motion).
Thus, setting
\begin{align} \label{321}
        \rho_Q^2:=\frac{t_{r-1}}{\Delta_Q},
        \qquad
        \Theta_Q:=\vartheta+\log\Delta_Q,
\end{align}
 and recalling  $  \mathscr{Z}^\vartheta_{t_{r-1},t_s}
        \stackrel{\textup{d}}{=}
        \mathscr{Z}^\vartheta_{0,\Delta_Q},$
we obtain
\begin{align}\label{si}
     Z_Q=     p_{r-1}\blacktriangleleft \mathscr{Z}^\vartheta_{t_{r-1},t_s}\blacktriangleright 1
        \stackrel{\textup{d}}{=}
        p(\rho_Q^2)\blacktriangleleft \mathscr Z^{\Theta_Q}_{0,1}\blacktriangleright 1.
\end{align}
{Then upon introducing the centered kernel $ \mathscr W_{t,t'}^\vartheta
    :=\mathscr Z_{t,t'}^\vartheta-P_{t,t'}$, and using the mass-preserving property of the heat kernel, i.e.,
\[
    p_{r-1}\blacktriangleleft P_{t_{r-1},t_s}
    \blacktriangleright 1=1,
    \qquad
    p(\rho_Q^2)\blacktriangleleft P_{0,1}
    \blacktriangleright 1=1,
\]
we may rewrite \eqref{si} as}
\begin{align} \label{315}
     Z_Q
        = 1+  p_{r-1}\blacktriangleleft \mathscr W^\vartheta_{t_{r-1},t_s}\blacktriangleright 1  \stackrel{\textup{d}}{=}
        1+ p(\rho_Q^2)\blacktriangleleft \mathscr W^{\Theta_Q}_{0,1}\blacktriangleright 1 .
\end{align}
Note that 
{since $t_{r-1} = t_s b^{2|Q|} $ (recall $|Q| = s-r+1$)}, we have
\begin{align} \label{322}
       \Delta_Q=t_s(1-b^{2|Q|}),
        \qquad
        \rho_Q^2=\frac{b^{2|Q|}}{1-b^{2|Q|}}.
\end{align}
{Since $b$ is assumed to be less than $1/10,$ and hence in particular less than $1,$} for \(1\le |Q|\le\Gamma\), 
\begin{align} \label{311}
     \log\rho_Q^{-1}
        \le  |Q|\log b^{-1}\le 
        {\Gamma}\log b^{-1}.
\end{align}

\nin
\textbf{Step 2. Proof of (i).}
Recalling \(t_s=\varepsilon^2 b^{-2s}\) and $  N_{\varepsilon,b}
        =
        \left\lfloor
        \frac{\log\varepsilon^{-1}}{\log b^{-1}}
        \right\rfloor,$ we get
\[
        -\log t_s
        =
        2\log\varepsilon^{-1}-2s\log b^{-1}   \ge
        2(N_{\varepsilon,b}-s)\log b^{-1}.
\]   
Hence by \eqref{321},
\begin{align} \label{312}
        -\Theta_Q =  - \vartheta - \log\Delta_Q \overset{\eqref{322}}{\ge} 
        -\vartheta-\log t_s   \ge 
        -\vartheta+2(N_{\varepsilon,b}-s)\log b^{-1}.
\end{align}
Thus, setting
\begin{equation}\label{def53}
        L_Q:=\frac{-\Theta_Q}{\log\rho_Q^{-1}},
\end{equation}
 by \eqref{311} and \eqref{312}, using $b\in (0,1/10)$, for {sufficiently large constant $\tau$} {(depending on $\Gamma$),}
\begin{align} \label{1423}
        L_Q\ge \frac{1}{\Gamma} (N_{\varepsilon,b}-s)- \frac{\vartheta}{\Gamma\log b^{-1}}   \ge \frac{\tau}{\Gamma}   - \frac{|\vartheta|}{\Gamma}\ge c_{**}
\end{align}
where $c_{**}>0$ is the constant from  \eqref{Lbound}. {As $\rho_Q < 1/2$ (see \eqref{322} and recall the condition $b<1/10$),} by \eqref{Lbound} again,  
\[
        \mathbb E\left[
        \left\vert
         p(\rho_Q^2)\blacktriangleleft \mathscr W^{\Theta_Q}_{0,1}\blacktriangleright 1 
        \right\vert^h
        \right] =   \mathbb E\left[
        \left\vert
         p(\rho_Q^2)\blacktriangleleft \mathscr W^{-L_Q \log\rho_Q^{-1}}_{0,1}\blacktriangleright 1 
        \right\vert^h
        \right] 
        \le
        C L_Q^{- h/2 }
        \le
       Cc_{**}^{- h/2 }.
\]
Therefore applying this to \eqref{315}, we conclude the proof of the first part.

~

\nin
\textbf{Step 3. Proof of (ii).}
Note that by \eqref{322}, we have $        \rho_Q^2
        \le
        2b^2.$
Thus, by taking \(b_0>0\) sufficiently small depending on
\(h,\kappa \), we may ensure that \(\rho_Q\) lies in the range where
Lemma~\ref{shrinking} applies, uniformly in all such \(Q\). Hence, noting that  $\Theta_Q \le \vartheta$ (see \eqref{321} with $\Delta_Q \le 1$) 
\begin{align}
        \mathbb E\left[
      Z_Q ^h
        \right]
        & \overset{\eqref{si}}{=}
        \mathbb E\left[
        \left(
        p(\rho_Q^2)
        \blacktriangleleft
        \mathscr Z^{\Theta_Q}_{0,1}
        \blacktriangleright1
        \right)^h
        \right] \nonumber \\
        &
    \le 
        \left(
        \log\rho_Q^{-1}
        \right)^{{h\choose2}+\kappa} \overset{\eqref{311}}{\le}   
         \left(
        |Q| \log b^{-1}
    \right)^{{h\choose2}+\kappa}\le  \left(
        \Gamma\log b^{-1}
    \right)^{{h\choose2}+\kappa}.
\label{eq:terminal-version-apply-shrinking}
\end{align}
Note that by the first inequality in \eqref{311} (which is valid for any block $Q\subset [1,N_{\varepsilon,b}]$),  we generally have
\begin{align} \label{anyblock}
     \mathbb E\left[
      Z_Q ^h \right] \le   \left(
        |Q| \log b^{-1}
    \right)^{{h\choose2}+\kappa} ,\qquad \forall  Q \subset [1,N_{\varepsilon,b}].
\end{align}
   
\end{proof}

\begin{remark}
Using the argument in the proof of Lemma \ref{lem:bounded-consecutive-fourth},  {we obtain a similar moment bound as in Lemma \ref{shrinking}} for the partition function $\mathcal Z_{\varepsilon,b}$ as well. 
Set $\Delta_{\varepsilon,b}
        :=
        t_{N_{\varepsilon,b}-1}-\varepsilon^2 .$
By time-translation invariance of the SHF, 
\[
        \mathcal Z_{\varepsilon,b}
        \stackrel{\mathrm d}{=}
        p(\varepsilon^2)
        \blacktriangleleft
        \mathscr Z^\vartheta_{0,\Delta_{\varepsilon,b}}
        \blacktriangleright1 .
\]
By a scale invariance of SHF, similarly as in \eqref{si},  
\begin{equation}
        \mathcal Z_{\varepsilon,b}
        \stackrel{\mathrm d}{=}
        p(\rho_{\varepsilon,b}^2)
        \blacktriangleleft
        \mathscr Z^{\vartheta_{\varepsilon,b}}_{0,1}
        \blacktriangleright1,
\label{eq:Zepsb-scaled-to-unit-time}
\end{equation}
where
\[
        \rho_{\varepsilon,b}^2
        :=
        \frac{\varepsilon^2}{\Delta_{\varepsilon,b}},
        \qquad
        \vartheta_{\varepsilon,b}
        :=
        \vartheta+\log\Delta_{\varepsilon,b}.
\]
 As $ b^4
        <
        t_{N_{\varepsilon,b}-1}
        \le
        b^2 ,$ we have $\frac{b^4}{2}
        \le
        \Delta_{\varepsilon,b}
        \le
        b^2 $
     for all sufficiently small
\(\varepsilon>0\). 
Consequently,  
\begin{equation}
        \log\rho_{\varepsilon,b}^{-1}
        =
        \log\varepsilon^{-1}+O_b(1).
\label{eq:rho-log-comparison}
\end{equation}
Moreover,   $\vartheta_{\varepsilon,b}
        \le
        \vartheta+\log b^2.$ 
 Hence by Lemma~\ref{shrinking}, for any  $\kappa>0$, for all sufficiently small \(\varepsilon>0\),
\begin{align}
        \mathbb E\left[
        \mathcal Z_{\varepsilon,b}^h
        \right]
        \le
        \left(
        \log\rho_{\varepsilon,b}^{-1}
        \right)^{{h \choose 2}+\kappa} \overset{\eqref{eq:rho-log-comparison}}{\le} 
(\log\varepsilon^{-1})^{{h \choose 2}+2\kappa}.\label{adaptmoment}
\end{align}
\end{remark}

\subsection{Ignoring cancellation for small collections}\label{cancel654}

As indicated earlier in the roadmap, the next lemma is used to control products of $D_I$ when the
total length of the collection is small.   
The following estimate is intentionally crude.  It does not use cancellation in
\(D_i\).  Instead, each \(D_i\) is dominated by a sum of two positive kernels  $Z_i$ and $(Z_i\blacktriangleright1)p_i $,
and the resulting products are factorized into independent blocks.  This
bound is sufficient because it will be applied only to  control $D_I$ with short intervals $I$.

\begin{lemma} 
\label{lem:short-fourth-product}
Let \(\kappa \in (0,1)\) and $\Gamma \ge 10$  be any constants. There exist  constants 
\(b_0=b_0(\kappa )>0\), $C = C(\Gamma )>0$ and  $c_0  >0$   such that for {any
\(\widetilde I\subset [1,N_{\varepsilon,b} - 1]\)},
\begin{equation}
        \left\|
        \prod_{I\in\widetilde I}|D_I|
        \right\|_{L^6}
        \le
       \left(\Gamma \log b^{-1}
        \right)^{C}  \exp\left\{
    c_0|\widetilde I|
        +
        \frac{15+\kappa}{6\Gamma}
        (\log\log b^{-1})|\widetilde I|
        \right\}
\label{eq:short-fourth-moment-refined-K}
\end{equation}
for all $b\in (0,b_0)$ and sufficiently small \(\varepsilon>0\). 
\end{lemma}
It turns out that the \(L^6\)-norm plays no special role: an analogous \(L^p\)-bound holds
for every integer \(p\ge2\), with a \(p\)-dependent coefficient in front of
\(\log\log b^{-1}\). However later arguments involving the H\"older inequality will demand bounds on $L^6$ norms and simple $L^2$ norms will not suffice.

~

The parameter \(\Gamma\) is an auxiliary cutoff  which separates
the   blocks arising from the expansion of \(\prod_{I\in\widetilde I}|D_I|\)
into bounded-size blocks (i.e. size at most $\Gamma$), for which we use the uniform moment bound from Lemma \ref{lem:bounded-consecutive-fourth}, and blocks of length larger than \(\Gamma\), for which we use  Lemma \ref{shrinking}. Since every block in the latter class has
length \(>\Gamma\), the number of such blocks is at most \(|\widetilde I|/\Gamma\).
This is why the only \(b\)-dependent exponential loss in
\eqref{eq:short-fourth-moment-refined-K} has coefficient inversely proportional to $\Gamma.$ 
{Later, in the proof of
Proposition~\ref{prop:second-moment-small-total-length-full}, we choose
\(\Gamma\)  large and \(\kappa\)   small so that the
exponential growth in the total length of the short side arising from
\eqref{eq:short-fourth-moment-refined-K} is dominated by the exponential
decay in the total length of the long side from the \(L^2\)-estimate
\eqref{eq:DI-L2-bound-current}.}

\begin{proof}

Set
\[
        M_i^0:=Z_i,
        \qquad
        M_i^1:=(Z_i\blacktriangleright1)p_i .
\]
Observe that both $M_i^0$ and $M_i^1$ are positive kernels.
Then,
\[
        D_i=Z_i-(Z_i\blacktriangleright1)p_i =  M_i^0- M_i^1.
\]
For \(I=[i_I,j_I]\), multiplying this over all
\(i\in[i_I,j_I-1]\), appending the terminal factor \(Z_{j_I}\), and applying
the triangle inequality yields
\begin{equation}
        |D_I|
        \le
        \sum_{\sigma_I\in\{0,1\}^{[i_I,j_I-1]}}
        p_{i_I-1}
        \blacktriangleleft
        M_{i_I}^{\sigma_I(i_I)}
        \bullet\cdots\bullet
        M_{j_I-1}^{\sigma_I(j_I-1)}
        \bullet Z_{j_I}
        \blacktriangleright1 .
\label{eq:DI-positive-expansion}
\end{equation}
We now explain how each summand above factorizes. Consecutive \(M^0\)-blocks    can be concatenated into a
single SHF kernel. More precisely, by the flow property of the
  SHF,
\[
        Z_u\bullet Z_{u+1}\bullet\cdots\bullet Z_v
        =
        \mathscr{Z}^\vartheta_{t_{u-1},t_v},
        \qquad u\le v .
\] 
On the other hand,  $M_k^1$ cuts the product: the part before the cut is
closed by integrating \(Z_k\) against \(1\), while the part after the cut is
restarted from the  heat kernel \(p_k\).

To be precise, for  \(\sigma_I\in\{0,1\}^{[i_I,j_I-1]}\), define the cut set
\[
        \mathrm{Cut}_I(\sigma_I)
        :=
        \{i_I-1\}
        \cup
        \{k\in[i_I,j_I-1]:\sigma_I(k)=1\}
        \cup
        \{j_I\}.
\]
We write its elements in
increasing order as
\[
        c_0^{ \sigma_I}
        <
        c_1^{ \sigma_I}
        <
        \cdots
        <
        c_{\nu(\sigma_I)}^{ \sigma_I}.
\] 
The identity
\begin{equation}
        p_i\blacktriangleleft B_1\bullet M_j^1\bullet B_2\blacktriangleright 1
        =
        \bigl(p_i\blacktriangleleft B_1\bullet Z_j\blacktriangleright1\bigr)
        \bigl(p_j\blacktriangleleft B_2\blacktriangleright 1 \bigr),
\label{eq:rank-one-cutting-identity}
\end{equation}
valid for arbitrary positive kernels \(B_1,B_2\), shows that every index
\(k\) with \(\sigma_I(k)=1\) creates a cut. Applying
\eqref{eq:rank-one-cutting-identity} at all such indices, and using the flow
property on the stretches where \(\sigma_I=0\),  we get
\begin{equation}
 p_{i_I-1}
        \blacktriangleleft
        M_{i_I}^{\sigma_I(i_I)}
        \bullet\cdots\bullet
        M_{j_I-1}^{\sigma_I(j_I-1)}
        \bullet Z_{j_I}
        \blacktriangleright1                                                 =
        \prod_{r=1}^{\nu(\sigma_I)}
 Z_{\bigl[
        c_{r-1}^{ \sigma_I}+1,\,
        c_r^{ \sigma_I}
        \bigr]},
\label{eq:sigmaI-factorization}
\end{equation}
where, for any interval \(Q=[u,v]\) {(here, unlike in the definition of the error term \(D_I\), in which \(I\) is
assumed to consist of at least two consecutive time intervals, \(Q\) is
allowed to consist of a single interval),}
\begin{align}\label{zq}
        Z_Q
        :=
        p_{u-1}
        \blacktriangleleft
        Z_u\bullet Z_{u+1}\bullet\cdots\bullet Z_v
        \blacktriangleright1   =
        p_{u-1}
        \blacktriangleleft
        \mathscr{Z}^\vartheta_{t_{u-1},t_v}
        \blacktriangleright1.
\end{align} 
Therefore, multiplying \eqref{eq:DI-positive-expansion} over all intervals \(I\in\widetilde I\), we
obtain
\begin{equation}
        \prod_{I\in\widetilde I}|D_I|
        \le
        \sum_{\sigma\in\{0,1\}^{\mathscr D(\widetilde I)}}
        \overline Z_\sigma 
\label{eq:product-positive-expansion}
\end{equation}
where the term $ \overline Z_\sigma $ is defined shortly in  \eqref{eq:Zsigma-block-decomposition}, and \[
        \mathscr D(\widetilde I)
        :=
        \bigcup_{I=[i_I,j_I]\in\widetilde I}
        [i_I,j_I-1].
\]
For a global choice
\(\sigma\in\{0,1\}^{\mathscr D(\widetilde I)}\), we denote by
\(\sigma_I\) its restriction to \([i_I,j_I-1]\). 
We define \(\mathcal B_\sigma\) to be the collection of all subintervals
produced by these cut sets, i.e.,
\begin{equation}
        \mathcal B_\sigma
        :=
        \left\{
        \left[
        c_{r-1}^{ \sigma_I}+1,\,
        c_r^{ \sigma_I}
        \right]:
        I\in\widetilde I,\ 
        1\le r\le \nu(\sigma_I)
        \right\}.
\label{eq:B-sigma-definition}
\end{equation} 
Then the factorization \eqref{eq:sigmaI-factorization}, applied to every
\(I\in\widetilde I\), gives
\begin{equation}
        \overline Z_\sigma
        =
        \prod_{Q\in\mathcal B_\sigma}Z_Q 
\label{eq:Zsigma-block-decomposition}
\end{equation}
where $Z_Q $ appears above in \eqref{zq}.
Since the original intervals \(I\in\widetilde I\) are disjoint, and $[
        c_{r-1}^{ \sigma_I}+1,\,
        c_r^{ \sigma_I}]$ ($  1\le r\le \nu(\sigma_I)$) 
form a partition of \(I\), the intervals \(Q\in\mathcal B_\sigma\) are  
disjoint as well. 
In addition,
\begin{align} \label{346}
     \#\{0,1\}^{\mathscr D(\widetilde I)}
        =
        2^{|\mathscr D(\widetilde I)|}
        =
        2^{\sum_{I\in\widetilde I}(|I|-1)}
        \le
        2^{|\widetilde I|}.
\end{align}

We now split the produced blocks according to their length.  Bounded-size
blocks are controlled by Lemma~\ref{lem:bounded-consecutive-fourth}, and  large blocks (whose size is greater than $\Gamma$) are controlled by Lemma \ref{lem:short-fourth-product}.  Since there can be at most
\(|\widetilde I|/\Gamma\) such long blocks, their total logarithmic loss is
proportional to \(|\widetilde I|/\Gamma\).
More precisely, we split
\[
        \mathcal B_\sigma^{\le \Gamma}
        :=
        \{Q\in\mathcal B_\sigma:|Q|\le \Gamma\},
        \qquad
        \mathcal B_\sigma^{>\Gamma}
        :=
        \{Q\in\mathcal B_\sigma:|Q|>\Gamma\}.
\]
We first control $\mathcal B_\sigma^{\le \Gamma}$ part. Then  by the first part of  
Lemma~\ref{lem:bounded-consecutive-fourth},  there is $\tau>0$ such that for $b\in (0,1/10),$
\begin{align} \label{376}
        \|Z_Q\|_{L^6}\le C,
        \qquad
       \forall Q\in\mathcal B_\sigma^{\le \Gamma}\ \text{  and } \ Q\subset [1,N_{\varepsilon,b}-\tau].
\end{align}
In addition   by the second part of  
Lemma~\ref{lem:bounded-consecutive-fourth},   for any    small enough constant $b>0$,
\begin{align} \label{377}
      \|Z_Q\|_{L^6}\le   \left(\Gamma \log b^{-1}
        \right)^{3}, \qquad
       \forall Q\in\mathcal B_\sigma^{\le \Gamma}\ \text{  and } \ Q\subset [1,N_{\varepsilon,b}] 
\end{align}
We multiply \eqref{376} and \eqref{377} for all  $Q\in\mathcal B_\sigma^{\le \Gamma}$. Precisely, we use  \eqref{376}  when $Q\subset [1,N_{\varepsilon,b}-\tau],$ and  we use  \eqref{377} otherwise. Note that the number of intervals $Q\in\mathcal B_\sigma^{\le \Gamma}$ that belong to the latter case is at most $\tau.$  
Thus by independence, using  $|\mathcal B_\sigma^{\le \Gamma}|\le |\widetilde I|,$
\begin{equation}
        \prod_{Q\in\mathcal B_\sigma^{\le \Gamma}}
        \|Z_Q\|_{L^6}
        \le
        C^{|\mathcal B_\sigma^{\le \Gamma}|}   \cdot \left(\Gamma \log b^{-1}
        \right)^{3\tau} =   \left(\Gamma \log b^{-1}
        \right)^{3\tau}
        \exp\{\log C \cdot |\widetilde I|\}.
\label{eq:bounded-blocks-L6}
\end{equation} 
We next control the  $ \mathcal B_\sigma^{>\Gamma}$ part. For $  Q=[r,s]\in\mathcal B_\sigma^{>\Gamma},$  by \eqref{anyblock} with $h=6,$ 
\[
\begin{aligned}
  \|Z_Q\|_{L^6}
\le     \bigl(|Q|\log b^{-1}\bigr)^{(15+\kappa)/6}.
\end{aligned}
\label{eq:ZQ-L6-long-block}
\]
We now multiply this estimate for all
\(Q\in\mathcal B_\sigma^{>\Gamma}\). Setting $ p_\kappa:=\frac{15+\kappa}{6},$ we have 
\begin{align} \label{366}
        \prod_{Q\in\mathcal B_\sigma^{>\Gamma}}
        \|Z_Q\|_{L^6}
        \le
2^{p_\kappa|\mathcal B_\sigma^{>\Gamma}|}
        \left(\log b^{-1}\right)^{
        p_\kappa|\mathcal B_\sigma^{>\Gamma}|}
        \prod_{Q\in\mathcal B_\sigma^{>\Gamma}}|Q|^{p_\kappa}.
\end{align}
Since $x> \log x$ for $x \ge 10$, recalling that  $\Gamma >10$ (by choice in the statement of the lemma) 
\[
        \prod_{Q\in\mathcal B_\sigma^{>\Gamma}}|Q|^{p_\kappa}
       = \exp\Big\{{p_\kappa} \sum_{Q\in\mathcal B_\sigma^{>\Gamma}}\log |Q|\Big\} \le \exp\Big\{{p_\kappa} \sum_{Q\in\mathcal B_\sigma^{>\Gamma}}|Q|\Big\}    \le
        \exp\{{p_\kappa} |\widetilde I|\}.
\]
In addition, as each block in \(\mathcal B_\sigma^{>\Gamma}\) has length
larger than \(\Gamma\), we have $  |\mathcal B_\sigma^{>\Gamma}|
        \le
        {|\widetilde I|}/{\Gamma}.$ Hence,
\[
        \left(\log b^{-1}\right)^{
        p_\kappa|\mathcal B_\sigma^{>\Gamma}|}
        \le
        \exp\left\{
        \frac{15+\kappa}{6\Gamma}
        (\log\log b^{-1})|\widetilde I|
        \right\}.
\]
Applying the above two estimates  to \eqref{366}, noting that $2^{p_\kappa} \le  e^{p_\kappa} \le e^3$ (since $\kappa \in (0,1)$),  
\begin{equation}
        \prod_{Q\in\mathcal B_\sigma^{>\Gamma}}
        \|Z_Q\|_{L^6}
        \le
        \exp\left\{
      {6|\widetilde I|}
        +
        \frac{15+\kappa}{6\Gamma}
        (\log\log b^{-1})|\widetilde I|
        \right\}.
\label{eq:long-blocks-L6}
\end{equation}
Combining this with \eqref{eq:bounded-blocks-L6}, and using
the independence of the disjoint blocks \(Q\in\mathcal B_\sigma\),  
\[
\begin{aligned}
        \|\overline Z_\sigma\|_{L^6}
         =
        \prod_{Q\in\mathcal B_\sigma}
        \|Z_Q\|_{L^6}                                 \le
    \left(\Gamma \log b^{-1}
        \right)^{3\tau}     \exp\left\{
   (\log C + 4) |\widetilde I|
        +
        \frac{15+\kappa}{6\Gamma}
        (\log\log b^{-1})|\widetilde I|
        \right\}.
\end{aligned}
\label{eq:Zsigma-L6-final}
\]
Note that the above bound is independent of $\sigma\in \{0,1\}^{\mathscr D(\widetilde I)}$.

Recall that the claim was 

\begin{equation}
        \left\|
        \prod_{I\in\widetilde I}|D_I|
        \right\|_{L^6}
        \le
       \left(\Gamma \log b^{-1}
        \right)^{C}  \exp\left\{
    c_0|\widetilde I|
        +
        \frac{15+\kappa}{6\Gamma}
        (\log\log b^{-1})|\widetilde I|
        \right\}.
\end{equation}
To obtain this we simply apply the $L^6$-norm Minkowski's inequality and the entropy bound $ \#\{0,1\}^{\mathscr D(\widetilde I)}
        \le
        2^{|\widetilde I|}$ (see \eqref{346}) to the inequality \eqref{eq:product-positive-expansion}. {Note that the above bound holds with $c_0=\log 2 +  (\log C + 4 )$ and $C=3\tau$.}
\end{proof}

\subsection{Conditional densities projected onto small number of coordinates.}\label{rn535} This is the central ingredient in the proof of Proposition \ref{prop:second-moment-small-total-length-full} and one of the main technical underpinnings of this paper. While the estimate recorded here will be a consequence of a much more elaborate statement recorded and proved later in Section \ref{rnproof34}, we extract the input we need for the upcoming proof of Proposition \ref{prop:second-moment-small-total-length-full}. 
We start with a bit of notational preparation.
Recall from \eqref{utdef}, \begin{align} \label{utdef21}
\mathrm{UT}_{\varepsilon,\eta}(\alpha)
        :=
        \left\{
        G_{\eta;1}        \ge
      N_{\varepsilon,b}^{-\frac{1-\eta}{2} + \alpha}
        \right\}.
\end{align}

Note that the upper-tail event $\mathrm{UT}_{\varepsilon,\eta}(\alpha)$ is imposed on $\prod_{i=1}^{M_{\varepsilon, \eta}} G_i,$ or equivalently, the full bulk sum
\[
 \sum_{i=1}^{M_{\varepsilon, \eta}}\log G_i,
\]
whereas in our 
applications we will only need to track the conditional law of $G_{i}$ for $i$ in a small subset \(U\) of the
coordinates. {To this end, we write $    G|_U:=(G_i)_{i\in U}$.}

\begin{proposition}
\label{prelimcond}
Let \(\alpha>0\). 
Then there exists \(b_0=b_0(\alpha )>0\) such that the following holds for
any constant \(b\in(0,b_0)\). Let $\eta \in (0,1/10)$ and $  U\subset\{1,2,\dots,M_{\varepsilon, \eta}\}$. Define the Radon-Nikodym derivative
\begin{equation}
        \mathsf{RN}_{U}^{(\alpha)}
        :=
        \frac{
        d\textup{Law}(G|_U\mid  \mathrm{UT}_{\varepsilon,\eta}(\alpha))
        }{
        d\textup{Law}(G|_U).
        }
\label{eq:projected-G-RN-definition}
\end{equation}  
That is, \(\mathsf{RN}_{U}^{(\alpha)}\) is the Radon--Nikodym derivative
of the marginal law of the \(U\)-coordinates of \((G_i)_{1\le i\le M_{\varepsilon,\eta}}\) under the conditional measure
\(\mathbb P(\,\cdot\,\mid \mathrm{UT}_{\varepsilon,\eta}(\alpha))\), with respect to
their original marginal law (note that \(\mathrm{UT}_{\varepsilon,\eta}(\alpha)\) is measurable with respect to
\((G_i)_{1\le i\le M_{\varepsilon,\eta}}\)). Then for any integer \(p\ge2\),   as \(\varepsilon \rightarrow 0\),
if $|U| = O(\log \log \varepsilon^{-1}),$ 
\begin{equation}
\mathbb E\bigl[|\mathsf{RN}^{(\alpha)}_U-1|^p\bigr]
\le  
       (\log \varepsilon^{-1})^{-p\eta/2 + o(1)}.
\label{rnspecial32}
\end{equation}
\end{proposition}

\vspace{.1in}
Given the preparation we now proceed with proving Proposition \ref{prop:second-moment-small-total-length-full}.

\subsection{Proof of Proposition \ref{prop:second-moment-small-total-length-full}}\label{centralsec} 
Fix any $\kappa\in (0,1)$. Then we take $\Gamma > 10$ large enough  such that 
\begin{equation}
        \frac15>\frac{15+\kappa}{6\Gamma}.
\label{eq:K0-kappa-full-choice}
\end{equation}
Then choose a constant \(b>0\) sufficiently small so that Lemma \ref{lem:short-fourth-product} holds and further satisfies
\begin{equation}
        \left(
        \frac15-\frac{15+\kappa}{6\Gamma}
        \right)\log\log b^{-1}>
      c_0 + {2\log 2}, 
\label{eq:Gamma-b-full}
\end{equation}
where $c_0$ is the constant from Lemma \ref{lem:short-fourth-product}.

We begin with the   expansion of $\mathcal R_{\eta,\le\delta}  $ in \eqref{short}. On the event \(\Omega_{\varepsilon,\eta}\), we have $ (G_I^{\rm bulk})^{-1}
        \le
        2^{|I^{\rm bulk}|}
        \le
        2^{|I|}.$ Thus by triangle inequality,
\begin{equation}
\begin{aligned}
&\mathbb E\left[
\mathbf 1_{\Omega_{\varepsilon,\eta}}
\left|
\mathcal R_{\eta,\le\delta}  
\right|^2
\,\middle|\,
\mathrm{UT}_{\varepsilon,\eta}(\alpha)
\right]                                          \le
\sum_{\substack{
\widetilde I,\widetilde J\in\mathfrak I_{\varepsilon,b}  \\
{2}\le |\widetilde I|,|\widetilde J|\le\delta \mathcal L_\varepsilon}}
\mathbb E\left[
\mathcal Y(\widetilde I,\widetilde J)
\,\middle|\,
\mathrm{UT}_{\varepsilon,\eta}(\alpha)
\right],
\end{aligned}
\label{eq:full-second-moment-expanded-direct}
\end{equation}
where
\begin{align} \label{404}    
        \mathcal Y(\widetilde I,\widetilde J)
        :=
        2^{|\widetilde I|+|\widetilde J|}
        \prod_{I\in\widetilde I}|D_I|\cdot 
        \prod_{J\in\widetilde J}|D_J|
      \cdot           \widehat G_{\widetilde I}^{\rm tail}
               \widehat G_{\widetilde J}^{\rm tail} \ge 0.
\end{align} 
Recall that by \eqref{defii} every $I\in \wt I$ and similarly $J\in \wt J$ are subsets of $[1, N_{\e,b}-1]$ and are of size at least $2.$
Now for $\widetilde I,\widetilde J\in\mathfrak I_{\varepsilon,b}  $ with $
{2}\le |\widetilde I|,|\widetilde J|\le\delta \mathcal L_\varepsilon$, let $ \widetilde K:=\widetilde K(\widetilde I,\widetilde J)$ 
be the collection of  connected components (intervals) $K$ generated by the
intervals in \(\widetilde I\cup\widetilde J\) in  
\([1,N_{\varepsilon,b}-1]\). First note that each connected component is of size at least $2.$ Further, let us  
write $\widetilde K^{\text{bulk}}:=
     \widetilde K  \cap[1,M_{\varepsilon,\eta}],$  as the collection of intervals $K^{\text{bulk}}=K \cap[1,M_{\varepsilon,\eta}].$ Now  define the $\sigma$-algebra
\[
        \mathcal H_{\widetilde K}
        :=
        \sigma\bigl((Z_i)_{i\in \widetilde K^{\text{bulk}}}\bigr)
        \vee
        \sigma\bigl((Z_i)_{M_{\varepsilon,\eta}<i<N_{\varepsilon,b}}\bigr).
\]
Here, recall that $Z_i$ is a random element in 
 \(\mathcal M_+ \), i.e.  the space of positive locally finite Borel
measures, equipped with the Borel \(\sigma\)-field induced by the
vague topology.
Then \(\mathcal Y(\widetilde I,\widetilde J)\) is
\(\mathcal H_{\widetilde K}\)-measurable. Indeed,  \(D_I\) and
\(D_J\) (for $I\in \widetilde I$ and $J\in \widetilde J$) use only the blocks lying in the intervals of
\(\widetilde I\) and \(\widetilde J\), respectively; their bulk indices are
therefore contained in \(\widetilde K^{\text{bulk}}\). The terminal indices in  \(D_I\) and
\(D_J\), along with
\(        \widehat G_{\widetilde I}^{\rm tail}\) and
\(       \widehat G_{\widetilde J}^{\rm tail}\), only use the tail variables
\((Z_i)_{M_{\varepsilon,\eta}<i< N_{\varepsilon,b}}\).  Hence \(\mathcal Y(\widetilde I,\widetilde J)\) is
\(\mathcal H_{\widetilde K}\)-measurable by 
Corollary \ref{cormeasure}.

Recall from \eqref{eq:projected-G-RN-definition} that \(\mathsf{RN}_{U}^{(\alpha)}\) is 
 the marginal law of the \(U\)-coordinates of \((G_i)_{1\le i\le M_{\varepsilon,\eta}}\) under the conditional measure
\(\mathbb P(\,\cdot\,\mid \mathrm{UT}_{\varepsilon,\eta}(\alpha))\), with respect to
their original marginal law (note that \(\mathrm{UT}_{\varepsilon,\eta}(\alpha)\) is measurable with respect to
\((G_i)_{1\le i\le M_{\varepsilon,\eta}}\)). Thus this conditioning only affects the distributions of the bulk variables. In particular,  by Lemma \ref{lem:projected-RN-coarse-conditioning}, writing $ \widetilde K=\widetilde K(\widetilde I,\widetilde J)$,
\begin{align} \label{rnrn}
\mathbb E\left[
\mathcal Y(\widetilde I,\widetilde J)
\,\middle|\,
 \mathrm{UT}_{\varepsilon,\eta}(\alpha)
\right]  =
\mathbb E\left[
\mathcal Y(\widetilde I,\widetilde J) \cdot 
\mathsf{RN}_{\widetilde K^{\text{bulk}}}^{(\alpha)}
\right].
\end{align}

Note above that the Radon-Nikodym factor only involves the bulk variables.
Moreover, noting that $\widetilde K^{\text{bulk}} \subset \{1,2,\ldots,M_{\varepsilon,  \eta} \} $ and  $ |\widetilde K^{\text{bulk}}|
        \le  |\widetilde K| \le 
        |\widetilde I|+|\widetilde J|
        \le
        2\delta \log \log\varepsilon^{-1} $,
 by the bound on   the Radon-Nikodym derivative recorded in \eqref{rnspecial32}, for any small enough constant $b>0$ (depending on $\alpha$) and an integer $p \ge 2$, for sufficiently small $\varepsilon>0,$
\begin{equation}
        \left\|\mathsf{RN}_{\widetilde K^{\text{bulk}}}^{(\alpha)}\right\|_{L^p}
        \le 1+\left\|\mathsf{RN}_{\widetilde K^{\text{bulk}}}^{(\alpha)} - 1  \right\|_{L^p}\le 1+C
        (\log \varepsilon^{-1})^{- \eta/2 + o(1)} \le 2.
\label{eq:RN-density-L6-bounded}
\end{equation} 
We next estimate the terminal product in \eqref{404}.   
By  \eqref{eq:G-tail-product-moment-consequence}, for any $p \ge 2,$ for sufficiently small $\varepsilon>0,$
\begin{align} \label{eq:tail-product-L6-eta}
\left\|
        \widehat G_{\widetilde I}^{\rm tail}
       \widehat G_{\widetilde J}^{\rm tail}
\right\|_{L^p}
&\le
\left(
\prod_{i\in \mathcal  T(\widetilde I)}
\mathbb E G_i^{2p} \cdot 
\prod_{j\in \mathcal  T(\widetilde J)}
\mathbb E G_j^{2p}
\right)^{1/(2p)}     \le  
        (\log\varepsilon^{-1})^{ C\eta}.
\end{align}

Now, 
the following long/short split is the main device in the proof.  In each
connected component, the side with larger total length is estimated in \(L^2\) using the estimate recorded in \eqref{eq:DI-L2-bound-current},
where each \(D_I\) essentially contributes an exponentially decaying factor
\((\log b^{-1})^{-(|I|-1)/2}\). The shorter side is  controlled via the simpler bound in Lemma \ref{lem:short-fourth-product},
but its total length is dominated by the long side.  By choosing \(b\) sufficiently small, the
\(L^2\)-decay on the longer side becomes strong enough to absorb the 
loss from the shorter side.
Precisely,  for each interval
\(K\in\widetilde K=\widetilde K(\widetilde I,\widetilde J)\), define
\[
        \widetilde I_K:=\{I\in\widetilde I:I\subset K\},
        \qquad
        \widetilde J_K:=\{J\in\widetilde J:J\subset K\}.
\]
If
\[
       \sum_{I\in\widetilde I_K}|I| \ge 
 \sum_{J\in\widetilde J_K}|J| 
\]
then we set
\begin{align} \label{441}
        \text{Long}_K:=\widetilde I_K,
        \qquad
        \text{Short}_K:=\widetilde J_K,
\end{align}
 otherwise we set 
\begin{align} 
        \text{Long}_K:=\widetilde J_K,
        \qquad
        \text{Short}_K:=\widetilde I_K.
\end{align}
{Note that,  in contrast to  $\text{Long}_K$, the set $  \text{Short}_K$ may be empty.}
In other words, \(\text{Long}_K\) is the longer local side and \(\text{Short}_K\) is the shorter
local side  in $K$. Define
\begin{align}  \label{deflong}
        m_{\text{long}}:=
        \sum_{K\in\widetilde K}\sum_{A\in\text{Long}_K}|A|,
        \qquad
        m_{\text{short}}:=
        \sum_{K\in\widetilde K}\sum_{A\in\text{Short}_K}|A|.
\end{align}
Then
\begin{equation}
        m_{\text{short}}\le m_{\text{long}},
        \qquad
        |\widetilde K|
        =
        \sum_{K\in\widetilde K}|K|
        \le
        m_{\text{long}}+m_{\text{short}}
        \le
        2m_{\text{long}}.
\label{eq:long-short-full-relation}
\end{equation}
Now, we set
\begin{align}
        Y_{\widetilde K}^{\rm long}
        :=
        \prod_{K\in\widetilde K}
        \prod_{A\in\text{Long}_K}|D_A|,
        \qquad
        Y_{\widetilde K}^{\rm short}
        :=
        \prod_{K\in\widetilde K}
        \prod_{A\in\text{Short}_K}|D_A|.
\end{align}
{Here, we adopt the convention that a product indexed by the empty set is equal to 1.}
Then
\begin{align} \label{444}
        \prod_{I\in\widetilde I}|D_I| \cdot 
        \prod_{J\in\widetilde J}|D_J|
        =
        Y_{\widetilde K}^{\rm long}\cdot 
        Y_{\widetilde K}^{\rm short}.    
\end{align}
Thus by \eqref{rnrn} and the definition of $\mathcal Y(\widetilde I,\widetilde J)$ in  \eqref{404}, applying Hölder's inequality with
exponents \(2,6,6,6\), i.e., using the fact that for random variables $X_1, X_2, X_3, X_4$ on the same probability space, $$\E(|X_1X_2X_3X_4|)\le (\E X_1^{2})^{1/2} (\E X_2^{6})^{1/6} (\E X_3^{6})^{1/6} (\E X_4^{6})^{1/6},$$ we get
\begin{equation}
\begin{aligned}
\mathbb E\left[
\mathcal Y(\widetilde I,\widetilde J)
\,\middle|\,
\mathrm{UT}_{\varepsilon,\eta}(\alpha)
\right]
                                & =   2^{|\widetilde I|+|\widetilde J|} \mathbb E\left[
        \prod_{I\in\widetilde I}|D_I|\cdot 
        \prod_{J\in\widetilde J}|D_J|\cdot 
                \widehat G_{\widetilde I}^{\rm tail} 
               \widehat G_{\widetilde J}^{\rm tail}   \cdot \mathsf{RN}_{\widetilde K^{\text{bulk}}}^{(\alpha)} \right] \\
        &=2^{|\widetilde I|+|\widetilde J|} \mathbb E\left[
       Y_{\widetilde K}^{\rm long} \cdot 
        Y_{\widetilde K}^{\rm short} \cdot  
                \widehat G_{\widetilde I}^{\rm tail}
               \widehat G_{\widetilde J}^{\rm tail}\cdot  \mathsf{RN}_{\widetilde K^{\text{bulk}}}^{(\alpha)} \right]\\
& \le
2^{|\widetilde I|+|\widetilde J|}
\left\|Y_{\widetilde K}^{\rm long}\right\|_{L^2} 
\left\|Y_{\widetilde K}^{\rm short}\right\|_{L^6}
\left\|
        \widehat G_{\widetilde I}^{\rm tail}
       \widehat G_{\widetilde J}^{\rm tail}
\right\|_{L^6}
\left\|\mathsf{RN}_{\widetilde K^{\text{bulk}}}^{(\alpha)}\right\|_{L^6}.
\end{aligned}
\label{eq:holder-direct-full}
\end{equation}
We estimate the ``long'' factor first. By  independence, using the decay estimate \eqref{eq:DI-L2-bound-current},
\[
\left\|Y_{\widetilde K}^{\rm long}\right\|_{L^2} =   \prod_{K\in\widetilde K}
\prod_{A\in\text{Long}_K}  \left\|D_A\right\|_{L^2} 
\le
\prod_{K\in\widetilde K}
\prod_{A\in\text{Long}_K}
\frac{\mathsf c^{|A|/2}}
{(\log b^{-1})^{(|A|-1)/2}}
\frac1{N_{\varepsilon,b}-j_A},
\]
where $A=[i_A,j_A].$
Since each interval $A$ has length at least \(2\), using $(|A|-1)/2 \ge |A|/4$ and recalling the definition of $m_{\text{long}}$ in \eqref{deflong}, for any small enough constant  \(b>0\),
\begin{equation}
        \left\|Y_{\widetilde K}^{\rm long}\right\|_{L^2}
        \le
        \exp\left\{
        -\frac15(\log\log b^{-1})m_{\text{long}}
        \right\}
        \prod_{K\in\widetilde K}
        \frac1{N_{\varepsilon,b}-j_K},
\label{eq:long-L2-full}
\end{equation} 
where $K=[i_K,j_K].$ This is because for a fixed $K\in \widetilde K$,
the
product of all the factors  $  \frac1{N_{\varepsilon,b}-j_A}$ over  $A\in\text{Long}_K$ is bounded by a single
factor $\frac1{N_{\varepsilon,b}-j_K}$, since $A\subseteq K$ implies $j_A\le j_K$ and all remaining 
factors are at most one. Here it is crucially used  that  $\text{Long}_K$ is non-empty.

For the ``short'' factor, by Lemma \ref{lem:short-fourth-product},  for any small enough constant $b>0,$
\begin{equation}
        \left\|Y_{\widetilde K}^{\rm short}\right\|_{L^6}
        \le       \left(\Gamma \log b^{-1}
        \right)^{C}  
        \exp\left\{
        c_0m_{\text{short}}
        +
        \frac{15+\kappa}{6\Gamma}
        (\log\log b^{-1})m_{\text{short}}
        \right\}.
\label{eq:short-L6-full}
\end{equation}
For the factors, $\left\|
        \widehat G_{\widetilde I}^{\rm tail}
       \widehat G_{\widetilde J}^{\rm tail}
\right\|_{L^6}$ and $
\left\|\mathsf{RN}_{\widetilde K^{\text{bulk}}}^{(\alpha)}\right\|_{L^6}$ we apply \eqref{eq:tail-product-L6-eta} and  \eqref{eq:RN-density-L6-bounded} respectively, with $p=6$ to obtain the bounds $(\log\varepsilon^{-1})^{ C\eta}$ and $2$.
Putting everything together and using
\[
         m_{\text{short}}\le m_{\text{long}},\qquad |\widetilde I|+|\widetilde J|
        =
        m_{\text{long}}+m_{\text{short}}
        \le
        2m_{\text{long}},       
\]
we get
for any small enough constant $b>0$ satisfying the condition \eqref{eq:Gamma-b-full}, for all sufficiently small $\varepsilon>0,$
\begin{equation}
\mathbb E\left[
\mathcal Y(\widetilde I,\widetilde J)
\,\middle|\,
\mathrm{UT}_{\varepsilon,\eta}(\alpha)
\right]
\le 
(\log\varepsilon^{-1})^{ C\eta}
\prod_{K\in\widetilde K(\widetilde I,\widetilde J)}
\frac1{N_{\varepsilon,b}-j_K}.
\label{eq:component-bound-before-sum}
\end{equation} Here we absorbed the factor $   \left(\Gamma \log b^{-1}
        \right)^{C} $ into the term  $(\log\varepsilon^{-1})^{ C\eta}$ (by increasing the value of $C$), as we take $\varepsilon \rightarrow 0$ after $b>0$ is taken to be a small constant.
{Indeed, for each fixed \(b\), the former factor is independent of
\(\varepsilon\), whereas \((\log\varepsilon^{-1})^\eta\to\infty\), and
hence the desired absorption holds for all sufficiently small
\(\varepsilon\), with the threshold depending on \(b\).}
Observe that the bound obtained above depends on \((\widetilde I,\widetilde J)\) only
through the connected components \(\widetilde K=\widetilde K(\widetilde I,\widetilde J)\)
and their right endpoints.  We therefore regroup the sum in \eqref{eq:full-second-moment-expanded-direct} according to
\(\widetilde K\).   Let
\(\mathfrak K_{\varepsilon,b}^{\delta}\) be the set of all
collections \(\widetilde K\) of the union of disjoint intervals
\(K\subset[1,N_{\varepsilon,b}-1] \), each satisfying \(|K|\ge2\), such that
\[
        |\widetilde K|
        :=
        \sum_{K\in\widetilde K}|K|
        \le
        2\delta \mathcal L_\varepsilon .
\]
For \(\widetilde K\in\mathfrak K_{\varepsilon,b}^{\delta}\), define
\[
        \mathfrak S_{\widetilde K}^{2,\delta}
        :=
        \left\{
        (\widetilde I,\widetilde J)\in
        (\mathfrak I_{\varepsilon,b}  )^2:
        \widetilde K(\widetilde I,\widetilde J)=\widetilde K,\ 
        {2}\le |\widetilde I|,|\widetilde J|\le\delta \mathcal L_\varepsilon
        \right\}.
\] 
We regroup the sum according to the  map $     (\widetilde I,\widetilde J)
        \longmapsto
        \widetilde K(\widetilde I,\widetilde J),$ i.e. we write
\begin{equation}
\begin{aligned}
&\sum_{\substack{
\widetilde I,\widetilde J\in\mathfrak I_{\varepsilon,b}  \\
2\le|\widetilde I|,|\widetilde J|\le\delta \mathcal L_\varepsilon}}
\mathbb E\left[
\mathcal Y(\widetilde I,\widetilde J)
\,\middle|\,
\mathrm{UT}_{\varepsilon,\eta}(\alpha)
\right]                                            =
\sum_{\widetilde K\in\mathfrak K_{\varepsilon,b}^{\delta}}
\sum_{(\widetilde I,\widetilde J)\in\mathfrak S_{\widetilde K}^{2,\delta}}
\mathbb E\left[
\mathcal Y(\widetilde I,\widetilde J)
\,\middle|\,
\mathrm{UT}_{\varepsilon,\eta}(\alpha)
\right].
\end{aligned}
\label{eq:double-counting-by-K-full}
\end{equation} 
We bound the cardinality of $\mathfrak S_{\widetilde K}^{2,\delta}$.  Every local interval collection \(\widetilde I_K\) is determined by the
set of its left endpoints and the set of its right endpoints.  Each of these
sets is a subset of \(K\).  Thus the number of possible local interval
collections inside \(K\) is at most $ 2^{|K|}\cdot 2^{|K|}
        =
        2^{2|K|}.$ 
Hence, the number of possible local pairs
\((\widetilde I_K,\widetilde J_K)\) is  at most $(2^{2|K|})^2 = 16^{|K|},$ implying that  $ |  \mathfrak S_{\widetilde K}^{2,\delta}|
        \le
        \prod_{K\in\widetilde K}16^{|K|}
        =
        16^{|\widetilde K|}.$ Thus
using \eqref{eq:component-bound-before-sum} and \eqref{eq:double-counting-by-K-full},
\begin{equation}
\begin{aligned}
&\sum_{\substack{
\widetilde I,\widetilde J\in\mathfrak I_{\varepsilon,b}  \\
2\le|\widetilde I|,|\widetilde J|\le\delta \mathcal L_\varepsilon}}
\mathbb E\left[
\mathcal Y(\widetilde I,\widetilde J)
\,\middle|\,
\mathrm{UT}_{\varepsilon,\eta}(\alpha)
\right]                   \le
 (\log\varepsilon^{-1})^{ C\eta}
\sum_{\widetilde K\in\mathfrak K_{\varepsilon,b}^{\delta}} 16^{|\widetilde K|} 
\prod_{K\in\widetilde K}
\frac1{N_{\varepsilon,b}-j_K}.
\end{aligned}
\label{eq:sum-over-K-before-fourfold}
\end{equation}
For \(\widetilde K\in
\mathfrak K_{\varepsilon,b}^{\delta}\) consisting of disjoint intervals $K_1,\ldots,K_q$, with $K_\ell=[i_\ell,j_\ell]$ for $1\le \ell \le q,$ we set
\[
        m:=|\widetilde K|=\sum_{r=1}^q |K_r| \le 2\delta \mathcal L_\varepsilon,
        \qquad
        v_r:=|K_r|,
        \qquad
        h_r:=j_{K_r}
        \quad (\text{for}  \ r=1,\ldots,q).
\] 
Then as before (see after \eqref{11correspond}), the data $(m,q,(v_1,\ldots,v_q),(h_1,\ldots,h_q))$ 
determine \(\widetilde K\) uniquely.  
Conversely, every \(\widetilde K\in \mathfrak K_{\varepsilon,b}^{\delta}\)
gives such a tuple, subject to the admissibility conditions.  Dropping the admissibility
constraints, the RHS of  \eqref{eq:sum-over-K-before-fourfold} is bounded by
\begin{equation}
\begin{aligned}  
  (\log\varepsilon^{-1})^{ C\eta}
\sum_{m=1}^{2\delta \mathcal L_\varepsilon}
16^m
\sum_{q=1}^{m}
\sum_{\substack{v_1,\ldots,v_q\ge2\\
v_1+\cdots+v_q=m}}
\sum_{  1\le  h_1<\cdots<  h_q <  N_{\varepsilon,b}}
\prod_{r=1}^{q}
\frac1{N_{\varepsilon,b}-  h_r}.
\end{aligned}
\label{eq:fourfold-sum-full}
\end{equation}
For fixed \(q\), recalling $N_{\varepsilon,b}\le \frac{1}{10} \log \varepsilon^{-1}$ for a small constant $b>0$ (see \eqref{ndef}) and $\mathcal L_\varepsilon = \log \log \varepsilon^{-1},$ 
\begin{equation}
        \sum_{ 1\le  h_1<\cdots<  h_q< N_{\varepsilon,b}}
        \prod_{r=1}^{q}
        \frac1{N_{\varepsilon,b}-  h_r}
        \le
        \frac1{q!}
        \left(
        \sum_{h=1}^{N_{\varepsilon,b}-1}
        \frac1{N_{\varepsilon,b}-h}
        \right)^q
        \le
        \frac{ \mathcal L_\varepsilon^q}{q!}.
\label{eq:endpoint-sum-full}
\end{equation}
For fixed \(m,q\), the number of component-length profiles is bounded by (see \eqref{combinatorics})
\[
        \#\{(v_1,\ldots,v_q):v_r\ge2,\ v_1+\cdots+v_q=m\}
        \le
        2^m.
\]
Applying the above two bounds to \eqref{eq:fourfold-sum-full}, we obtain
\begin{equation}
\begin{aligned}
&\sum_{\substack{
\widetilde I,\widetilde J\in\mathfrak I_{\varepsilon,b}  \\
{2}\le|\widetilde I|,|\widetilde J|\le\delta \mathcal L_\varepsilon}}
\mathbb E\left[
\mathcal Y(\widetilde I,\widetilde J)
\,\middle|\,
\mathrm{UT}_{\varepsilon,\eta}(\alpha)
\right]     \le
  (\log\varepsilon^{-1})^{ C\eta}
\sum_{m=1}^{2\delta \mathcal L_\varepsilon}
16^m\cdot 2^m
\sum_{q=1}^{m}
\frac{\mathcal L_\varepsilon^q}{q!}.
\end{aligned}
\label{eq:entropy-before-final-full}
\end{equation}
We now estimate the last  sum.  
Fix \(1\le m\le2\delta \mathcal L_\varepsilon < \mathcal L_\varepsilon\) (recall $\delta \in (0,1/10)$). As $ q\mapsto \frac{ \mathcal L_\varepsilon^q}{q!}$
is increasing for \(1\le q\le m\) (since the ratio of two adjacent terms is $ \frac{\mathcal L_\varepsilon}{q+1} \ge 1$), 
\begin{align}\label{349}
        32^m \sum_{q=1}^{m}
        \frac{ \mathcal L_\varepsilon^q}{q!}
        \le
        32^m\cdot  m\frac{ \mathcal L_\varepsilon^m}{m!}
        \le
        m\left(\frac{32e \mathcal L_\varepsilon}{m}\right)^m.
\end{align}
{In the last inequality we used that $m ! \ge \left(\frac{m}{e}\right)^m.$}
Taking logarithm, and using $m\le \cL_{\e}$ 
\[
     \log m+ m\log\left(\frac{32e \mathcal L_\varepsilon}{m}\right)
        \le \log \mathcal L_\varepsilon +
        2\delta \mathcal L_\varepsilon
        \log\left(\frac{32e}{2\delta}\right)
        \le   C\delta \mathcal L_\varepsilon\log(1/\delta),
\]
for all sufficiently small \(\delta>0\). {Above we used that $x\log(1/x)$ is monotonically increasing for $x\in (0,2\delta)$ for a small enough $\delta>0.$} Thus,
uniformly over \(1\le m\le2\delta \mathcal L_\varepsilon\), the quantity \eqref{349} is at most $ \exp\left\{
        C\delta \mathcal L_\varepsilon\log(1/\delta)
        \right\}.$ 
Summing over  \(1\le m\le2\delta \mathcal L_\varepsilon\), and
absorbing the subexponential factor, we deduce that for all sufficiently small $\varepsilon>0$,
\[
        \sum_{m=1}^{2\delta \mathcal L_\varepsilon}
   32^m
        \sum_{q=1}^{m}
        \frac{ \mathcal L_\varepsilon^q}{q!}
        \le
        \exp\left\{
        C\delta \mathcal L_\varepsilon\log(1/\delta)
        \right\}   \le
        (\log\varepsilon^{-1})^{C\delta\log(1/\delta)}.
\]
Therefore applying this to \eqref{eq:entropy-before-final-full}, and recalling \eqref{eq:full-second-moment-expanded-direct} we conclude the proof.
  
\qed

We end this section by recording an unconditional statement.

\subsection{Unconditional second moment estimate}
We need an unconditional version of the decoupling estimate.  The
conditional estimates are effective on upper-tail events of the bulk product
\(G_{\eta;1}\), but later we must also control contributions coming from the
complement of those events.  The unconditional estimate below provides this
control.  Its proof is essentially the same as the conditional small- and
large-length arguments, except that it is simpler since the projected Radon--Nikodym factor is
absent.  

\begin{proposition} [Unconditional decoupling of the partition function] \label{unconditionalprop}
 There exists $C>0$ such that the following holds.
Let \(\eta \in (0,1/10)\) be any constant. 
Then for any small enough constant \(b>0\) (depending only on $\eta$),
\begin{equation}
\mathbb E\left[
\mathbf 1_{\Omega_{\varepsilon,\eta}}
\left|
\frac{\mathcal Z_{\varepsilon,b}}{G_{\eta;1}}
-
G_{\eta;2}
\right|^2 
\right] \le 
   C 
        (\log\varepsilon^{-1})^{C
 \eta }  
\end{equation}
for all 
sufficiently small \(\varepsilon>0\).
\end{proposition}

As in the proof of Proposition~\ref{prop:terminal-block-extension}, 
Proposition~\ref{unconditionalprop} is  immediately obtained by combining the following \emph{unconditional}
small-total-length estimate with the already unconditional large-total-length
estimate from Proposition~\ref{Unconditional second moment: large-total-length}. 

\begin{proposition}[Unconditional second moment: small-total-length]
\label{prop:unconditional-second-moment-small-total-length-full12}
 There exists $C>0$ such that for any $\eta,\delta \in (0,1/10)$ and a small enough constant  \(b>0\),  
\begin{equation}
        \mathbb E\left[
        \mathbf 1_{\Omega_{\varepsilon,\eta}}
        \left|
        \mathcal R_{\eta,\le\delta}
        \right|^2
        \right]
        \le
        C
        (\log\varepsilon^{-1})^{
        C\eta+C\delta\log(1/\delta)} 
\label{eq:unconditional-full-small-second-moment-polylog}
\end{equation}
for all 
sufficiently small \(\varepsilon>0\).
\end{proposition}

\begin{proof}
The proof is the same as the proof of Proposition \ref{prop:second-moment-small-total-length-full}, except that no
Radon--Nikodym derivative appears. For completeness we provide the details but only at places where
the argument differs slightly.

Choose \(\Gamma\) and \(\kappa>0\) so that $ \frac15>\frac{15+\kappa}{6\Gamma}.$  Then  take a small enough constant \(b>0\)    (see \eqref{eq:Gamma-b-full} for the explicit condition). 
Expanding the square, we have an unconditional analog of \eqref{eq:full-second-moment-expanded-direct}:
\begin{equation}
\begin{aligned}
&\mathbb E\left[
\mathbf 1_{\Omega_{\varepsilon,\eta}}
\left|
\mathcal R_{\eta,\le\delta}
\right|^2
\right]                                        \le
\sum_{\substack{
\widetilde I,\widetilde J\in\mathfrak I_{\varepsilon,b}\\
1\le |\widetilde I|,|\widetilde J|\le \delta \mathcal L_\varepsilon}}
\mathbb E\left[
\mathcal Y(\widetilde I,\widetilde J)
\right],
\end{aligned}
\label{eq:unconditional-second-moment-expanded-short}
\end{equation}
where, as in \eqref{404},
\[
        \mathcal Y(\widetilde I,\widetilde J)
        :=
        2^{|\widetilde I|+|\widetilde J|}
        \prod_{I\in\widetilde I}|D_I|
        \cdot \prod_{J\in\widetilde J}|D_J| \cdot 
                \widehat G_{\widetilde I}^{\rm tail}
               \widehat G_{\widetilde J}^{\rm tail}.
\]
For \((\widetilde I,\widetilde J) \in (\mathfrak I_{\varepsilon,b})^2\),  let $ \widetilde K:=\widetilde K(\widetilde I,\widetilde J)$ 
be the corresponding collection of connected components. We split the
\(D\)-factors inside each component into long and short sides as before (see \eqref{441}-\eqref{444}). Then,
\[
       \mathcal Y(\widetilde I,\widetilde J)
        =
        2^{|\widetilde I|+|\widetilde J|}\cdot
        Y_{\widetilde K}^{\rm long} \cdot
        Y_{\widetilde K}^{\rm short} \cdot
                \widehat G_{\widetilde I}^{\rm tail}
               \widehat G_{\widetilde J}^{\rm tail}.    
\]
In the   proof in Proposition \ref{prop:second-moment-small-total-length-full},  we applied Hölder's
inequality to the four factors (see \eqref{eq:holder-direct-full})
\[
        Y_{\widetilde K}^{\rm long},
        \qquad
        Y_{\widetilde K}^{\rm short},
        \qquad
                \widehat G_{\widetilde I}^{\rm tail}
               \widehat G_{\widetilde J}^{\rm tail},
        \qquad
        \mathsf{RN}_{\widetilde K^{\text{bulk}}}^{(\alpha)} .
\]
Here the last factor is absent, or equivalently is equal to \(1\). Hence Hölder's
inequality gives
\begin{equation}
\begin{aligned}
\mathbb E\left[
\mathcal Y(\widetilde I,\widetilde J)
\right]
&\le
2^{|\widetilde I|+|\widetilde J|}
\left\|Y_{\widetilde K}^{\rm long}\right\|_{L^2}
\left\|Y_{\widetilde K}^{\rm short}\right\|_{L^6}
\left\|
        \widehat G_{\widetilde I}^{\rm tail}
       \widehat G_{\widetilde J}^{\rm tail}
\right\|_{L^6}.
\end{aligned}
\label{eq:unconditional-holder-key-difference}
\end{equation}
This is the only difference from the conditional argument. {We could have improved the above bound using $2,4,4$ as the exponents but the gain will be insignificant for our purposes and hence we refrain from doing so.}
The three factors on the RHS are controlled exactly as before (see \eqref{eq:long-L2-full}, \eqref{eq:short-L6-full} and \eqref{eq:tail-product-L6-eta} respectively). 
By the choice of \(\Gamma,\kappa\), and  \(b\) (see the condition \eqref{eq:Gamma-b-full}), we have an unconditional analog of  \eqref{eq:component-bound-before-sum}:
\begin{equation}
        \mathbb E\left[
        \mathcal Y(\widetilde I,\widetilde J)
        \right]
        \le
        C
        (\log\varepsilon^{-1})^{C\eta}
        \prod_{K\in\widetilde K(\widetilde I,\widetilde J)}
        \frac1{N_{\varepsilon,b}-j_{K}} .
\label{eq:unconditional-component-bound-short}
\end{equation}
It remains to sum   over
\((\widetilde I,\widetilde J)\), which can be done exactly as before (see the discussion below \eqref{eq:component-bound-before-sum}).  
\end{proof}

\begin{remark} \label{remarkvary3}
While {the main propositions in this section are stated for fixed \(b\) and a
fixed coupling parameter \(\vartheta\),} the same conclusions remain valid for  \(\varepsilon\)-dependent parameters $b=b_\varepsilon$ and $\vartheta=\vartheta_\varepsilon$ 
provided that, for   
some \(\vartheta^\star\in\mathbb R\) and small enough \(b^\star>0\),
\begin{align} \label{general}
        b_\varepsilon\to b^\star,
        \qquad
        \vartheta_\varepsilon\to\vartheta^\star
        \qquad
        \text{as }\varepsilon\rightarrow0.
\end{align}
Indeed, the moment bounds in Lemmas~\ref{lem:bounded-consecutive-fourth} and
\ref{lem:short-fourth-product} do not rely on \(b\) being fixed independently
of \(\varepsilon\); they only require \(b\) to lie in the admissible small-\(b\)
range.  Moreover, these bounds hold uniformly whenever the coupling parameter \(\vartheta\) 
is bounded above.  This follows from the monotonicity of the SHF moments in
\(\vartheta\); see Lemma~\ref{lem:moment-monotonicity-theta}. Moreover, the key input \eqref{rnspecial32} remains valid under the  assumption \eqref{general} (we elaborate on this later in Remark \ref{remarkvary5} when we state and prove Proposition \ref{thm:direct-Lp-no-split-corrected}, a general version of Proposition \ref{prelimcond}).  Thus the proofs apply without any further
modification in this varying-parameter setting.

\end{remark}

\section{Level sets of the SHF and mass concentration}\label{masconc4567}
In this section we prove Proposition \ref{prop:Z-mass-concentration-from-Geta1-Geta2}. We start by recalling the statement.
\begin{proposition}
\label{prop:Z-mass-concentration-from-Geta1-Geta234} 
Let \(\zeta\in(0,1/10)\) be an arbitrary constant. 
Then there exists $\upsilon_\zeta>0$ such that the following holds 
 for  any small  constant $b>0$: For all  sufficiently small $\varepsilon>0,$
\begin{equation}
        \mathbb E\left[
        \mathcal Z_{\varepsilon,b}\mathbf 1_{\{\mathcal Z_{\varepsilon,b}<N_{\varepsilon,b}^{1/2-\zeta}\}}
        \right]
        +
        \mathbb E\left[
        \mathcal Z_{\varepsilon,b}\mathbf 1_{\{\mathcal Z_{\varepsilon,b}>N_{\varepsilon,b}^{1/2+\zeta}\}}
        \right] \le (\log \varepsilon^{-1})^{-\upsilon_\zeta}. 
\end{equation} 
\end{proposition}

\nin
As highlighted in Section \ref{iop}, the comparison between  $\mathcal Z_{\varepsilon,b}$
and the decoupled product $G_{\eta; 1}$ developed in the previous sections through Propositions \ref
{prop:terminal-block-extension} and \ref{unconditionalprop} allows us to transfer estimates about the latter to the former.  
Thus the following  corresponding statement for $G_{\eta;1}$ whose proof will  be deferred to Section \ref{sec5} will serve as our key input. 
\begin{lemma} \label{gconcen}
There exists
\(\widetilde\nu_b \ge0\)  with \(\widetilde\nu_b\to0\) as \(b\downarrow0\)  such that the following holds.
For \(\eta \in (0,1/10)\) and \(\gamma\in(0,1/10)\), define the event
\[
        \mathcal V_{\eta,\gamma}
        :=
        \left\{
         N_{\varepsilon,b}^{1/2-\eta/2-\gamma}
        \le
        G_{\eta;1}
        \le
         N_{\varepsilon,b}^{1/2-\eta/2+\gamma}
        \right\}.
\]
Then for any constant  \(\xi>0\), the following holds for any small enough constant $b>0$. For all sufficiently small $\varepsilon>0,$
\begin{equation}
        \mathbb E\left[
        G_{\eta;1}\mathbf 1_{\mathcal V_{\eta,\gamma}^c}
        \right]
        \le
         N_{\varepsilon,b}^{-\frac{\gamma^2}{2(1-\eta)}+\xi}.
\label{eq:size-biased-Geta1-concentration}
\end{equation}
\end{lemma}

Given the above, we now embark on proving Proposition \ref{prop:Z-mass-concentration-from-Geta1-Geta234}. The idea is simply that Propositions \ref
{prop:terminal-block-extension} and \ref{unconditionalprop} ensure that the part of the space where $\frac{\cZ_{\e,b}}{G_{\eta,1}}$ is much larger than $1$ (quantified as a small power of $\log \e^{-1}$) can be ignored. On the other hand since both ${G_{\eta,1}}$ and ${\cZ_{\e,b}}$ have expectation one, the ratio also cannot be much smaller than $1$ on the relevant part of the probability space. This allows to conclude that on the latter, $G_{\eta;1}$ and $\cZ_{\e,b}$ are comparable and hence the exact $\cZ_{\e,b}$ counterpart of \eqref{eq:size-biased-Geta1-concentration} holds which is indeed the statement of Proposition \ref{prop:Z-mass-concentration-from-Geta1-Geta234}. 

\begin{proof}[Proof of Proposition \ref{prop:Z-mass-concentration-from-Geta1-Geta234}]
Recalling the event $\Omega_{\e,\eta}$ from \eqref{nice}, note that  by  Lemma \ref{goodevent} for any {small enough constant} $\eta>0$, and for all sufficiently small $\varepsilon>0,$ we have $\mathbb P( \Omega_{\varepsilon,\eta}^c)\le (\log \varepsilon^{-1})^{-10}.$
Further, by \eqref{adaptmoment} with $h=2$, we have $\mathbb  E[\mathcal Z_{\varepsilon,b}^2]\le (\log \varepsilon^{-1})^{1.1}.$
 Combining the above with Cauchy-Schwarz inequality we get 
\begin{align*}
     \mathbb E\left[\mathbf 1_{\Omega_{\varepsilon,\eta}^c}  
       \mathcal  Z _{\varepsilon,b}
        \right] \le   \mathbb P( \Omega_{\varepsilon,\eta}^c)  ^{1/2}  (\mathbb  E[\mathcal Z_{\varepsilon,b}^2])^{1/2} \le  (\log \varepsilon^{-1})^{-5} \cdot  ((\log \varepsilon^{-1})^{1.1})^{1/2} \le (\log \varepsilon^{-1})^{-4}.
        \end{align*} 
        Hence
it suffices to prove that  for   small enough constants $b>0$ and $\eta>0,$
\begin{equation}
        \mathbb E\left[\mathbf 1_{\Omega_{\varepsilon,\eta}}  
        \mathcal Z_{\varepsilon,b}\mathbf 1_{\{\mathcal Z_{\varepsilon,b}<N_{\varepsilon,b}^{1/2-\zeta}\}}
        \right]
        +
        \mathbb E\left[\mathbf 1_{\Omega_{\varepsilon,\eta}}  
        \mathcal Z_{\varepsilon,b}\mathbf 1_{\{\mathcal Z_{\varepsilon,b}>N_{\varepsilon,b}^{1/2+\zeta}\}}
        \right]  \le (\log \varepsilon^{-1})^{-\upsilon_\zeta}. 
\end{equation}
The proof consists of several steps. The role of the parameter \(\eta\) is to separate the bulk and terminal
scales.  We first prove estimates with a fixed small \(\eta \in (0,1)\), and at the end
we choose \(\eta\) sufficiently small dictated by the target exponent
\(\zeta\).  \\

\nin
\textbf{Step 1.(Setup)} Define the ratio
\begin{equation}\label{ratio12}
        R_\eta:=\frac{\mathcal Z_{\varepsilon,b}}{G_{\eta;1}}.
\end{equation}
Then by Proposition \ref{unconditionalprop} and \eqref{eq:G-tail-product-moment-consequence}, the following   bounds hold for any  small enough constant $b>0$ (depending on $\eta \in (0,1)$): For all sufficiently small $\varepsilon>0,$
\begin{equation}
        \mathbb E\left[
      \mathbf 1_{\Omega_{\varepsilon,\eta}}  \left(R_\eta-G_{\eta;2}\right)^2
        \right]
        \le
        CN_{\varepsilon,b}^{C\eta},
        \qquad
        \mathbb E[G_{\eta;2}^2]\le CN_{\varepsilon,b}^{C\eta},
\label{eq:unconditional-Reta-Geta2-assumption}
\end{equation}
{where we used $\eta<1$ and the fact that $ \log \varepsilon^{-1} \le C_b N_{\varepsilon,b}$ for some $C_b>0$, to  absorb all
multiplicative constants into a single constant \(C\).} Using $ R_\eta^2
        \le
        2(R_\eta-G_{\eta;2})^2+2G_{\eta;2}^2,$  we have 
\begin{equation}
        \mathbb E \left[\mathbf 1_{\Omega_{\varepsilon,\eta}} R_\eta^2 \right]
        \le
        C N_{\varepsilon,b}^{C\eta} .
\label{eq:Reta-unconditional-second-moment}
\end{equation}
  Moreover by Proposition \ref{prop:terminal-block-extension}, for every \(\alpha>0\), for  any  small enough constant $b>0$ (depending on $\alpha$ and  $\eta$),
\begin{equation}
        \mathbb E\left[ \mathbf 1_{\Omega_{\varepsilon,\eta}} 
        \left(R_\eta-G_{\eta;2}\right)^2
        \,\middle|\,
        G_{\eta;1} \ge N_{\varepsilon,b}^{-\frac{1-\eta}{2}+\alpha}
        \right]
        \le
        CN_{\varepsilon,b}^{C\eta} .
\label{eq:conditional-Reta-Geta2-assumption}
\end{equation} 
Using again $ R_\eta^2
        \le
        2(R_\eta-G_{\eta;2})^2+2G_{\eta;2}^2$ and noting that  \(G_{\eta;2}\) is independent of \(G_{\eta;1}\), we get similarly, 
\begin{equation}
        \mathbb E\left[ \mathbf 1_{\Omega_{\varepsilon,\eta}}
        R_\eta^2
        \,\middle|\,
        G_{\eta;1} \ge N_{\varepsilon,b}^{-\frac{1-\eta}{2}+\alpha}
        \right]
        \le
        C N_{\varepsilon,b}^{C\eta} .
\label{eq:Reta-conditional-second-moment}
\end{equation}
We may assume that the constant $C>0$ appearing in \eqref{eq:unconditional-Reta-Geta2-assumption}-\eqref{eq:Reta-conditional-second-moment} is the same.
We next choose some of the parameters as follows.  First choose $0<\gamma<\zeta,$  and then choose \(\eta\in(0,1/(10C))\) sufficiently small so that there exists
\(\beta>0\) satisfying
\begin{equation}
        10C\eta
        <
        \beta
        <
        \min\left\{
        \zeta-\frac{\eta}{2}-\gamma,\,
        \frac{\gamma^2}{4(1-\eta)}
        \right\}.
\label{eq:beta-choice-Geta}
\end{equation} 

Towards proving that  $\cZ_{\e,b}$ and $G_{\eta; 1}$ are comparable, we first prove that the contribution of $\cZ_{\e,b}$ when $R_{\eta}$ from \eqref{ratio12} is larger than a small power of $\log \e^{-1}$, given by $\beta,$ is negligible. \\

\nin
\textbf{Step 2.} More precisely we have the following.
\begin{equation}
  \mathbb E\left[ \mathbf 1_{\Omega_{\varepsilon,\eta}}
        \mathcal Z_{\varepsilon,b} \mathbf 1_{\{R_\eta>N_{\varepsilon,b}^\beta\}}
        \right] \le (\log \varepsilon^{-1})^{-\upsilon_\zeta}.
\label{eq:upper-ratio-tail-target-Geta}
\end{equation}
Let \(a>0\) be fixed so small that
\begin{equation}
        a+ \frac{C+1}{2} \eta<\frac14.
\label{eq:a-choice-Geta}
\end{equation}
Let \(B>0\) be a large constant  which will be chosen later.  We now split the LHS in \eqref{eq:upper-ratio-tail-target-Geta} according to the value of the bulk product \(G_{\eta;1}\).  Namely, \begin{equation}
\begin{aligned}
\mathbb E\left[\mathbf 1_{\Omega_{\varepsilon,\eta}}
\mathcal Z_{\varepsilon,b} \mathbf 1_{\{R_\eta>N_{\varepsilon,b}^\beta\}}
\right]
&\le
\mathbb E\left[\mathbf 1_{\Omega_{\varepsilon,\eta}}
\mathcal Z_{\varepsilon,b} \mathbf 1_{\{G_{\eta;1}<N_{\varepsilon,b}^{-\frac{1-\eta}{2}+a}\}}
\right]                                                   \\
&\quad+
\mathbb E\left[\mathbf 1_{\Omega_{\varepsilon,\eta}}
\mathcal Z_{\varepsilon,b} \mathbf 1_{\{R_\eta>N_{\varepsilon,b}^\beta\}}
\mathbf 1_{\{
N_{\varepsilon,b}^{-\frac{1-\eta}{2}+a}
\le G_{\eta;1}\le N_{\varepsilon,b}^B
\}}
\right]                                                   \\
&\quad+
\mathbb E\left[\mathbf 1_{\Omega_{\varepsilon,\eta}}
\mathcal Z_{\varepsilon,b} \mathbf 1_{\{G_{\eta;1}>N_{\varepsilon,b}^B\}}
\right].
\end{aligned}
\label{eq:ratio-tail-three-region-decomp-Geta}
\end{equation}
\emph{Sub-step 2-1.}
 Using \(\mathcal Z_{\varepsilon,b}=R_\eta G_{\eta;1}\), for any  small enough constant $b>0$,
\begin{align} \label{eq:low-Geta1-negligible}
&\mathbb E\left[\mathbf 1_{\Omega_{\varepsilon,\eta}}
\mathcal Z_{\varepsilon,b}    \mathbf 1_{\{G_{\eta;1}<N_{\varepsilon,b}^{-\frac{1-\eta}{2}+a}\}}
\right]
 =
\mathbb E\left[\mathbf 1_{\Omega_{\varepsilon,\eta}}
R_\eta G_{\eta;1}
\mathbf 1_{\{G_{\eta;1}<N_{\varepsilon,b}^{-\frac{1-\eta}{2}+a}\}}
\right]                                                   \nonumber \\
&\qquad \qquad \qquad \le
N_{\varepsilon,b}^{-\frac{1-\eta}{2}+a}\mathbb E [\mathbf 1_{\Omega_{\varepsilon,\eta}}R_\eta]    \overset{\eqref{eq:Reta-unconditional-second-moment}}{\le} 
C N_{\varepsilon,b}^{-1/2+a+(C+1)\eta/2} \overset{\eqref{eq:a-choice-Geta}}{\le}  (\log \varepsilon^{-1})^{-1/10}
\end{align}
{for sufficiently small $\varepsilon.$ Note that we used again $N_{\varepsilon,b} \le \log \varepsilon^{-1}$ for small $b>0$, in the last inequality.}\\

\nin
\emph{Sub-step 2-2.} In this step, we split the term $\mathbb E\left[\mathbf 1_{\Omega_{\varepsilon,\eta}}
\mathcal Z_{\varepsilon,b} \mathbf 1_{\{R_\eta>N_{\varepsilon,b}^\beta\}}
\mathbf 1_{\{
N_{\varepsilon,b}^{-\frac{1-\eta}{2}+a}
\le G_{\eta;1}\le N_{\varepsilon,b}^B
\}}
\right]$ further based on the value of $G_{\eta;1}.$ 
 Setting
$        h:=\frac{\beta}{2}$, 
we take finitely many exponents
\[
        a_j:=a+jh,
        \qquad
        j=0,1,\ldots,J,
\]
where \(J\) is chosen so that $-\frac{1-\eta}{2}+a_J < B \le -\frac{1-\eta}{2}+a_{J+1}$. 
For a fixed \(j = 0,1,\ldots,J,\) using $\mathcal Z_{\varepsilon,b} = R_\eta G_{\eta;1},$ for any  small enough constant $b>0$,
\begin{equation}
\begin{aligned}
&\mathbb E\left[\mathbf 1_{\Omega_{\varepsilon,\eta}}
\mathcal Z_{\varepsilon,b} \mathbf 1_{\{R_\eta>N_{\varepsilon,b}^\beta\}}
\mathbf 1_{\{
N_{\varepsilon,b}^{-\frac{1-\eta}{2}+a_j}
\le
G_{\eta;1}
\le
N_{\varepsilon,b}^{-\frac{1-\eta}{2}+a_j+h}
\}}
\right]         \\
&\qquad\le
N_{\varepsilon,b}^{-\frac{1-\eta}{2}+a_j+h}
\mathbb E\left[\mathbf 1_{\Omega_{\varepsilon,\eta}}
R_\eta\mathbf 1_{\{R_\eta>N_{\varepsilon,b}^\beta\}}
\mathbf 1_{\{
G_{\eta;1}>N_{\varepsilon,b}^{-\frac{1-\eta}{2}+a_j}
\}} 
\right]      \\
&\qquad\le N_{\varepsilon,b}^{-\frac{1-\eta}{2}+a_j+h}\cdot 
N_{\varepsilon,b}^{-\beta}
\mathbb E\left[\mathbf 1_{\Omega_{\varepsilon,\eta}}
R_\eta^2
\mathbf 1_{\{
G_{\eta;1}>N_{\varepsilon,b}^{-\frac{1-\eta}{2}+a_j}
\}}
\right]                                   \\ 
&\qquad \overset{\eqref{eq:Reta-conditional-second-moment} }{\le} 
N_{\varepsilon,b}^{-\frac{1-\eta}{2}+a_j+h}\cdot  N_{\varepsilon,b}^{-\beta+C\eta}
\mathbb P\left(
G_{\eta;1}>N_{\varepsilon,b}^{-\frac{1-\eta}{2}+a_j}
\right) \\
&\qquad \le   N_{\varepsilon,b}^{-\frac{1-\eta}{2}+a_j+h}\cdot  N_{\varepsilon,b}^{-\beta+C\eta} \cdot  N_{\varepsilon,b}^{ \frac{1-\eta}{2}-a_j} = 
C N_{\varepsilon,b}^{-\beta/2+ C\eta},
\end{aligned}
\label{eq:middle-bin-first-bound-Geta}
\end{equation}  
where we used  Markov's inequality and \(\mathbb E G_{\eta;1}=1\)  to bound $\mathbb P\left(
G_{\eta;1}>N_{\varepsilon,b}^{-\frac{1-\eta}{2}+a_j}
\right)$. Finally, the last equality uses $h=\beta/2.$
Note that we could have used a sharper bound on the latter using the precise large deviation tail estimate from Proposition \ref{largedeviation45}, but the crude bound above deduced from Markov's inequality will suffice.

Since \(10C\eta<\beta\) (see  the condition \eqref{eq:beta-choice-Geta})  and  there are only finitely many bins, summing \eqref{eq:middle-bin-first-bound-Geta} over the bins, we deduce that for any  small enough constant $b>0$,
\begin{equation}
\mathbb E\left[\mathbf 1_{\Omega_{\varepsilon,\eta}}
\mathcal Z_{\varepsilon,b} \mathbf 1_{\{R_\eta>N_{\varepsilon,b}^\beta\}}
\mathbf 1_{\{
N_{\varepsilon,b}^{-\frac{1-\eta}{2}+a}
\le
G_{\eta;1}
\le
N_{\varepsilon,b}^B
\}}
\right] \le 
 N_{\varepsilon,b}^{-\beta/10}.
\label{eq:middle-region-negligible-Geta}
\end{equation}
\emph{Sub-step 2-3.}
For the high-\(G_{\eta;1}\) term, by H\"older inequality,\begin{equation}
\begin{aligned}
\mathbb E\left[\mathbf 1_{\Omega_{\varepsilon,\eta}}
\mathcal Z_{\varepsilon,b} \mathbf 1_{\{G_{\eta;1}>N_{\varepsilon,b}^B\}}
\right]
&=
\mathbb E\left[\mathbf 1_{\Omega_{\varepsilon,\eta}}
R_\eta G_{\eta;1}
\mathbf 1_{\{G_{\eta;1}>N_{\varepsilon,b}^B\}}
\right]                                                   \\
&\le
\left(\mathbb E  [\mathbf 1_{\Omega_{\varepsilon,\eta}}R_\eta^2]\right)^{1/2}
\left(\mathbb E G_{\eta;1}^4\right)^{1/4}
\left(
\mathbb P(G_{\eta;1}>N_{\varepsilon,b}^B)
\right)^{1/4} \\
&\le  C  N_{\varepsilon,b}^{C\eta/2} N_{\varepsilon,b}^{C} N_{\varepsilon,b}^{-B/4},
\end{aligned}
\label{eq:high-Geta1-holder-244}
\end{equation}  
where  we used  
\eqref{eq:Reta-unconditional-second-moment}, \eqref{eq:G-full-product-moment-consequence}  and Markov's inequality (note that $\mathbb E G_{\eta;1}=1$) respectively. Hence by taking $B>0$ large enough, the above quantity is at most  $(\log \varepsilon^{-1})^{-10}.$

Applying the above estimates to
\eqref{eq:ratio-tail-three-region-decomp-Geta},   we deduce that there exists $\upsilon_\zeta>0$ such that for any  small enough constant $b>0$,
\begin{equation}
              \mathbb E\left[\mathbf 1_{\Omega_{\varepsilon,\eta}}
        \mathcal Z_{\varepsilon,b} \mathbf 1_{\{R_\eta>N_{\varepsilon,b}^\beta\}}
        \right]\le (\log \varepsilon^{-1})^{-\upsilon_\zeta}.
\label{eq:upper-ratio-tail-proved-Geta}
\end{equation}
\textbf{Step 3.}
The analysis of the first term in \eqref{eq:ratio-tail-three-region-decomp-Geta} is immediate:
\begin{equation}
\begin{aligned}
\mathbb E\left[\mathbf 1_{\Omega_{\varepsilon,\eta}}
\mathcal Z_{\varepsilon,b} \mathbf 1_{\{R_{\eta}<N_{\varepsilon,b}^{-\beta}\}}
\right]
&=
\mathbb E\left[\mathbf 1_{\Omega_{\varepsilon,\eta}}
R_\eta G_{\eta;1}
\mathbf 1_{\{R_\eta<N_{\varepsilon,b}^{-\beta}\}}
\right]                                                   \\
&\le
N_{\varepsilon,b}^{-\beta}\mathbb E G_{\eta;1}
=
N_{\varepsilon,b}^{-\beta} \le (\log \varepsilon^{-1})^{-\beta/2}.
\end{aligned}
\label{eq:lower-ratio-tail-Geta}
\end{equation}
\textbf{Conclusion.}
Define the event
\[
        \mathcal V_{\eta,\gamma} :
        =
        \left\{
        N_{\varepsilon,b}^{1/2-\eta/2-\gamma}
        \le
        G_{\eta;1}
        \le
        N_{\varepsilon,b}^{1/2-\eta/2+\gamma}
        \right\}.
\]
Recalling the condition \eqref{eq:beta-choice-Geta},  we have the implication
\[
        \left\{
        N_{\varepsilon,b}^{-\beta}
        \le
        \frac{\mathcal Z_{\varepsilon,b}}{G_{\eta;1}}
        \le
        N_{\varepsilon,b}^\beta
        \right\}
        \cap \mathcal V_{\eta,\gamma}\Rightarrow      \left\{N_{\varepsilon,b}^{1/2-\zeta}
      \le
        \mathcal Z_{\varepsilon,b}
        \le
        N_{\varepsilon,b}^{1/2+\zeta} \right\}.
\]   
Therefore,
\begin{align*}
&\quad\quad\quad \mathbb E\left[\mathbf 1_{\Omega_{\varepsilon,\eta}}
\mathcal Z_{\varepsilon,b} \mathbf 1_{\{\mathcal Z_{\varepsilon,b}<N_{\varepsilon,b}^{1/2-\zeta}\}}
\right]
+
\mathbb E\left[\mathbf 1_{\Omega_{\varepsilon,\eta}}
\mathcal Z_{\varepsilon,b} \mathbf 1_{\{\mathcal Z_{\varepsilon,b}>N_{\varepsilon,b}^{1/2+\zeta}\}}
\right]                                                   \\
&\le
\mathbb E\left[\mathbf 1_{\Omega_{\varepsilon,\eta}}
\mathcal Z_{\varepsilon,b} \mathbf 1_{\{\mathcal Z_{\varepsilon,b}/G_{\eta;1}>N_{\varepsilon,b}^\beta\}}
\right]
+
\mathbb E\left[\mathbf 1_{\Omega_{\varepsilon,\eta}}
\mathcal Z_{\varepsilon,b} \mathbf 1_{\{\mathcal Z_{\varepsilon,b}/G_{\eta;1}<N_{\varepsilon,b}^{-\beta}\}}
\right]     +
\mathbb E\left[
\mathcal Z_{\varepsilon,b} \mathbf 1_{\mathcal V_{\eta,\gamma}^c}
\mathbf 1_{\{N_{\varepsilon,b}^{-\beta}\le \mathcal Z_{\varepsilon,b}/G_{\eta;1}\le N_{\varepsilon,b}^\beta\}}
\right]\\
&\le \mathbb E\left[\mathbf 1_{\Omega_{\varepsilon,\eta}}
\mathcal Z_{\varepsilon,b} \mathbf 1_{\{\mathcal Z_{\varepsilon,b}/G_{\eta;1}>N_{\varepsilon,b}^\beta\}}
\right]
+
\mathbb E\left[\mathbf 1_{\Omega_{\varepsilon,\eta}}
\mathcal Z_{\varepsilon,b} \mathbf 1_{\{\mathcal Z_{\varepsilon,b}/G_{\eta;1}<N_{\varepsilon,b}^{-\beta}\}}
\right]     +
\mathbb E\left[
\mathcal Z_{\varepsilon,b} \mathbf 1_{\mathcal V_{\eta,\gamma}^c}
\mathbf 1_{\{\mathcal Z_{\varepsilon,b}/G_{\eta;1}\le N_{\varepsilon,b}^\beta\}}
\right].
\end{align*} 
The first and second terms are controlled by
\eqref{eq:upper-ratio-tail-proved-Geta}  and   \eqref{eq:lower-ratio-tail-Geta} respectively.  For the third term, 
\begin{equation*}
\mathbb E\left[
\mathcal Z_{\varepsilon,b} \mathbf 1_{\mathcal V_{\eta,\gamma}^c}
\mathbf 1_{\{\mathcal Z_{\varepsilon,b}/G_{\eta;1}\le N_{\varepsilon,b}^\beta\}}
\right] \le N_{\varepsilon,b}^\beta
\mathbb E\left[
G_{\eta;1}\mathbf 1_{\mathcal V_{\eta,\gamma}^c}
\right].
\end{equation*}

The RHS is now bounded by  Lemma  \ref{gconcen} which states the following.
For any constant  \(\xi>0\), for any small enough constant $b>0$, for all sufficiently small $\varepsilon>0,$
\begin{equation}
        \mathbb E\left[
        G_{\eta;1}\mathbf 1_{\mathcal V_{\eta,\gamma}^c}
        \right]
        \le
         N_{\varepsilon,b}^{-\frac{\gamma^2}{2(1-\eta)}+\xi}.
\label{eq:size-biased-Geta1-concentration00}
\end{equation}
Thus,
\begin{equation}
\begin{aligned}
\mathbb E\left[
\mathcal Z_{\varepsilon,b} \mathbf 1_{\mathcal V_{\eta,\gamma}^c}
\mathbf 1_{\{\mathcal Z_{\varepsilon,b}/G_{\eta;1}\le N_{\varepsilon,b}^\beta\}}
\right]
&\le
N_{\varepsilon,b}^\beta
\mathbb E\left[
G_{\eta;1}\mathbf 1_{\mathcal V_{\eta,\gamma}^c}
\right]     \le
N_{\varepsilon,b}^{
\beta-\frac{\gamma^2}{2(1-\eta)}+\xi
}.
\end{aligned}
\label{eq:window-error-Geta}
\end{equation}
By the condition \eqref{eq:beta-choice-Geta}, taking $\xi>0$ small, the exponent above is strictly negative for any small enough constant $b>0$.  Putting things together finishes the proof.
\end{proof}

\begin{remark} 
\label{allow}  {The conclusion of Proposition~\ref{prop:Z-mass-concentration-from-Geta1-Geta234} also} 
holds under deterministic \(\varepsilon\)-dependent choices of the
parameters.  More precisely, it remains
valid if \(b=b_\varepsilon\) and \(\vartheta=\vartheta_\varepsilon\) satisfy that, for   
some \(\vartheta^\star\in\mathbb R\) and small enough \(b^\star>0\),
\[
        b_\varepsilon\to b^\star,
        \qquad
        \vartheta_\varepsilon\to\vartheta^\star
        \qquad
        \text{as }\varepsilon\rightarrow0.
\]In particular, the key inputs used---namely the propositions in
Section~\ref{sec3} and Corollary~\ref{gconcen}---remain valid in this
varying-parameter setting; see Remarks~\ref{remarkvary3} and
\ref{remarkvary5} for details.
\end{remark}

Proposition \ref{prop:Z-mass-concentration-from-Geta1-Geta234} 
was about $\cZ_{\e,b}$. The following corollary transfers the result to $\cZ_{\e}$ using an exact scaling identity. 
\begin{corollary} \label{corcorcor}
Let \(\zeta\in(0,1/10)\) be an arbitrary constant. 
Then there exists $\upsilon_\zeta>0$ such that for all  sufficiently small $\varepsilon>0,$
\begin{equation}
        \mathbb E\left[
        \mathcal Z_{\varepsilon} \mathbf 1_{ \mathcal Z_{\varepsilon} < (\log \varepsilon^{-1})^{1/2-\zeta}}
        \right]
        +
        \mathbb E\left[
        \mathcal Z_{\varepsilon} \mathbf 1_{ \mathcal Z_{\varepsilon} > (\log \varepsilon^{-1})^{1/2+\zeta}} 
        \right] \le (\log \varepsilon^{-1})^{-\upsilon_\zeta}.
\label{eq:Z-mass-concentration-final-Geta}
\end{equation}
\end{corollary}

\begin{proof} 
Let  \(b^\star>0\) be a small enough constant.
Given the  smoothing scale \(\varepsilon>0\), define
\begin{align} \label{barvar}
        \bar\varepsilon
        :=
        \frac{\varepsilon}{\sqrt{1+\varepsilon^2}},
\end{align}
and set
\begin{equation}
        N_\varepsilon^\star
        :=
        1+
        \left\lfloor
        \frac{\log\bar\varepsilon^{-1}}{\log (b^\star)^{-1}}
        \right\rfloor ,\qquad  b_\varepsilon
        :=
        \bar\varepsilon^{1/(N_\varepsilon^\star-1)}.
\label{eq:Nstar-definition-endpoint-excluded}
\end{equation}   Then as $\varepsilon \rightarrow 0,$
\begin{align}\label{eq:beps-compact-neighborhood-endpoint-excluded}
        b_\varepsilon
        =
        \exp\left\{
        -\frac{\log\bar\varepsilon^{-1}}{N_\varepsilon^\star-1}
        \right\}
        \longrightarrow
        e^{-\log (b^\star)^{-1}}
        =
        b^\star.
\end{align} 
Now
define the auxiliary microscopic scale
\begin{equation}
        \widehat\varepsilon
        :=
        b_\varepsilon\bar\varepsilon  \overset{\eqref{eq:Nstar-definition-endpoint-excluded}}{=} b_\varepsilon^{N_\varepsilon^\star}. 
\label{eq:epshat-definition-endpoint-excluded}
\end{equation} 
Hence, if we define the exponential time scale
\[
        \widehat t_i
        :=
        \widehat\varepsilon^2 b_\varepsilon^{-2i},
        \qquad
        i=0,1,\ldots,N_\varepsilon^\star,
\]
then $ \widehat t_0=\widehat\varepsilon^2$ and $  \widehat t_{N_\varepsilon^\star-1}=b_\varepsilon^2.$ 
Set
\begin{equation}
        \Delta_\varepsilon
        :=
        \widehat t_{N_\varepsilon^\star-1}
        -
        \widehat\varepsilon^2
        =
        b_\varepsilon^2-\widehat\varepsilon^2  =
        b_\varepsilon^2(1-\bar\varepsilon^2).
\label{eq:lambda-eps-endpoint-excluded}
\end{equation} 
Consequently,
\begin{equation}
        \frac{\widehat\varepsilon^2}{\Delta_\varepsilon}
        =
        \frac{b_\varepsilon^2\bar\varepsilon^2}
        {b_\varepsilon^2(1-\bar\varepsilon^2)}
        =
        \frac{\bar\varepsilon^2}{1-\bar\varepsilon^2}
         \overset{\eqref{barvar}}{=} 
        \varepsilon^2.
\label{eq:epshat-over-lambda-equals-eps}
\end{equation}
Define the coupling parameter
\begin{equation}
        \widehat \vartheta_\varepsilon
        :=
        \vartheta-\log\Delta_\varepsilon \overset{\eqref{eq:lambda-eps-endpoint-excluded}}{=}  
        \vartheta-2\log b_\varepsilon+\log(1+\varepsilon^2) \overset{\varepsilon \rightarrow 0}{\longrightarrow} \vartheta - 2\log b^\star.
\label{eq:vartheta-prime-endpoint-excluded}
\end{equation}  
Now define the auxiliary partition function
\begin{equation}
        \mathcal Z_{\widehat\varepsilon,b_\varepsilon}^{\rm aux}
        :=
        p(\widehat\varepsilon^2)
        \blacktriangleleft
        \mathscr Z^{\widehat \vartheta_\varepsilon}_
        {\widehat\varepsilon^2,\widehat t_{N_\varepsilon^\star-1}}
        \blacktriangleright1 .
\label{eq:Z-aux-definition-endpoint-excluded}
\end{equation}  
{By the scaling relation (see \eqref{eq:SHF-scaling-measure-current})} and time-translation of the  SHF, {$ \mathcal Z_{\widehat\varepsilon,b_\varepsilon}^{\rm aux}$ is equal in distribution to
\[
        p\left(\frac{\widehat\varepsilon^2}{\Delta_\varepsilon}\right)
        \blacktriangleleft
        \mathscr Z^{\widehat\vartheta_\varepsilon+\log\Delta_\varepsilon}_{0,1}
        \blacktriangleright1.
\]
The identities \eqref{eq:epshat-over-lambda-equals-eps} and
\eqref{eq:vartheta-prime-endpoint-excluded} identify the smoothing scale and
coupling parameter above with \(\varepsilon^2\) and \(\vartheta\), respectively.} Hence, we obtain that 
\begin{align} \label{eq:aux-truncated-equals-original-law}
           p(\widehat\varepsilon^2)
        \blacktriangleleft
        \mathscr Z^{\widehat \vartheta_\varepsilon}_
        {\widehat\varepsilon^2,\widehat t_{N_\varepsilon^\star-1}}
        \blacktriangleright1  
        \stackrel{\mathrm d}{=}
        p(\varepsilon^2)
        \blacktriangleleft
        \mathscr Z^\vartheta_{0,1}
        \blacktriangleright1 = \mathcal{Z}_\varepsilon.
\end{align}   
Now we apply Proposition~\ref{prop:Z-mass-concentration-from-Geta1-Geta2} to the auxiliary partition function \(\mathcal Z_{\widehat\varepsilon,b_\varepsilon}^{\rm aux}\).
Although the parameters \(b_\varepsilon\) and \(\widehat \vartheta_\varepsilon\) are not constants, this causes no difficulty.  Indeed, by
\eqref{eq:beps-compact-neighborhood-endpoint-excluded} and
\eqref{eq:vartheta-prime-endpoint-excluded},   
Proposition~\ref{prop:Z-mass-concentration-from-Geta1-Geta2} applies to
\(\mathcal Z_{\widehat\varepsilon,b_\varepsilon}^{\rm aux}\) as well; see
Remark~\ref{allow}.  Thus there exists \(\upsilon_{\zeta/2}>0\) such that 
the following holds for  any small  constant $b^\star>0$:
For all sufficiently small \(\widehat \varepsilon>0\),
\begin{equation}
\begin{aligned}
        &\mathbb E\left[
        \mathcal Z_{\widehat\varepsilon,b_\varepsilon}^{\rm aux}
        \mathbf 1_{\{
        \mathcal Z_{\widehat\varepsilon,b_\varepsilon}^{\rm aux}
        <
        (N_\varepsilon^\star)^{1/2-\zeta/2}
        \}}
        \right]  +
        \mathbb E\left[
        \mathcal Z_{\widehat\varepsilon,b_\varepsilon}^{\rm aux}
        \mathbf 1_{\{
        \mathcal Z_{\widehat\varepsilon,b_\varepsilon}^{\rm aux}
        >
        (N_\varepsilon^\star)^{1/2+\zeta/2}
        \}}
        \right]
        \le
        (\log\widehat\varepsilon^{-1})^{-\upsilon_{\zeta/2}}.
\end{aligned}
\label{eq:aux-size-biased-concentration-endpoint-excluded}
\end{equation}
Using  
\eqref{eq:aux-truncated-equals-original-law} and noting that $ (\log\varepsilon^{-1})^{1/2-\zeta}
        \le
      (N_\varepsilon^\star)^{1/2-\zeta/2} $    for all sufficiently small
\(\varepsilon>0\),
\[
\begin{aligned}
        &\mathbb E\left[
        \mathcal Z_{\varepsilon} 
        \mathbf 1_{\{
        \mathcal Z_{\varepsilon} 
        <
        (\log\varepsilon^{-1})^{1/2-\zeta}
        \}}
        \right]  =
        \mathbb E\left[
        \mathcal Z_{\widehat\varepsilon,b_\varepsilon}^{\rm aux} 
        \mathbf 1_{\{
        \mathcal Z_{\widehat\varepsilon,b_\varepsilon}^{\rm aux} 
        <
        (\log\varepsilon^{-1})^{1/2-\zeta}
        \}}
        \right]      \le
        \mathbb E\left[
        \mathcal Z_{\widehat\varepsilon,b_\varepsilon}^{\rm aux} 
        \mathbf 1_{\{
        \mathcal Z_{\widehat\varepsilon,b_\varepsilon}^{\rm aux} 
        <
       (N_\varepsilon^\star)^{1/2-\zeta/2}
        \}}
        \right].
\end{aligned}
\]
Similarly, using $ (\log\varepsilon^{-1})^{1/2+\zeta}
        \ge
      (N_\varepsilon^\star)^{1/2+\zeta/2} $    for all sufficiently small
\(\varepsilon>0\),
\[
\begin{aligned}
        &\mathbb E\left[
        \mathcal Z_{\varepsilon} 
        \mathbf 1_{\{
        \mathcal Z_{\varepsilon} 
        >
        (\log\varepsilon^{-1})^{1/2+\zeta}
        \}}
        \right]   =
        \mathbb E\left[
        \mathcal Z_{\widehat\varepsilon,b_\varepsilon}^{\rm aux} 
        \mathbf 1_{\{
        \mathcal Z_{\widehat\varepsilon,b_\varepsilon}^{\rm aux} 
        >
        (\log\varepsilon^{-1})^{1/2+\zeta}
        \}}
        \right]     \le
        \mathbb E\left[
        \mathcal Z_{\widehat\varepsilon,b_\varepsilon}^{\rm aux} 
        \mathbf 1_{\{
        \mathcal Z_{\widehat\varepsilon,b_\varepsilon}^{\rm aux} 
        >
         (N_\varepsilon^\star)^{1/2+\zeta/2}
        \}}
        \right].
\end{aligned}
\]
Applying the above two bounds to
\eqref{eq:aux-size-biased-concentration-endpoint-excluded}, as $   \log\widehat\varepsilon^{-1}
        =
        \log\varepsilon^{-1}+O_{b^\star}(1),$  concludes the proof.

\end{proof}

\section{Minkowski dimension and logarithmic correction}\label{min567891}
We are now ready to prove Theorem \ref{main1}.  
We first smoothen SHF at scale \(\varepsilon\), then apply the mass-concentration estimate Corollary \ref{corcorcor} pointwise in space, and  use the Borel--Cantelli lemma along a suitably sparse sequence of scales followed by a sandwiching argument which controls oscillations between scales.

Let us begin with a brief discussion of the argument which is simply an application of the first moment method given Corollary \ref{corcorcor}. Namely, we know that the total SHF mass in non-zero, and  simply by taking expectations one can argue that balls of radius $\e$ where the mass (relative to Lebesgue measure) is not close to $\sqrt{N_{\e,b}}$ do not contribute. Thus, there must be approximately $\frac{1}{\e^2\sqrt{N_{\e,b}}}$ many balls of radius $\e$ whose SHF mass is close to $\e^2 \sqrt{N_{\e,b}}$ that support almost the entirety of the SHF mass. 

We will actually prove the result first along a sparse sequence
\begin{align}\label{verysparse}
        r_m:=e^{-e^m},
        \qquad \forall m\ge1.
\end{align} 
which will make the relevant error
probabilities  summable, and hence allow an application of the Borel--Cantelli lemma to obtain an almost sure statement.
Once the result is established along the sparse scales, we extend it to an
arbitrary decreasing sequence \(\varepsilon_n\downarrow0\) by a covering
argument. That is, for any $\e$, we will choose \(m=m(\e)\) such that $ r_{m+1}<\varepsilon\le r_m .$ 
Then the covering at scale \(\varepsilon\) can be compared with coverings at the
neighboring sparse scales \(r_m\) and \(r_{m+1}\). The key property that allows this is that the logarithmic correction remains stable between $r_m$ and $r_{m+1}$ since $  \log r_{m+1}^{-1}
        =
        e \log r_m^{-1}$.  

\begin{proof}[Proof of Theorem \ref{main1}]
As in some of the earlier proofs, to ease readability, we break the argument into multiple steps.  \\

\noindent
\textbf{Step 1.}
For notational brevity, we set
\begin{align*}
   \mathscr Z  :=  \mathscr Z_{0,1}^\vartheta  \blacktriangleright 1 \in \mathcal M_+(\mathbb R^2)
\end{align*}
(we suppress the parameter $\vartheta$ in $   \mathscr Z $).
For \(a>0\), define the smoothed density
\begin{align}\label{sden}
        F_{\varepsilon,a}(x)
        :=
        \int_{\mathbb R^2}
        p(a^2\varepsilon^2,x-y)\,\mathscr Z (dy) = p(a^2 \varepsilon^2 , \cdot-x) \blacktriangleleft \mathscr Z_{0,1}^\vartheta \blacktriangleright 1  .
\end{align}
which is exactly the quantity $\cZ_{a\e}$ defined in \eqref{ori} centered at $x$ instead of the origin. 
Observe then that 
by  (spatial) translation invariance of SHF,  Corollary \ref{corcorcor}  holds for any reference point $x\in \mathbb R^2$ (i.e. for the translated version $p(\varepsilon^2, \cdot -x )\blacktriangleleft \mathscr{Z}^\vartheta_{0,1}\blacktriangleright 1$). Integrating it over a bounded set \(D\subset\mathbb R^2\), we deduce that for every constant 
\(a>0\), for all sufficiently small $\varepsilon>0,$  
\begin{align}
        \mathbb E\int_D
        F_{\varepsilon,a}(x)
        \mathbf 1_{\{
        F_{\varepsilon,a}(x)<(\log  \varepsilon^{-1})^{1/2-\xi/2}
        \}}
        \,dx
        & +
        \mathbb E\int_D
        F_{\varepsilon,a}(x)
        \mathbf 1_{\{
        F_{\varepsilon,a}(x)>(\log  \varepsilon^{-1})^{1/2+\xi/2}
        \}}
        \,dx \nonumber\\
    &  \qquad \qquad \qquad   \le
        C(\log  \varepsilon^{-1})^{-\upsilon_{\xi/2}}
\label{eq:quantitative-density-concentration}
\end{align}
{for some constant $C>0$ depending on $a$.}

~

We now begin with proving the statement for the sequence $(r_n)_{n \ge 1}$ defined in \eqref{verysparse}.

~

\noindent
\textbf{Step 2.} 
For \(A\subset\mathbb R^2\) and \(r>0\), we write $   A^{(r)}
        :=
        \{x\in\mathbb R^2:\operatorname{dist}(x,A)\le r\}.$ 
Recalling from the statement of the theorem that $\mathsf{Q}=[0,1]^2$, define
\[
        A_{n,\xi}
        :=
        \left\{
        x\in \mathsf Q^{(r_n/2)}:
        F_{r_n,1}(x)\ge (\log r_n^{-1})^{1/2-\xi/2}
        \right\}.
\]
We claim that \(A_{n,\xi}^{(r_n/2)}\) carries asymptotically all the
\(\mathscr Z\)-mass in \(\mathsf Q\).
Indeed, if \(y\in \mathsf Q\setminus A_{n,\xi}^{(r_n/2)}\), then $B(y,r_n/2)
        \subset
        \mathsf Q^{(r_n/2)}\setminus A_{n,\xi}.$ This implies that for any \(y\in \mathsf  Q\setminus A_{n,\xi}^{(r_n/2)}\),
\begin{align}
        \int_{\mathsf Q^{(r_n/2)}\setminus A_{n,\xi}}
        p(r_n^2,x-y)\,dx
        &\ge
        \int_{B(y,r_n/2)}
        p(r_n^2,x-y)\,dx  =
        \int_{B(0,1/2)}
        \frac{1}{2\pi}e^{-|u|^2/2}\,du
        =:c_1>0.
\label{eq:kernel-mass-lower-upper-cover}
\end{align}
Hence using Fubini's theorem, recalling the definition of  $ F_{\varepsilon,a}(x)$  in \eqref{sden},
\begin{align}
        c_1\,
        \mathscr Z\bigl(\mathsf Q\setminus A_{n,\xi}^{(r_n/2)}\bigr)
        &=   \int_{\mathsf Q\setminus A_{n,\xi}^{(r_n/2)}} c_1 \, \mathscr Z (dy) \notag \\
        &\le
       \int_{\mathsf Q\setminus A_{n,\xi}^{(r_n/2)}}\int_{\mathsf Q^{(r_n/2)}\setminus A_{n,\xi}}
        p(r_n^2,x-y)\,dx\,\mathscr Z(dy) \notag\\
        &\le
        \int_{\mathsf Q^{(r_n/2)}\setminus A_{n,\xi}}
        F_{r_n,1}(x)\,dx \le
        \int_{\mathsf Q^{(1)}}
        F_{r_n,1}(x)
        \mathbf 1_{\{
        F_{r_n,1}(x)<(\log r_n^{-1})^{1/2-\xi/2}
        \}}
        \,dx
\label{eq:mass-outside-good-upper}
\end{align}
Taking the expectation and summing over all $n \in \mathbb N$, by
\eqref{eq:quantitative-density-concentration} and noting that $\log r_n^{-1} = e^n$,  
\[
        \sum_{n=1}^\infty
        \mathbb E\,
        \mathscr Z\bigl(\mathsf Q\setminus A_{n,\xi}^{(r_n/2)}\bigr)
        <\infty.
\]
Since the summands are nonnegative, it follows that
\begin{align}
        \mathscr Z\bigl(\mathsf Q\setminus A_{n,\xi}^{(r_n/2)}\bigr)
        \longrightarrow0
        \qquad\text{almost surely}.
\label{eq:good-set-covers-mass-as}
\end{align}
It remains to cover \(A_{n,\xi}^{(r_n/2)}\). Let $x_1,\dots,x_{M_n}$ 
be a maximal \((r_n/5)\)-separated subset of \(A_{n,\xi}\). Then the
balls \(B(x_k,r_n/10)\) for $1\le k \le M_n$ are disjoint, and by maximality, $\bigcup_{k=1}^{M_n}B(x_k,r_n/5)$ covers $ A_{n,\xi}.$ 
Consequently,
\begin{align}
        A_{n,\xi}^{(r_n/2)}
        \subset
        \bigcup_{k=1}^{M_n}
        B(x_k,r_n).
\label{eq:cover-good-set-neighborhood}
\end{align}
We now upper bound \(M_n\). We use the following heat kernel comparison:  for some constant  $c>0$,
\begin{align}
          |y-x|\le \frac{r_n}{10} \Rightarrow p(4r_n^2,y-z)
        \ge
        c\,p(r_n^2,x-z) \qquad \forall z\in\mathbb R^2.
\label{eq:heat-kernel-comparison-upper}
\end{align} {That this holds is because  the displacement of the center,
\(\lvert y-x\rvert\le r_n/10\), is comparable to the diffusive fluctuation scale \(\sqrt{ r_n^2}= r_n\).} More precisely, writing \(\Delta:=y-x\), we have $|\Delta| \le r_n/10$ and thus
{\[
        \frac{p(4r_n^2,y-z)}
             {p(r_n^2,x-z)}
        =
        \frac14
        \exp\left\{
        -\frac{|x-z+\Delta|^2}{8r_n^2}
        +
        \frac{|x-z|^2}{2r_n^2}
        \right\} 
\ge 
        \frac14
        \exp\left\{
        -C\frac{|\Delta|^2}{r_n^2}
        \right\}
        \ge \frac14 \exp\{-C/100 \}.
\]}
Thus, for every \(y\in B(x_k,r_n/10)\), using \eqref{eq:heat-kernel-comparison-upper} and recalling $x_k \in A_{n,\xi}$,
\[
    F_{r_n,2}(y)
    \ge cF_{r_n,1}(x_k)
        \ge
        c(\log r_n^{-1})^{1/2-\xi/2}.
\]
Since the balls \(B(x_k,r_n/10)\)   for $1\le k \le M_n$ are disjoint and are contained in
\(\mathsf Q^{(2)}\),
\begin{align}
        \int_{\mathsf Q^{(2)}}F_{r_n,2}(y)\,dy
        &\ge
        \sum_{k=1}^{M_n}
        \int_{B(x_k,r_n/10)}
        F_{r_n,2}(y)\,dy\ge
        C M_nr_n^2(\log r_n^{-1})^{1/2-\xi/2}.
\label{eq:number-good-balls-bound-as}
\end{align} 
Since $ \mathbb E F_{r_n,2}(y)=1$ for all $y\in \mathbb R^2,$  
we get $\mathbb E M_n
        \le
        Cr_n^{-2}(\log r_n^{-1})^{-1/2+\xi/2},$ and thus by Markov's inequality,
\[
        \mathbb P\left(
        M_n>
        r_n^{-2}(\log r_n^{-1})^{-1/2+\xi}
        \right)
        \le
        C (\log r_n^{-1})^{-\xi/2}.
\]
By  the Borel--Cantelli lemma, almost surely,
\begin{align}
        M_n
        \le
        r_n^{-2}(\log r_n^{-1})^{-1/2+\xi}\qquad \text{for all large \(n\).}
\label{eq:number-good-balls-as}
\end{align}
Combining  this with \eqref{eq:good-set-covers-mass-as} and
\eqref{eq:cover-good-set-neighborhood}, we obtain the upper covering
bound for the sequence $(r_n)_{n \ge 1}$.

~

\noindent
\textbf{Step 3. Lower covering bound.} Let \({\mathcal C}_n\) be any union of \(m_n\) balls of
radius \(r_n\), where
\begin{align}    \label{mnlower}
m_n\le
        r_n^{-2}(\log r_n^{-1})^{-1/2-\xi}.  
\end{align}
If \(y\in {\mathcal C}_n\cap \mathsf  Q\), then $B(y,r_n/10)
        \subset {\mathcal C}_n^{(r_n/10)}.$ 
Thus as in \eqref{eq:kernel-mass-lower-upper-cover}, for any \(y\in {\mathcal C}_n\cap \mathsf Q\),
\begin{align}
        \int_{{\mathcal C}_n^{(r_n/10)}}
        p(4r_n^2,x-y)\,dx
        &\ge
        \int_{B(y,r_n/10)}
        p(4r_n^2,x-y)\,dx =
        \int_{B(0,1/20)}
        \frac1{2\pi}e^{-|u|^2/2}\,du
        =:c_2>0.
\label{eq:kernel-mass-lower-lower-cover}
\end{align}
Hence similarly as in \eqref{eq:mass-outside-good-upper}, using Fubini's theorem,
\begin{align}
        c_2\,\mathscr Z({\mathcal C}_n\cap \mathsf Q)
        &\le
        \int_{{\mathcal C}_n\cap \mathsf Q}
        \int_{{\mathcal C}_n^{(r_n/10)}}
        p(4r_n^2,x-y)\,dx\,\mathscr Z(dy) \le
        \int_{{\mathcal C}_n^{(r_n/10)}\cap \mathsf Q^{(1)}}
        F_{r_n,2}(x)\,dx .
\label{eq:small-cover-mass-by-density}
\end{align}
The intersection with \(\mathsf Q^{(1)}\) is harmless, since \(y\in \mathsf Q\) and
\(x\in B(y,r_n/10)\) imply \(x\in \mathsf Q^{(1)}\) for 
large \(n\).
Now we split the last integral into low- and high-density parts of $ F_{r_n,2}$:
\begin{align*}
        \mathscr Z({\mathcal C}_n\cap \mathsf Q)
        &\le
        C (\log r_n^{-1})^{1/2+\xi/2}
        |{\mathcal C}_n^{(r_n/10)}\cap \mathsf Q^{(1)}|+
        C
        \int_{\mathsf Q^{(1)}}
        F_{r_n,2}(x)
        \mathbf 1_{\{
        F_{r_n,2}(x)>(\log r_n^{-1})^{1/2+\xi/2}
        \}}
        \,dx. 
\end{align*}
Since \({\mathcal C}_n\) is a union of \(m_n\) balls of radius \(r_n\), we have $|{\mathcal C}_n^{(r_n/10)}\cap \mathsf  Q^{(1)}|
        \le
        C m_nr_n^2.$ Hence, using \eqref{mnlower},
\begin{align} \label{unifc}
    \mathbb E \mathscr Z({\mathcal C}_n\cap \mathsf Q)
        &\le
        C   (\log r_n^{-1})^{-\xi/2} +
        C \mathbb E
        \int_{\mathsf Q^{(1)}}
        F_{r_n,2}(x)
        \mathbf 1_{\{
        F_{r_n,2}(x)>(\log r_n^{-1})^{1/2+\xi/2}
        \}}
        \,dx.
\end{align}
By \eqref{eq:quantitative-density-concentration},  we have   $\sum_n  \mathbb E \mathscr Z({\mathcal C}_n\cap \mathsf Q)<\infty$,  and thus
\[
        \mathscr Z({\mathcal C}_n\cap \mathsf Q)\longrightarrow0
        \qquad\text{almost surely}.
\]Therefore the lower bound follows for the sequence $(r_n)_{n \ge 1}$.

~

\noindent
\textbf{Step 4. Extension to general sequences.}
Now   let
\((\varepsilon_n)_{n\ge1}\) be arbitrary decreasing sequence such that $\varepsilon_n\downarrow0.$  
For each   large enough \(n\), let  \(m=m(n)\)  be a unique integer such that
\begin{equation}
        r_{m+1}<\varepsilon_n\le r_m.
\label{eq:sandwich-scale-choice}
\end{equation}
Since \(\varepsilon_n\downarrow0\), we have \(m(n)\to\infty\). Also recalling  $r_m = e^{-e^m}$ (see \eqref{verysparse}),
\begin{equation}
        e^m
        =
        \log r_m^{-1}
        \le
        \log\varepsilon_n^{-1}
        <
        \log r_{m+1}^{-1}
        =
        e^{m+1}.
\label{eq:sandwich-log-comparison-direct}
\end{equation}
By Step 2,
for
each \(m\), there exists a collection \(\mathcal B_m\) of balls of radius
\(r_m\) such that
\begin{align} \label{911}
        |\mathcal B_m|
        \le
        r_m^{-2}
        (\log r_m^{-1})^{-1/2+\xi}
        =
        r_m^{-2}
        (e^m)^{-1/2+\xi},
\end{align}
and
\begin{align} \label{912}
      \mathscr Z
        \left(
        \mathsf Q\setminus\bigcup_{B\in\mathcal B_m}B
        \right)
        \longrightarrow0.
\end{align}
Every ball of
radius \(r_m\) in \(\mathbb R^2\) can be covered by at most $  C(\frac{r_m}{\varepsilon_n})^2$ 
balls of radius \(\varepsilon_n\). Replacing each ball in
\(\mathcal B_m\) by such a cover, we obtain a collection \(\widetilde {\mathcal B}_n\) of
balls of radius \(\varepsilon_n\) such that $\cup_{B\in\mathcal B_m}B
        \subset
        \cup_{B\in \widetilde {\mathcal B}_n }B$ 
and
\begin{align*}
        |\widetilde {\mathcal B}_n|
        &\le
        C\left(\frac{r_m}{\varepsilon_n}\right)^2
        |\mathcal B_m|  \overset{\eqref{911}}{\le} 
        C\left(\frac{r_m}{\varepsilon_n}\right)^2
        r_m^{-2}
        (e^m)^{-1/2+\xi} =
        C\varepsilon_n^{-2}
        (e^m)^{-1/2+\xi} \overset{\eqref{eq:sandwich-log-comparison-direct}}{\le} C\varepsilon_n^{-2}
        (\log\varepsilon_n^{-1})^{-1/2+\xi}.
\end{align*}
In addition, \eqref{912} along with $\cup_{B\in\mathcal B_m}B
        \subset
        \cup_{B\in\widetilde {\mathcal B}_n}B$ implies \eqref{900}, 
since \(m(n)\to\infty\)  as $n\rightarrow \infty$.

Similarly for the lower bound, 
note that every ball of
radius \(\varepsilon_n\)  can be covered by at most $  C(\frac{\varepsilon_n}{r_{m+1}})^2$ 
balls of radius \(r_{m+1}\). Replacing each ball in
\(\mathcal C_n\) by such a cover, we obtain a collection 
\(\widetilde{\mathcal C}_{m+1}\) of balls of radius \(r_{m+1}\) such that $\cup_{B\in\mathcal C_n}B
        \subset
        \cup_{B\in \widetilde{\mathcal C}_{m+1}}B$  and
\begin{align*}
        |\widetilde{\mathcal C}_{m+1}|
        &\le
        C\left(\frac{\varepsilon_n}{r_{m+1}}\right)^2
        \varepsilon_n^{-2}
        (\log\varepsilon_n^{-1})^{-1/2-\xi}=
        C r_{m+1}^{-2}
        (\log\varepsilon_n^{-1})^{-1/2-\xi} \overset{\eqref{eq:sandwich-log-comparison-direct}}{\le}
        Cr_{m+1}^{-2}
        (e^{m+1})^{-1/2-\xi}.
\end{align*} 
As $\cup_{B\in\mathcal C_n}B
        \subset
        \cup_{B\in \widetilde{\mathcal C}_{m+1}}B$, we have  
\[
      \mathscr Z ({\mathcal C}_n\cap  \mathsf  Q)
        \le
      \mathscr Z (\widetilde{\mathcal C}_{m+1}\cap  \mathsf  Q) 
\]
and the right-hand-side converges to 0 almost surely as $m\rightarrow \infty$ by Step 3. As \(m(n)\to\infty\) as $n\rightarrow \infty$,  
this proves \eqref{901}.
\end{proof}

\section{Decoupled partition function: Large deviations} \label{sec5}
In this section we prove Proposition \ref{largedeviation45} recalled again to aid reading. 
\begin{proposition}  
\label{cor:full-sum-upper-tail-from-local-ratio-relaxed} 
For any \(\alpha>0\) and \(\eta\in(0,1/10)\),    there exists
\(\widetilde \nu_b\ge0\) (depending on $\alpha$) with \(\widetilde \nu_b\to0\) as \(b\downarrow0\)  such that the following holds: For any small enough constant $b>0$ (depending on $\alpha$), as $\varepsilon \rightarrow 0,$
\begin{equation}\label{sharptail676}
        N_{\varepsilon,b}^{-\frac{\alpha^2}{2(1-\eta)}-\widetilde  \nu_b+o(1)}
        \le
        \mathbb P\left(G_{\eta;1}
        \ge  N_{\varepsilon,b}^{-\frac{1-\eta}{2}+\alpha} 
        \right)
        \le
        N_{\varepsilon,b}^{-\frac{\alpha^2}{2(1-\eta)}+ \widetilde  \nu_b+o(1)}.
\end{equation}
 
\end{proposition}
This statement will follow from a more general large deviation statement recorded next (Proposition \ref{prop:local-ratio-BE-correct}). 
To achieve the desired level of generality, the following statement is presented as a sharp upper-tail large deviation estimate for the contribution
of the coordinates outside a small set \(U\).  This
will be a key input for controlling Radon--Nikodym
derivatives projected on the coordinates in \(U\)  (see Proposition \ref{thm:direct-Lp-no-split-corrected} later). As already alluded to in Section \ref{iop}, the main ingredient is a sharp moment-generating-function estimate for the logarithm of $G_i$, which allows us to identify the correct
tilt for the upper-tail event.\\

\nin
We start with some notational preparation and definitions. 
For \(i\in \{1,2,\dots,M_{\varepsilon, \eta}\}\), recalling \eqref{sing34} we write  
\[
G_i=1+W_i,
\qquad
Y_i:=\log G_i,
\qquad
m_i:=N_{\varepsilon,b}-i.
\]
Recalling from \eqref{ellm}, $M_{\varepsilon, \eta}=N_{\varepsilon,b}-\ell_{\varepsilon,\eta} $, for \(i\in \{1,2,\dots,M_{\varepsilon, \eta}\}\),
\begin{align} \label{mimin}
    m_i \ge  \ell_{\varepsilon,\eta} = (\log \varepsilon^{-1})^\eta.
\end{align}

\begin{remark} 
\label{lem:Gi-strictly-positive} To ensure that \(\log G_i\), and hence \(Y_i\), is well defined almost surely, we again invoke  the strict positivity of \(G_i\). As already indicated in Remark \ref{re13}, this is a consequence of \cite{positive}. To be completely precise, the positivity result in \cite{positive} is formulated for a closely
related partition function involving different test functions and below we quickly explain
 why it applies in the present setting as well.
Let \(B,B'\subset\mathbb R^2\) be  open balls. Set $ c_{B,i}:=\inf_{x\in B}p(t_{i-1},x)>0.$ Then,
\begin{align}\label{positive}
        G_i 
        &=
        \int_{\mathbb R^2\times\mathbb R^2}
        p(t_{i-1},x)\,
        \mathscr Z^\vartheta_{t_{i-1},t_i}(dx,dy) \nonumber \\
        &\ge
        \int_{B\times B'}
        p(t_{i-1},x)\,
        \mathscr Z^\vartheta_{t_{i-1},t_i}(dx,dy)  \ge
        c_{B,i}\,
        \mathscr Z^\vartheta_{t_{i-1},t_i}(B\times B').
\end{align}
{We claim that the last term is strictly positive almost surely, as a consequence of \cite[Theorem~1.4]{positive}.}  Choose a  nonnegative continuous compactly supported function \(u_0\in C_c(\mathbb R^2)\) (which is not identically zero) with
\(\operatorname{supp}(u_0)\subset B\).  By the strict local positivity  result ~\cite[Theorem~1.4]{positive},    for any $T>0,$
\begin{align} \label{pos}
       u_0 \blacktriangleleft    \mathscr Z^\vartheta_{0,T} \blacktriangleright \mathbf{1}_{B'}
        = \int_{\mathbb R^2\times B'}
        u_0(x)\,
        \mathscr Z^\vartheta_{0,T}(dx,dy) > 0 \qquad \text{almost surely.}
\end{align}
Since \(u_0\) is supported in \(B\) and bounded, the above quantity is at most $ \|u_0\|_\infty\,
        \mathscr Z^\vartheta_{0,T}(B\times B').$ 
This  along with \eqref{pos} imply that $ \mathscr Z^\vartheta_{0,T}(B\times B')>0$ almost surely. {Hence applying this (with $T = t_i - t_{i-1}$) to \eqref{positive}, using time-translation invariance of SHF,} we conclude that $G_i>0$ almost surely.
\end{remark}

\nin
\textbf{Moment generating function:}
We next define the following moment generating function,
\begin{equation}\label{keydef34}
M_i(\theta):=\mathbb E[e^{\theta Y_i}]=\mathbb E[G_i^\theta],\qquad  \forall \theta>0.
\end{equation}
Note that we only consider  $\theta>0$. Indeed, this will be the regime of interest to study upper-tail probabilities. It is also worth pointing out that 
 in the upper-tail estimates below, the tilting parameter $\theta$  lies in a compact  interval of \((0,\infty)\).  Moreover, since \(G_i\) may take (positive) values close to \(0\),
negative powers \(G_i^\theta\) with \(\theta<0\) would require inverse-moment
bounds which are not available and are not needed for our argument.   \\

\nin
Given the above, let $\lambda_i(\theta)$  denote $\log M_{i}(\theta)$ for any $i\in \{1,2,\dots,M_{\varepsilon, \eta}\}$ and more generally for any $U\subset\{1,2,\dots,M_{\varepsilon, \eta}\}$ define

\begin{equation}\label{def2356}
\lambda _{U}(\theta)
        :=\sum_{i\in U}
        \lambda_i(\theta).
\end{equation}
For a subset \(U\subset\{1,2,\dots,M_{\varepsilon, \eta}\}\), we write
\begin{align}\label{complement}
        U^c:=
        \{1,2,\dots,M_{\varepsilon, \eta}\}\setminus U 
\end{align}
and set 
\begin{align}
      S_U:=\sum_{i\in U}Y_i .
\label{eq:t-eps-eta-local}
\end{align} 
Also, we set 
\begin{align} \label{deft}
    t_{\eta,\alpha}     :=
        \left(
        -\frac{1-\eta}{2}+\alpha
        \right)\log N_{\varepsilon,b}.
\end{align}
Note that the above definition is dictated by \eqref{eq:V-asymptotic-with-b-gap} below which shows that the typical behavior of 
\begin{equation}\label{intuit}
\sum_{i\le M_{\e, \eta}}{Y_i} \approx \sum_{i\le M_{\e, \eta}}{W_i-\frac{W_i^2}{2}} \approx -\Big(\frac{1-\eta}{2}\Big)\log N_{\varepsilon,b}.
\end{equation}
 This shows that \eqref{deft} is indeed a large deviation threshold. \\

\nin
We now arrive at our key technical large deviation statement.

\begin{proposition}[Tail probability of product of blocks]
\label{prop:local-ratio-BE-correct}
Let $\alpha>0$.
Then there exists \(b_0=b_0(\alpha )>0\) such that the following holds for
any constant \(b\in(0,b_0)\).
Let \(A >0\) and  \(\eta\in(0,1/10)\) be constants, and  $  U\subset\{1,2,\dots,M_{\varepsilon, \eta}\}$ 
be any subset such that
\begin{align}
        \Delta_U:= \sum_{j\in U}\frac1{m_j} \le 100,
\label{condit}
\end{align} 
(the above is the smallness criterion for the set $U$).
Define
\begin{align}\label{uppertailprob345}
        q_{U^c}(s)
        :=
        \mathbb P
        \left(
        S_{U^c}
        \ge
        t_{\eta,\alpha}-s
        \right),\qquad s\in \mathbb R.
\end{align}
Then for all sufficiently small $\varepsilon>0$ (uniformly in \(|s|\le A\) and $U$ satisfying \eqref{condit}),
\begin{align}
        q_{U^c}(s)
        =
        e^{-\mathcal I_{U^c}(t_{\eta,\alpha}-s)}
        K_{U^c}(s),
\label{eq:q-local-asym-correct}
\end{align}
where
\begin{align} \label{supsup}
      {  \mathcal I_{U^c}(x)
        :=
        \sup_{\theta > 0}
        \left\{
        \theta x-\lambda _{U^c}(\theta)
        \right\},}
        \qquad
        \lambda _{U^c}(\theta)
        :=
        \log\mathbb E
        \left[
        e^{\theta S_{U^c}}
        \right] \ (\forall \theta>0),
\end{align} 
and $ K_{U^c}(s)$ is a function such that as $\varepsilon \rightarrow 0,$
\begin{align}
        K_{U^c}(s)
       \asymp_\alpha
        \frac{1}{\sqrt{\log N_{\varepsilon,b}}},
        \qquad  |s|\le A .
\label{eq:Keps-prefactor-bound-A}
\end{align}
In addition, as $\varepsilon \rightarrow 0$ (uniformly in \(|s|\le A\) and $U$ satisfying \eqref{condit}),
\begin{align}
        \left|
        \log\frac{q_{U^c}(s)}{q_{U^c}(0)}
        \right|
        \le
        C
        \left(|s|+ 
       (\log \varepsilon^{-1})^{-\eta/2 + o(1)}
        \right).
\label{eq:local-ratio-final-correct}
\end{align}  
\end{proposition}

{
\begin{remark} We briefly mention some remarks regarding this proposition.\\

\nin
$(1)$
By standard convex analysis (see Lemma \ref{lem:interior-legendre-duality} below), the supremum for $\mathcal I_{U^c}(x)$ in \eqref{supsup} is attained at a point $\theta_x >0$ at which $\lambda_{U^c}'(\theta_x) = x$.\\

\nin
$(2)$ The estimate \eqref{eq:local-ratio-final-correct} is a regularity (in $s$) estimate of the large deviation probability \eqref{uppertailprob345}. This will be a key input in comparing conditional and unconditional distributions of $ (G_i)_{i\in U}$. This is carried out in Section \ref{rnproof34} where further explanation is provided. 
\end{remark}
}

\nin
Before proceeding with the rest of the section we include a roadmap to help guide the reader.
The proof of our large deviation result relies on exponential tilting and hence demands sharp estimates on the moment generating functions of $Y_i=\log G_i.$
{In Subsection \ref{momentsingle} we provide such estimates relying on Taylor expansions. This crucially allows us to go beyond integer moments usually computable for partition functions via connections to random walk or Brownian motion local times. Further, we will also require and deliver regularity and smoothness estimates.} Armed with this, in Subsection \ref{keytechnical123} we prove Proposition \ref{prop:local-ratio-BE-correct}. 
In Subsection \ref{sharptail23}, we analyze the rate function provided by the latter to prove Proposition \ref{cor:full-sum-upper-tail-from-local-ratio-relaxed}. Finally, in Subsection \ref{sizebias23} we use the above to quickly deduce which level sets contribute to the expectation of 
$G_{\eta;1}.$

 \subsection{Moment generating function of a single block}\label{momentsingle}

We begin by recalling that 
{for every integer \(k\ge1\) and a constant $b\in (0,1/2]$},  {by \eqref{eq:centered-G-moment-B3} along with the fact $   M_{\varepsilon, \eta} = N_{\varepsilon, b} - 
        \left\lfloor
        (\log\varepsilon^{-1})^\eta
        \right\rfloor,
$ for all sufficiently small $\varepsilon>0$ (depending on $\eta$)},
\begin{equation}
\mathbb E[W_i]=0,
\qquad
\mathbb E|W_i|^k\le C_k m_i^{-k/2},\qquad \forall i \in [1,M_{\varepsilon, \eta}].
\label{eq:assump-allmom-direct-Lp-final}
\end{equation}
In addition by the second moment estimate \eqref{2moment}, 
there exists   \(\nu_b\ge0\), depending only on \(b\), such
that $\nu_b\rightarrow0$ as $  b\downarrow0,$ 
and the following holds:  For fixed \(b>0\), as \(\varepsilon \rightarrow 0\),
\begin{equation}
        V
        :=
        \sum_{j=1}^{N_{\varepsilon,b}-\ell_{\varepsilon,\eta}}
        \mathbb E[W_j^2]
        =
        a_{\varepsilon,b}\log N_{\varepsilon,b},
        \qquad
        |a_{\varepsilon,b}-(1-\eta)|\le \nu_b+o(1)
\label{eq:V-asymptotic-with-b-gap}
\end{equation}
(since $\sum_{i= \ell_{\varepsilon,\eta}}^{N_{\varepsilon,b}-1} \frac{1}{i} = (1-\eta + o(1)) \log N_{\varepsilon,b}$).\\

 The next lemma is based on Taylor expansion. It provides
a basic expansion, control of \(M_i\) and its first two derivatives, uniformly in \(\theta\) in
compact subsets of \((0,\infty)\).  
\begin{lemma} \label{lem:Mi-expansion-full}
Let $b\in (0,1/2]$ be a constant.
Fix \(0<\Theta_1<\Theta_2\). 
Then, uniformly in \(\theta\in (\Theta_1,\Theta_2)\) and  \(i\in \{1,2,\dots,M_{\varepsilon, \eta}\}\),  as $\varepsilon \rightarrow 0,$ 
\begin{align}
M_i(\theta)
&=
1+\frac{\theta(\theta-1)}{2} \mathbb E[W_i^2]
+O_{\Theta_1,\Theta_2}(m_i^{-3/2}), \label{eq:Mi0-full}\\
M_i'(\theta)
&=
\left(\theta-\frac12\right)\mathbb E[W_i^2]
+O_{\Theta_1,\Theta_2}(m_i^{-3/2}), \label{eq:Mi1-full}\\
M_i''(\theta)
&=
\mathbb E[W_i^2]
+O_{\Theta_1,\Theta_2}(m_i^{-3/2}). \label{eq:Mi2-full}
\end{align} 
In addition, {for every nonnegative integer $r$,} $M_i(\theta)$ is $C^r$ on \((\Theta_1,\Theta_2)\) for all sufficiently small $\varepsilon>0$,  uniformly in \(\theta\in (\Theta_1,\Theta_2)\) and  \(i\in \{1,2,\dots,M_{\varepsilon, \eta}\}\).  {More precisely, for every fixed
integer \(r\ge0\) and a compact interval
\([\Theta_1,\Theta_2]\subset(0,\infty)\), there exists
\(\varepsilon_0=\varepsilon_0(r,\Theta_1,\Theta_2)>0\) such that, whenever
\(0<\varepsilon<\varepsilon_0\),  the function
\(\theta\mapsto M_i(\theta)\) is \(C^r\) on
\((\Theta_1,\Theta_2)\) for every
\(1\le i\le M_{\varepsilon,\eta}\).} 
%\red{order of quantifiers, what does this mean, maybe one should say for a given interval uniform control on all derivatives for all small $\e$}
\end{lemma}
By the second moment estimate   \eqref{2moment}, we have $\mathbb E [ W_i^2] \ge 1/(2m_i)$ for any small enough constant $b>0$, and thus  $O_{\Theta_1,\Theta_2}(m_i^{-3/2})$  may and will be regarded as an error term.

\begin{proof}
%\red{Note that the above claims implicitly imply that $M_i(\theta)$ is at least twice differentiable. \blue{We will in fact prove that for every nonnegative integer $r$, the map $\theta \mapsto M_i(\theta) $ is $ C^{r}$ on \((\Theta_1,\Theta_2)\)  for all sufficiently small $\varepsilon>0$.}}
We first record some consequences of Taylor expansion.\\

Fix \(i\in \{1,2,\dots,M_{\varepsilon, \eta}\}\), and abbreviate
\[
m:=m_i,\qquad W:=W_i,\qquad G:=1+W,\qquad Y:=\log G,\qquad M(\theta):=M_i(\theta).
\]

\noindent
\textbf{Step 1. Taylor expansion.}
For \(w>-1\), define
\begin{align*}
R_0(\theta,w)
&:=(1+w)^\theta-1-\theta w-\frac{\theta(\theta-1)}{2} w^2,\\
R_1(\theta,w)
&:=(1+w)^\theta\log(1+w)-w-\Big(\theta-\frac12\Big) w^2,\\
R_2(\theta,w)
&:=(1+w)^\theta(\log(1+w))^2-w^2.
\end{align*}
We first prove the  remainder bounds
\begin{equation}\label{eq:R-small-full}
|R_j(\theta,w)|\le C_{\Theta_1,\Theta_2} |w|^3,
\qquad
\theta\in[\Theta_1,\Theta_2],\quad |w|\le \frac12,\quad j=0,1,2,
\end{equation}
and for some integer \(q=q(\Theta_2)\ge 3\),
\begin{equation}\label{eq:R-large-full}
|R_j(\theta,w)|\le C_{\Theta_1,\Theta_2} |w|^q,
\qquad
\theta\in[\Theta_1,\Theta_2],\quad w\in \Big(-1,-\frac12 \Big)\cup \Big(\frac12, \infty\Big),\quad j=0,1,2.
\end{equation}

\medskip
\noindent
\textit{Proof of \eqref{eq:R-small-full}.}
Set
\[
f_\theta(w):=(1+w)^\theta,
\qquad
g_\theta(w):=(1+w)^\theta\log(1+w),
\qquad
h_\theta(w):=(1+w)^\theta(\log(1+w))^2.
\] 
{For each fixed \(\theta\in[\Theta_1,\Theta_2]\), we apply Taylor's theorem in the variable \(w\), expanded around \(w=0\).
To make the remainder estimate uniform in \(\theta\), observe that the third
\(w\)-derivatives of the above three functions are jointly continuous in \((\theta,w)\) on the compact set
$        [\Theta_1,\Theta_2]\times[-1/2,1/2].$
Consequently,
\[
        \max_{F\in\{f,g,h\}}
        \sup_{\substack{\theta\in[\Theta_1,\Theta_2]\\ |w|\le1/2}}
        \left|\partial_w^3 F_\theta(w)\right|
       <\infty.
\]
Now fix \(\theta\in[\Theta_1,\Theta_2]\). Since
\[
\begin{gathered}
        f_\theta(0)=1,\qquad
        f_\theta'(0)=\theta,\qquad
        f_\theta''(0)=\theta(\theta-1),\\
        g_\theta(0)=0,\qquad
        g_\theta'(0)=1,\qquad
        g_\theta''(0)=2\theta-1,\\
        h_\theta(0)=0,\qquad
        h_\theta'(0)=0,\qquad
        h_\theta''(0)=2,
\end{gathered}
\] 
 Taylor's
theorem in the variable \(w\)  proves  \eqref{eq:R-small-full}.}

\medskip
\noindent
\textit{Proof of \eqref{eq:R-large-full}.}
Choose an integer $q>\Theta_2\vee 3.$ 
If \(|w|>1/2\), then
\[
1\le 2^q |w|^q,\qquad |w|\le 2^{q-1}|w|^q,\qquad |w|^2\le 2^{q-2}|w|^q.
\]
Hence it suffices to bound
\[
(1+w)^\theta,\qquad
(1+w)^\theta |\log(1+w)|,\qquad
(1+w)^\theta (\log(1+w))^2
\]
by \(C |w|^q\) on  $ (-1,-\frac12  )\cup (\frac12, \infty ).$
\begin{enumerate}
    \item 
If \(-1<w<-1/2\), then \(1+w\in(0,1/2)\). Recalling  the condition $\Theta_1>0$, the functions
\[x\mapsto x^\theta ,
\qquad
x\mapsto x^\theta |\log x|,
\qquad
x\mapsto x^\theta (\log x)^2
\]
are bounded on \(x\in(0,1/2]\), uniformly in \(\theta\in[\Theta_1,\Theta_2]\). Since \(|w|^q\asymp 1\), the desired bounds follow.
\item 
If \(w> 1/2\), then \(1+w\le 3w\), hence $(1+w)^\theta\le C w^{\Theta_2} .$ 
Also \((\log(1+w))^2\le C w^{q-\Theta_2}\) and \(|\log(1+w)|\le C w^{q-\Theta_2}\). Thus the desired bounds hold. 
\end{enumerate}

\vspace{.1in}

\nin
Combining \eqref{eq:R-small-full} and \eqref{eq:R-large-full}, we obtain
\[
|R_j(\theta,W)|
\le
C |W|^3 \mathbf 1_{\{|W|\le 1/2\}}
+
C |W|^q \mathbf 1_{\{|W|>1/2\}},
\qquad j=0,1,2.
\]
Taking expectations and using the moment condition \eqref{eq:assump-allmom-direct-Lp-final}, we conclude that {for some constant $C$ (depending on $\Theta_1$ and $\Theta_2$),}
\begin{equation}\label{eq:Rj-mean-full}
\sup_{\theta\in[\Theta_1,\Theta_2]}\mathbb E|R_j(\theta,W)|
\le C m^{-3/2},
\qquad j=0,1,2.
\end{equation}

\noindent
\textbf{Step 2. Differentiability of $M(\theta)$.} Note that the lemma in fact claims that {for every nonnegative integer $r$, the function  $\theta \mapsto M(\theta)$ is $C^r$ on \((\Theta_1,\Theta_2)\) for all sufficiently small $\varepsilon>0$.} We will rely on a DCT argument. 
Fix an integer \(r\ge0\).
For every \(0\le k\le r+1\), there exists a positive integer \(q\) {(depending on  $\Theta_1,\Theta_2$ and $k$)} such that
% \red{again where is the smallness of $\e$ coming in? It is probably not the case that we can control all derivatives simultaneously for all small $\e$ or is that given $r$, there is a small enough $\e$ dependent on $r$ beyond which we can control the first $r$ derivatives}
\begin{align} \label{diff1}
        \sup_{\theta\in (\Theta_1,\Theta_2)}
        G^\theta|\log G|^k
        \le
        C(1+G^q),\qquad \forall G>0.
\end{align} 
Now define, 
\[
        F_k(\theta)
        :=
        \mathbb E\left[
        G^\theta(\log G)^k
        \right],
        \qquad \forall \theta\in (\Theta_1,\Theta_2).
\]  Fix
\(\theta\in(\Theta_1,\Theta_2)\), and choose \(h>0\)  small so that
\(\theta+h\in (\Theta_1,\Theta_2)\).  By the mean value theorem, for each realization of \(G\),
\[
        \frac{G^{\theta+h}-G^\theta}{h}
        =
        G^\xi\log G,\qquad \exists \xi=\xi(h,G) \in (\theta, \theta+h).
\]
Hence, as $\xi \in (\Theta_1,\Theta_2),$  
\[
        \left|
        \frac{
        G^{\theta+h}(\log G)^k
        -
        G^\theta(\log G)^k
        }{h}
        \right|
        =
        G^\xi|\log G|^{k+1}  \overset{\eqref{diff1}}{\le}
        C(1+G^q)
\] 
{for all $G>0$.}
Note that the RHS above is integrable for {sufficiently small $\varepsilon>0$ (depending on $q$)}, due to 
\eqref{eq:assump-allmom-direct-Lp-final} and recalling \(G=1+W \). Moreover, pointwisely as
\(h\downarrow 0\),
\[
        \frac{
        G^{\theta+h}(\log G)^k
        -
        G^\theta(\log G)^k
        }{h}
        \longrightarrow
        G^\theta(\log G)^{k+1}.
\]
{Similarly this holds when \(h\uparrow 0\).}
Thus, by dominated convergence theorem,  \(F_k\) is differentiable on \((\Theta_1,\Theta_2)\) and with $  F_k'(\theta)=F_{k+1}(\theta).$ Hence
 we conclude inductively that \(F_0=M\)
is \(C^r\) on \((\Theta_1,\Theta_2)\), and
\begin{align} \label{diff2}
        M^{(k)}(\theta)=F_k(\theta)
        =
        \mathbb E\left[
        G^\theta(\log G)^k
        \right],
        \qquad 0\le k\le r.
\end{align}

\noindent
\textbf{Step 3. Conclusion.}  The above preparation now allows us to conclude the proof of the lemma rather quickly. 
Note that
\[
(1+W)^\theta = 1+\theta W + \frac{\theta(\theta-1)}{2} W^2 + R_0(\theta,W).
\]
Taking expectation and using \(\mathbb E[W]=0\),
\[
M(\theta)
=
1+\frac{\theta(\theta-1)}{2} \mathbb E[W^2]+\mathbb E[R_0(\theta,W)].
\]
By \eqref{eq:Rj-mean-full}, we have \eqref{eq:Mi0-full}.
Next, considering the derivative in $\theta$ (see \eqref{diff2}),
\[
M'(\theta)=\mathbb E\big[(1+W)^\theta \log(1+W)\big].
\]
From the definition of \(R_1\),
\[
(1+W)^\theta \log(1+W)
=
W+\Big(\theta-\frac12\Big) W^2+R_1(\theta,W).
\]
Taking expectation and using \(\mathbb E[W]=0\), we have \eqref{eq:Mi1-full}.
Finally, taking the second derivative in $\theta$  as recorded in \eqref{diff2},
\[
M''(\theta)=\mathbb E\big[(1+W)^\theta (\log(1+W))^2\big].
\]
From the definition of \(R_2\),
\[
(1+W)^\theta (\log(1+W))^2 = W^2+R_2(\theta,W).
\]
Taking expectation, we obtain \eqref{eq:Mi2-full}.
\end{proof}

\nin
Another key input in the analysis of the large deviation behavior will be a quantitative central limit theorem under exponential tilting of the $G_i$s. An ingredient for the latter will be the following uniform bound on the higher (in particular, third) centered moment of $   Y_i =\log G_i$, which is provided in the following lemma. 

\begin{lemma} \label{lem:tilted-third-moment-full} 
Let $b\in (0,1/2]$ be a constant.
Let $p \ge 2$ be an  integer and $0<\Theta_1<\Theta_2<\infty.$ 
For  \(i\in \{1,2,\dots,M_{\varepsilon, \eta}\}\) and \(\theta\in (\Theta_1,\Theta_2)\), define the tilted probability measure \(\mathbb Q_{i,\theta}\) by
\[
\frac{d\mathbb Q_{i,\theta}}{d\mathbb P}
=
\frac{G_i^\theta}{M_i(\theta)}, 
\qquad
M_i(\theta):=\mathbb E[G_i^\theta].
\] 
Then there exists a constant \(C=C(p,\Theta_1,\Theta_2)>0\) such that  for all sufficiently small $\varepsilon>0$ (uniformly in  \(\theta\in (\Theta_1,\Theta_2)\) and \(i\in \{1,2,\dots,M_{\varepsilon, \eta}\}\)),
\[
\mathbb E_{\mathbb Q_{i,\theta}}
\left|
Y_i-\mathbb E_{\mathbb Q_{i,\theta}}Y_i
\right|^p
\le C m_i^{-p/2}.
\]
\end{lemma}
{We remark that in the above the centering plays no role.  Indeed, we will apply  triangle
inequality, and show that both \(Y_i\) and
\(\mathbb E_{\mathbb Q_{i,\theta}}Y_i\), are small.}

\begin{proof}Fix \(i\in \{1,2,\dots,M_{\varepsilon, \eta}\}\), and abbreviate
\begin{align*}
        m:=m_i,\qquad
        W:=W_i=G_i-1,\qquad
        G:=G_i , \qquad
        Y:=Y_i ,\qquad
        M(\theta):=M_i(\theta),
        \qquad
        \mathbb Q_\theta:=\mathbb Q_{i,\theta}.
\end{align*}

\noindent
\textbf{Step 1: Lower bound on \(M(\theta)\).}  In fact, a crude argument will suffice. 
Let
\begin{align*}
      \mathcal   E:=\{|W|\le 1/2\}.
\end{align*}
On \( \mathcal  E\), we have \(G\in[1/2,3/2]\), hence
\(G^\theta\ge 2^{-\Theta_2}\). Thus
\begin{align*}
        M(\theta)
        =
        \mathbb E[G^\theta]
        \ge
        \mathbb E[G^\theta; \mathcal E]
        \ge
        2^{-\Theta_2}\mathbb P( \mathcal E).
\end{align*}
Also for sufficiently small $\varepsilon>0,$
\begin{align*}
        \mathbb P( \mathcal E^c)
        =
        \mathbb P(|W|>1/2)
        \le
        4\mathbb E[W^2]
        \overset{\eqref{eq:assump-allmom-direct-Lp-final}}{\le} 
        4C m^{-1} \overset{\eqref{mimin}}{\le} \frac12.
\end{align*}
Thus,  
\begin{align}
        M(\theta)\ge 2^{-\Theta_2-1}.
\label{eq:M-lower-bound-tilted-third}
\end{align}

\noindent
\textbf{Step 2: Bound on \(\mathbb E_{\mathbb Q_{\theta}}|Y|^p\).}
Recall $Y = \log G$. Note that
\begin{align*}
        \mathbb E_{\mathbb Q_{\theta}}|Y|^p
        =
        \frac{\mathbb E[G^\theta |Y|^p]}{M(\theta)}.
\end{align*}
We claim that, uniformly in \(\theta\in (\Theta_1,\Theta_2)\),
\begin{align}
        \mathbb E[G^\theta |Y|^p]
        \le
        C m^{-p/2}.
\label{eq:GY-third-moment-bound}
\end{align}
We split according to the three regions
\begin{align*}
        |W|\le 1/2,
        \qquad
        W\ge 1/2,
        \qquad
        -1<W<-1/2.
\end{align*}
\begin{enumerate}
    \item 
{If \(|W|\le 1/2\), then $  |Y|
        =
        |\log(1+W)|
        \le
      2 |W|$} and  $  G^\theta = (1+W)^\theta \le (3/2)^{\Theta_2},$ 
so
\begin{align*}
        \mathbb E[G^\theta |Y|^p;\ |W|\le 1/2]
        \le
        C\mathbb E|W|^p
   \overset{\eqref{eq:assump-allmom-direct-Lp-final}}{\le} 
        C m^{-p/2}.
\end{align*}

\item  
If \(W\ge 1/2\), choose an integer \(q>\Theta_2+p\). Then
\begin{align*}
        x^\theta(\log x)^p
        \le
        Cx^q
        \qquad
        (x\ge 3/2,\ \theta\in[\Theta_1,\Theta_2]).
\end{align*}
Thus recalling \(G=1+W\) and \(Y=\log(1+W)\), we have $  G^\theta |Y|^p
        \le
        C G^q
        \le
        C|W|^q$ for \( W\ge 1/2\),  
implying that 
\begin{align*}
        \mathbb E[G^\theta |Y|^p;\ W\ge 1/2]
        \le
        C\mathbb E|W|^q
   \overset{\eqref{eq:assump-allmom-direct-Lp-final}}{\le} 
        C m^{-q/2}
        \le
        C m^{-p/2}.
\end{align*}

\item 
If \(-1<W<-1/2\), then \(G\in(0,1/2)\). Since
\(\theta\ge \Theta_1>0\), the function
\(x\mapsto x^{\theta}|\log x|^p\) is bounded on \((0,1/2]\), uniformly in
\(\theta\in[\Theta_1,\Theta_2]\). Thus
\begin{align*}
        \mathbb E[G^\theta |Y|^p;\ -1<W<-1/2]
        \le
        C\mathbb P(W<-1/2).
\end{align*}
Using Markov's inequality,
\begin{align*}
        \mathbb P(W<-1/2)
        \le
        2^p\mathbb E|W|^p
  \overset{\eqref{eq:assump-allmom-direct-Lp-final}}{\le} 
        C m^{-p/2}.
\end{align*}
\end{enumerate}
Combining the three regions proves \eqref{eq:GY-third-moment-bound}. Hence,
combining this with \eqref{eq:M-lower-bound-tilted-third}, uniformly in \(\theta\in (\Theta_1,\Theta_2)\),
\begin{align}
        \mathbb E_{\mathbb Q_{\theta}}|Y|^p
        \le
        C m^{-p/2}.
\label{eq:tilted-Y-third-moment-bound}
\end{align}

\noindent
\textbf{Step 3: Bound on  tilted mean.}
Note that
\begin{align*}
        \mathbb E_{\mathbb Q_{\theta}}Y
        =
        \frac{\mathbb E[G^\theta Y]}{M(\theta)}
        =
        \frac{M'(\theta)}{M(\theta)}.
\end{align*}
As
\(\mathbb E[W^2]\le C m^{-1}\),  by \eqref{eq:Mi1-full} in Lemma~\ref{lem:Mi-expansion-full},  $  |M'(\theta)|
        \le
        C m^{-1} $ 
uniformly in \(\theta\in (\Theta_1,\Theta_2)\).  
Combining this with \eqref{eq:M-lower-bound-tilted-third}, we obtain $  \left|
        \mathbb E_{\mathbb Q_{\theta}}Y
        \right|
        \le
        C m^{-1}.$ 
Therefore, this along with \eqref{eq:tilted-Y-third-moment-bound} implies
\begin{align*}
        \mathbb E_{\mathbb Q_{\theta}}
        \left|
        Y-\mathbb E_{\mathbb Q_{\theta}}Y
        \right|^p
        &\le
        C\mathbb E_{\mathbb Q_{\theta}}|Y|^p
        +
        C\left|
        \mathbb E_{\mathbb Q_{\theta}}Y
        \right|^p  \le
        C m^{-p/2}.
\end{align*}
\end{proof}

With the above preparation, we are now in a position to prove Proposition \ref{prop:local-ratio-BE-correct}.

\subsection{The proof of Proposition \ref{prop:local-ratio-BE-correct}}\label{keytechnical123}
We divide the proof into five steps.
In Step 1 we provide sharp estimates for the moment generating function for the sum of $Y_i$s with $i\in U^c.$
In Step 2 we exponentially tilt this sum appropriately so that the mean now shifts to the vicinity of $t_{\eta,\alpha}.$ In Step 3, we prove a Berry-Esseen quantitative central limit theorem for the sum of $Y_i$s under the tilted measure. In Step 4, we use the tilted measure to compute the large deviation probability and use the last step to reduce things to a Gaussian computation. Finally, in Step 5, we prove a regularity estimate (see \eqref{eq:local-ratio-final-correct}) for the large deviation probabilities at various levels. 

\medskip
\noindent
\textbf{Step 1. Moment generating function of the complement sum.}
 Define
\begin{align}
       V_{U^c}
        :=
        \sum_{j\in U^c}
        \mathbb E[W_j^2]
\label{eq:lambda-Uc-def}
\end{align}
(see \eqref{complement} for the definition of $U^c$).
Since \(\mathbb E[W_j^2]\le C m_j^{-1}\) and recalling $\eta \in (0,1/10)$, by \eqref{eq:V-asymptotic-with-b-gap} and the assumption
\eqref{condit}, for any small enough constant $b>0,$
\begin{align}
        V_{U^c} \in  \left(\frac{3}{4}\log N_{\varepsilon,b}, \frac{5}{4}\log N_{\varepsilon,b}\right).
\label{eq:V-Uc-asymptotic}
\end{align} 
By independence,
\begin{align} \label{add}
      \lambda _{U^c}(\theta)
      =
        \log\mathbb E
        \left[
        e^{\theta S_{U^c}}
        \right] = \log \prod_{j\in U^c}  \mathbb E
        \left[
        e^{\theta Y_j}
        \right]  =
        \sum_{j\in U^c}
        \log M_j(\theta).
\end{align}
Note that $   \lambda _{U^c}(\theta)$ is convex in $(0,\infty)$, i.e. as a consequence of H\"older's inequality,
\begin{align}
    \lambda _{U^c}((1-a)\theta_1+a\theta_2)
        \le
        (1-a)\lambda _{U^c}(\theta_1)+a\lambda _{U^c}(\theta_2),
        \qquad
        \theta_1,\theta_2>0,\quad a\in[0,1].
\end{align}
We claim that for every fixed \(\Theta_2>\Theta_1>0\), as $\varepsilon \rightarrow 0$ (uniformly in
\(\theta\in (\Theta_1,\Theta_2)\)),
\begin{align}
         \lambda _{U^c} '(\theta)
        &=
        \left(\theta-\frac12\right)
        V_{U^c}
        +
        O (1),
\label{eq:Lambda1-local-correct}\\
         \lambda _{U^c} ''(\theta)
        &=
        V_{U^c}
        +
        O (1).
\label{eq:Lambda2-local-correct}
\end{align}
{Note that by Lemma~\ref{lem:Mi-expansion-full} and the additive structure of $  \lambda _{U^c} $ (see \eqref{add}), $  \lambda _{U^c} $ is $C^2$ on $(\Theta_1,\Theta_2)$ for sufficiently small $\varepsilon>0.$}
By \eqref{eq:Mi0-full} in  Lemma~\ref{lem:Mi-expansion-full} and that \(\mathbb E[W_j^2]\le C m_j^{-1}\), we have $   M_j(\theta)
        =
        1+O(m_j^{-1})$ as $\varepsilon \rightarrow 0$, uniformly in
\(\theta\in (\Theta_1,\Theta_2)\) and $j\in \{1,2,\dots,M_{\varepsilon, \eta}\}$. As $m_j^{-1} \rightarrow 0$ (see  \eqref{mimin}), 
\begin{align}
        \frac1{M_j(\theta)}
        =
        1+O(m_j^{-1}).
\label{eq:Mj-inverse-close-one}
\end{align}
Hence, using \eqref{eq:Mi1-full} in  Lemma~\ref{lem:Mi-expansion-full}, as  $\mathbb E[W_j^2]\le C m_j^{-1},$
\begin{align}
        \frac{d}{d\theta}\log M_j(\theta)  = \frac{M_j'(\theta)}{M_j(\theta)}
        &=
        \left[
        \left(\theta-\frac12\right)\mathbb E[W_j^2]
        +
        O(m_j^{-3/2})
        \right]
        \left[
        1+O(m_j^{-1})
        \right] \notag\\
        &=
        \left(\theta-\frac12\right)\mathbb E[W_j^2]
        +
        O(m_j^{-3/2}).
\label{eq:Mj-log-derivative-1}
\end{align} 
Summing over $j\in U^c,$
we obtain
\eqref{eq:Lambda1-local-correct}.
Similarly,   using \eqref{eq:Mi2-full} in  Lemma~\ref{lem:Mi-expansion-full},  as  $\mathbb E[W_j^2]\le C m_j^{-1},$
\begin{align} 
        \frac{M_j''(\theta)}{M_j(\theta)}
        &=
        \left[
        \mathbb E[W_j^2]
        +
        O(m_j^{-3/2})
        \right]
        \left[
        1+O(m_j^{-1})
        \right] =
        \mathbb E[W_j^2]
        +
        O(m_j^{-3/2}).
\label{eq:Mj-second-over-Mj}
\end{align} 
This and \eqref{eq:Mj-log-derivative-1} {along with  the bound  $\mathbb E[W_j^2]\le C m_j^{-1}$} imply that 
\begin{align}
        \frac{d^2}{d\theta^2}\log M_j(\theta)
        &=
        \frac{M_j''(\theta)}{M_j(\theta)}
        -
        \left(
        \frac{M_j'(\theta)}{M_j(\theta)}
        \right)^2=  \mathbb E[W_j^2]
        +
        O(m_j^{-3/2}).
\label{eq:Mj-log-derivative-2}
\end{align}
Summing over $j\in U^c,$
we obtain \eqref{eq:Lambda2-local-correct}.

\medskip
\noindent
\textbf{Step 2. Tilted law at level \(t_{\eta,\alpha}-s\).}
Set  \begin{align} \label{defj}
       J:=[\theta_-,\theta_+],\qquad  \theta_-:=
        \frac{\alpha}{2(1-\eta)},
        \qquad
        \theta_+:=
        \frac{2\alpha}{1-\eta}.
\end{align}
The above choice is guided by the following fact proved later in Lemma \ref{lem:compact}. Namely, for any small enough constant $b>0$ and \(|s|\le A\), for all sufficiently small \(\varepsilon>0\) {(uniformly in
\(|s|\le A\) and subsets \(U\) satisfying \eqref{condit}),} there is a unique solution  \(\theta=\theta_s \in J\) to   the  equation 
% \red{where is the $\e$ quantification, why is there no A dependence}
\begin{align}
         \lambda _{U^c} '(\theta)
        =
        t_{\eta,\alpha}-s
\label{eq:saddle-eq}
\end{align} 
{By Legendre duality} (see Lemma \ref{lem:interior-legendre-duality} below),
\begin{align}\label{legend1}
  \mathcal I_{U^c}(t_{\eta,\alpha}-s)+\lambda_{U^c}(\theta_s ) =    \theta_s (t_{\eta,\alpha}-s).
\end{align} 
Also, {recalling that $\theta_s$ belongs to the compact interval $J \subset (0,\infty)$ once $|s|\le A,$ by \eqref{eq:Lambda2-local-correct},   as $\varepsilon \rightarrow 0$ (uniformly in $|s|\le A$ and subsets \(U\) satisfying \eqref{condit}),}
\begin{align}
        v_s
        &:=
         \lambda _{U^c} ''
        (\theta_s) =
        V_{U^c}
        +
        O(1) \overset{\eqref{eq:V-Uc-asymptotic}}{\in}
   \left(\frac{1}{2}\log N_{\varepsilon,b}, \frac{3}{2}\log N_{\varepsilon,b}\right). 
\label{eq:v-close-V-local-correct}
\end{align}
Now define the tilted law \(\mathbb Q_{ \theta_s}\) on
\(\sigma(Y_j:j\in U^c)\) by
\begin{align}
        \frac{d\mathbb Q_{ \theta_s}}{d\mathbb P}
        =
        \exp
        \left(
        \theta_sS_{U^c}
        -
        \lambda _{U^c}(\theta_s)
        \right).
\end{align}
Under \(\mathbb Q_{ \theta_s}\), the variables \((Y_j)_{j\in U^c}\)
remain independent, and
\begin{align}
        \mathbb E_{\mathbb Q_{ \theta_s}}
        [S_{U^c}]
        = \sum_{i\in U^{c}} \mathbb E_{\mathbb Q_{ \theta_s}}(Y_i)=\sum_{i\in U^{c}}\lambda_{i}'(\theta_s)
        = \lambda _{U^c} '
        (\theta_s)
        =
        t_{\eta,\alpha}-s,
\end{align}
and similarly
\begin{align}
        \operatorname{Var}_{\mathbb Q_{ \theta_s}}
        (S_{U^c})
        =
         \lambda _{U^c} ''
        (\theta_s)
        =
        v_s.
\end{align}
Hence, setting
\begin{align}\label{ldef}
        S_{U^c}^{(s)}
        :=
        S_{U^c}
        -
        (t_{\eta,\alpha}-s),
\end{align}
we have 
\begin{align}
        \mathbb E_{\mathbb Q_{ \theta_s}}
        [S_{U^c}^{(s)}]
        =
        0,
        \qquad
        \operatorname{Var}_{\mathbb Q_{ \theta_s}}
        (S_{U^c}^{(s)})
        =
        v_s.
\end{align}
Since $\theta_s \in J$,  by Lemma~\ref{lem:tilted-third-moment-full},  for all sufficiently small $\varepsilon>0$ (uniformly in $|s|\le A$ and \(j\in \{1,2,\dots,M_{\varepsilon, \eta}\}\)),
\begin{align} \label{9755}
        \mathbb E_{\mathbb Q_{ \theta_s}}
        \left|
        Y_j-\mathbb E_{\mathbb Q_{ \theta_s}}Y_j
        \right|^3
        \le
        C m_j^{-3/2}.
\end{align}
This implies that
\begin{align}
        \sum_{j\in U^c}
        \mathbb E_{\mathbb Q_{ \theta_s}}
        \left|
        Y_j-\mathbb E_{\mathbb Q_{ \theta_s}}Y_j
        \right|^3
        \le
        C\sum_{j=\ell_{\varepsilon,\eta} }^\infty j^{-3/2}
        \le C(\log \varepsilon^{-1})^{-\eta/2}.
\label{eq:third-moment-sum-local}
\end{align}

\medskip

We next prove a central limit theorem for the $Y_i$s under the tilting from the preceding discussion.\\

\nin
\textbf{Step 3. Berry--Esseen bound under the tilted law.}
Define
\begin{align} \label{defz}
       \mathsf{Z}_{ s}
        :=
        \frac{S_{U^c}^{(s)}}{\sqrt{v_s}} = \frac{S_{U^c}-\mathbb E_{\mathbb Q_{ \theta_s}}
        (S_{U^c})}{\sqrt{\operatorname{Var}_{\mathbb Q_{ \theta_s}}
        (S_{U^c})}}.
\end{align}
Let \(\mathsf{Z} \) be a standard Gaussian. 
Recalling that   \((Y_j)_{j\in U^c}\)
are  independent under \(\mathbb Q_{ \theta_s}\), 
by the Berry--Esseen theorem  (see \cite{berry1941accuracy,esseen1942liapounoff}) {which says that for independent centered random variables $X_1,\cdots,X_n,$
\[
    \sup_{x\in\mathbb R}
    \left|
        \mathbb P\left(
            \frac{\sum_{i=1}^n X_i}{\sqrt{\sum_{i=1}^n \sigma_i^2}}\le x
        \right)
        -
       \mathbb P(\mathsf{Z}\le x) 
    \right|
    \le
    \frac{C  }{(\sum_{i=1}^n \sigma_i^2)^{3/2}}
    \sum_{i=1}^n \mathbb E |X_i|^3
\]
(here, $\sigma_i^2 := \mathbb E |X_i|^2>0$),} we obtain
 \begin{align} \label{eq:BE-uniform-local-correct}
        \sup_{x\in\mathbb R}
        \left|
        \mathbb Q_{ \theta_s}
        ( \mathsf{Z}_{ s}\le x)
        -
   \mathbb P(\mathsf{Z}\le x) 
        \right|
        &\le
        C
        \frac{
        \sum_{j\in U^c}
        \mathbb E_{\mathbb Q_{ \theta_s}}
        \left|
        Y_j-\mathbb E_{\mathbb Q_{ \theta_s}}Y_j
        \right|^3
        }{
        v_s^{3/2}
        } \notag \\
        &\le C
        \frac{(\log \varepsilon^{-1})^{-\eta/2}}{(\log N_{\varepsilon,b})^{3/2}} \le (\log \varepsilon^{-1})^{-\eta/2},
\end{align}
where we used \eqref{eq:v-close-V-local-correct} and
\eqref{eq:third-moment-sum-local} in the second last inequality. 

\medskip
\noindent\medskip
\noindent
\textbf{Step 4. Tilted representation.} We now compute the upper-tail probability using the tilted measure. Namely, using \eqref{legend1} and \eqref{ldef},
\begin{align}
        q_{U^c}(s)
        &=
        \mathbb P
        \left(
        S_{U^c}
        \ge
        t_{\eta,\alpha}-s
        \right) \notag\\
        &=
        \mathbb E_{\mathbb Q_{\theta_s}}
        \left[
        \exp
        \left(
        -\theta_sS_{U^c}
        +
        \lambda_{U^c}(\theta_s)
        \right)
        \mathbf 1_{\{
        S_{U^c}
        \ge
        t_{\eta,\alpha}-s
        \}}
        \right] \notag\\
        &=
        e^{-\mathcal I_{U^c}(t_{\eta,\alpha}-s)}
        \mathbb E_{\mathbb Q_{\theta_s}}
        \left[
        e^{-\theta_sS_{U^c}^{(s)}}
        \mathbf 1_{\{S_{U^c}^{(s)}\ge0\}}
        \right].
\label{eq:q-exact-local-correct}
\end{align}
Drawing parallel between the above expression and \eqref{eq:q-local-asym-correct} we define the pre-factor
\begin{align}
        K_{U^c}(s)
        :=
        \mathbb E_{\mathbb Q_{\theta_s}}
        \left[
        e^{-\theta_sS_{U^c}^{(s)}}
        \mathbf 1_{\{S_{U^c}^{(s)}\ge0\}}
        \right].
\label{eq:Keps-def}
\end{align}
The preceding central limit theorem will now allow us to essentially perform a Gaussian computation for the RHS.
To this end set $ g_{\theta,v}(z):=
        e^{-\theta\sqrt v z}\mathbf 1_{\{z\ge0\}},$ 
and recalling \eqref{defz}, we have
\begin{align}
        K_{U^c}(s)
        =
        \mathbb E_{\mathbb Q_{\theta_s}}
        \left[
        g_{\theta_s,v_s}(\mathsf Z_s)
        \right].
\end{align}
Since \(z\mapsto g_{\theta_s,v_s}(z)\) vanishes on \((-\infty,0)\), jumps
from \(0\) to \(1\) at \(0\), and decreases monotonically from \(1\) to \(0\)
on \([0,\infty)\), we have $ \|g_{\theta_s,v_s}\|_{\mathrm{TV}}\le2$ where $\|\cdot\|_{\mathrm{TV}}$ denotes the total variation semi-norm. Hence, by Lemma~\ref{tv} below (a straightforward consequence of integration by parts) 
\[
        \Big|
        \mathbb E_{\mathbb Q_{\theta_s}}[g_{\theta_s,v_s}(\mathsf Z_s)]-\mathbb E[g_{\theta_s,v_s}(\mathsf Z)]]
        \Big|
        \le
        \|g_{\theta_s,v_s}\|_{\mathrm{TV}}
          \sup_{x\in\mathbb R}
        \left|
        \mathbb Q_{ \theta_s}
        ( \mathsf{Z}_{ s}\le x)
        -
   \mathbb P(\mathsf{Z}\le x) 
        \right|.
\]
Thus by 
\eqref{eq:BE-uniform-local-correct},
\begin{align}
        \left|
        K_{U^c}(s)
        -
        \mathbb E
        \left[
        g_{\theta_s,v_s}(\mathsf Z)
        \right]
        \right| =           \left|  \mathbb E_{\mathbb Q_{\theta_s}}
        \left[
        g_{\theta_s,v_s}(\mathsf Z_s)
        \right] -   \mathbb E
        \left[
        g_{\theta_s,v_s}(\mathsf Z)
        \right]
        \right| 
        \le
      C(\log \varepsilon^{-1})^{-\eta/2},
        \qquad |s|\le A.
\label{eq:Keps-Gaussian-comparison}
\end{align} 
{Note that the restriction \(|s|\le A\) enters through the Berry--Esseen estimate under
the tilted measure.  Indeed, the centered third
moment bound \eqref{9755} is obtained when the tilting parameter \(\theta_s\) remains in the
fixed compact interval \(J\subset(0,\infty)\), which is
guaranteed for bounded shifts \(|s|\le A\).} 

{It is not hard to see and we will show below that  $\mathbb E\left[g_{\theta_s,v_s}(\mathsf Z)\right]
    \asymp \frac{1}{\sqrt{\log N_{\varepsilon,b}}}$. This is because $v_s$ is up to constants $\log N_{\varepsilon,b}$ and hence the entire contribution of the expectation of $e^{-\theta\sqrt v z}\mathbf 1_{\{z\ge0\}}$ comes from when $\mathsf Z \approx \frac{1}{\sqrt{\log N_{\varepsilon,b}}}$ and the integrand is of unit order.  Thus, it is crucial that the error in \eqref{eq:Keps-Gaussian-comparison} is poly-logarithmic in $\e$ and hence is negligible compared
to this quantity.} It remains to estimate the Gaussian  factor $ \mathfrak G(s)
        :=
        \mathbb E
        \left[
        g_{\theta_s,v_s}(\mathsf Z)
        \right] .$ 
Then setting $\Phi(x):=\mathbb P(\mathsf Z>x),$
\begin{align}
        \mathfrak G(s)
        &=
        \frac1{\sqrt{2\pi}}
        \int_0^\infty
        \exp
        \left(
        -\theta_s\sqrt{v_s}z-\frac{z^2}{2}
        \right)\,dz \nonumber\\
        &= 
        e^{(\theta_s\sqrt{v_s})^2/2}
        \frac1{\sqrt{2\pi}}
        \int_0^\infty
        e^{-(z+\theta_s\sqrt{v_s})^2/2}\,dz  =
        e^{\theta_s^2v_s/2}
        \Phi(\theta_s\sqrt{v_s}).
\label{eq:Gaussian-integral-exact}
\end{align}  
Since \(\theta_s \in J\), using \eqref{eq:v-close-V-local-correct} {and the fact that $e^{x^2/2}\Phi(x)
      \asymp
        \frac1{ x} $ for $x>1$  ({see e.g. \cite[(7) and (8)]{gordon1941values}}),}
\begin{align}
        \mathfrak G(s)
        \asymp_\alpha
        \frac1{\sqrt{\log N_{\varepsilon,b}}},
        \qquad |s|\le A.
\label{eq:Gaussian-prefactor-asymp}
\end{align}
Combining this and \eqref{eq:Keps-Gaussian-comparison}, recalling $N_{\varepsilon,b} \asymp \log \varepsilon^{-1},$ as $\varepsilon \rightarrow 0$ {(uniformly in $|s|\le A$ and a subset $U$ satisfying \eqref{condit})},
\begin{align}
        K_{U^c}(s)
        =
        \mathfrak G(s)(1+\iota_s),
        \qquad
        |\iota_s|
        \le
          (\log \varepsilon^{-1})^{-\eta/2 + o(1)} .
\label{eq:K-Gaussian-relative-error}
\end{align}
In particular,
\begin{align}
        K_{U^c}(s)
       \asymp_\alpha
        \frac1{\sqrt{\log N_{\varepsilon,b}}},
        \qquad |s|\le A.
\label{eq:prefactor-local-bounds-correct}
\end{align}
Applying this to \eqref{eq:q-exact-local-correct},  we establish
\eqref{eq:q-local-asym-correct} with the property \eqref{eq:Keps-prefactor-bound-A}.\\

We now estimate the ratio \(K_{U^c}(s)/K_{U^c}(0)\).  By
\eqref{eq:K-Gaussian-relative-error},
\begin{align}
        \left|
        \log\frac{K_{U^c}(s)}{K_{U^c}(0)}
        -
        \log\frac{\mathfrak G(s)}{\mathfrak G(0)}
        \right|
        &=
        \left|
        \log(1+\iota_s)
        -
        \log(1+\iota_0)
        \right|  \le
          (\log \varepsilon^{-1})^{-\eta/2 + o(1)}.
\label{eq:K-ratio-vs-G-ratio}
\end{align} 
Thinking of the above as negligible we next estimate \(\log \frac{\mathfrak G(s)}{\mathfrak G(0)}\). 
To obtain a useful expression, for \(x>0\) we use the  representation
\begin{align}
        e^{x^2/2}\Phi(x)
        &=
        \frac1{\sqrt{2\pi}}
        \int_0^\infty
        e^{-xy-y^2/2}\,dy =
        \frac1{\sqrt{2\pi}x}
        \int_0^\infty
        e^{-t}
        e^{-t^2/(2x^2)}
        \,dt=  
        \frac1{\sqrt{2\pi}x}
        \left(1+\Lambda(x)\right),
\label{eq:Mills-integral-refined}
\end{align} 
where $\Lambda(x)
        :=
        \int_0^\infty
        e^{-t}
    (
        e^{-t^2/(2x^2)}-1
   )\,dt.$ Here, the first identity follows from $  \Phi(x)
        =
        \frac1{\sqrt{2\pi}}
        \int_x^\infty e^{-u^2/2}\,du$ 
by the change of variables \(u=x+y\), and the second identity follows from the
change of variables \(t=xy\).
Thus, the above along with  \eqref{eq:Gaussian-integral-exact} implies \begin{align}
        \mathfrak G(s)
        =
        \frac1{\sqrt{2\pi}\,\theta_s\sqrt{v_s}}
        \left(1+\Lambda(\theta_s\sqrt{v_s})\right),
        \qquad
        \mathfrak G(0)
        =
        \frac1{\sqrt{2\pi}\,\theta_0\sqrt{v_0}}
        \left(1+\Lambda(\theta_0\sqrt{v_0})\right).
\end{align}
Hence we write
\begin{align}
        \log\frac{\mathfrak G(s)}{\mathfrak G(0)}
        &=
        -\log\frac{\theta_s}{\theta_0}
        -
        \frac12\log\frac{v_s}{v_0}
        +
        \log(1+\Lambda(\theta_s\sqrt{v_s}))
        -
        \log(1+\Lambda(\theta_0\sqrt{v_0})).
\label{eq:G-ratio-log-expanded}
\end{align}
We will now estimate the above terms by several applications of the mean value theorem.
First,
\begin{align}
        |s|
        \overset{\eqref{eq:saddle-eq}}{=}
        \left|
        \lambda_{U^c}'(\theta_s)
        -
        \lambda_{U^c}'(\theta_0)
        \right|
        =
        \lambda_{U^c}''(\xi_s)
        |\theta_s-\theta_0|,\qquad \text{ for some } \xi_s \in (\theta_s,\theta_0).
\label{eq:theta-s-theta-0-MVT-explain}
\end{align} Using this along with \eqref{eq:V-Uc-asymptotic} and \eqref{eq:Lambda2-local-correct},  
\begin{align}
      |\theta_s-\theta_0|
        \le
        C
        \frac{|s|}{\log N_{\varepsilon,b}}, \qquad   \left|
        \log {\theta_s} - \log {\theta_0}
        \right|
        \le
        C
        \frac{|s|}{\log N_{\varepsilon,b}} ,
\label{eq:log-theta-ratio}
\end{align}
where the latter inequality follows from local Lipschitz property of $x \mapsto \log x$ along with $\theta_0,\theta_s \in J.$
Next, 
by the mean value theorem again,
\[
        |v_s-v_0|
        \overset{\eqref{eq:v-close-V-local-correct}}{=}
        \left|
        \lambda_{U^c}''(\theta_s)
        -
        \lambda_{U^c}''(\theta_0)
        \right|
        =    
        |\lambda_{U^c}'''(\xi_s)|
        \,|\theta_s-\theta_0|,\qquad \exists \xi_s \in (\theta_s,\theta_0).
\]
Now the third derivative $\lambda'''_{j}(\theta)$ is given by the third central moment $\E_{\mathbb Q_{j,\theta}}
        \left(
        Y_j-\mathbb E_{\mathbb Q_{j,\theta}}Y_j
        \right)^3.$ {More generally, the $k^{th}$ derivative of the logarithm of the moment generating function is given by the $k^{th}$ cumulant. While they are the same as the $k^{th}$ central moment for $k=2,3$, the expressions begin to differ from $k=4$ onwards. See e.g., \cite[Appendix B.8.3]{FriedliVelenik2017}.}
Thus,  
\begin{align}
        |\lambda_{U^c}'''(\theta)|     \le
        \sum_{j\in U^c}
        |\lambda_j'''(\theta)|   =    \sum_{j\in U^c}  \left| \mathbb E_{\mathbb Q_{j,\theta}}
        \left(
        Y_j-\mathbb E_{\mathbb Q_{j,\theta}}Y_j
        \right)^3
        \right |
        &\le
        \sum_{j\in U^c}
        \mathbb E_{\mathbb Q_{j,\theta}}
        \left|
        Y_j-\mathbb E_{\mathbb Q_{j,\theta}}Y_j
        \right|^3     
\overset{\eqref{eq:third-moment-sum-local}}{\le} 
        C.
\label{eq:lambda-third-derivative-bound-explain}
\end{align}
Hence, using  \eqref{eq:log-theta-ratio},
\begin{align} \label{eq:vs-v0-comparison}
    |v_s-v_0|
        \le
        C
        \frac{|s|}{\log N_{\varepsilon,b}},\qquad  \left|
        \log {v_s} - \log {v_0}
        \right|
        \le
        C
        \frac{|v_s-v_0|}{\log N_{\varepsilon,b}}
        \le
        C
        \frac{|s|}{(\log N_{\varepsilon,b})^2} ,
\end{align} 
where the latter inequality follows from the mean value theorem along with  $v_s,v_0 > (\log N_{\varepsilon,b})/2$ (see \eqref{eq:v-close-V-local-correct}). This covers the first two terms on the RHS of \eqref{eq:G-ratio-log-expanded} and hence it remains to analyze
$$ \log(1+\Lambda(\theta_s\sqrt{v_s}))
        -
        \log(1+\Lambda(\theta_0\sqrt{v_0})).$$

\nin
Using the triangle inequality and the preceding bounds along with the fact that $\theta_0, \theta_s \in J$ and hence are bounded,
\begin{align}
        \left|
        \theta_s\sqrt{v_s}
        -
        \theta_0\sqrt{v_0}
        \right| 
        &\le
        \sqrt{v_s}\,|\theta_s-\theta_0|
        +
        \theta_0
        \left|
        \sqrt{v_s}-\sqrt{v_0}
        \right| \notag\\
        &\le
        C\sqrt{\log N_{\varepsilon,b}}\,
        \frac{|s|}{\log N_{\varepsilon,b}}
        +
        C
        \frac{|v_s-v_0|}
        {\sqrt{v_s}+\sqrt{v_0}}  \le
        C
        \frac{|s|}{\sqrt{\log N_{\varepsilon,b}}}.
\label{eq:xs-x0-bound}
\end{align} 
 Observe that 
 $$ 0< \Lambda'(x)
        =
        \int_0^\infty
        e^{-t}
        e^{-t^2/(2x^2)}
        \frac{t^2}{x^3}
        \,dt 
        \le
        \frac1{x^3}
        \int_0^\infty t^2e^{-t}\,dt
        \le
        Cx^{-3}.$$ 
Thus noting that 
\(\theta_s\sqrt{v_s},\theta_0\sqrt{v_0}\asymp\sqrt{\log N_{\varepsilon,b}}\) (since \(\theta_s,\theta_0 \in J\) 
and \(v_s,v_0\asymp\log N_{\varepsilon,b}\)),  by the mean value theorem,
\begin{align}
        |\Lambda(\theta_s\sqrt{v_s})-\Lambda(\theta_0\sqrt{v_0})|
        &\le
        C
        (\log N_{\varepsilon,b})^{-3/2}
        |\theta_s\sqrt{v_s}-\theta_0\sqrt{v_0}| \overset{\eqref{eq:xs-x0-bound}}{\le} 
        C
        \frac{|s|}
        {(\log N_{\varepsilon,b})^2}.
\label{eq:delta-xs-difference}
\end{align} 
Since \(1-e^{-u}\le u\) for \(u\ge0\), we have $$ e^{-t^2/(2x^2)}-1
        \ge
        -\frac{t^2}{2x^2}.$$ 
Thus, for $x>10$, we have $$ \Lambda(x)
        \ge
        -\frac1{2x^2}
        \int_0^\infty t^2e^{-t}\,dt
        =
        -\frac1{x^2} \ge -\frac{1}{100}.$$
Since $\theta_s\sqrt{v_s},\theta_0\sqrt{v_0}\asymp\sqrt{\log N_{\varepsilon,b}}$ the above lower bound applies to $\Lambda(\theta_0\sqrt{v_0})$ and 
$\Lambda(\theta_s\sqrt{v_s}).$
This along with \eqref{eq:delta-xs-difference} and the mean value theorem implies that
\begin{align*}
   \big\vert  \log(1+\Lambda(\theta_s\sqrt{v_s}))
        -
        \log(1+\Lambda(\theta_0\sqrt{v_0})) \big\vert  \le    C
        \frac{|s|}
        {(\log N_{\varepsilon,b})^2}.
\end{align*}
Applying this together with \eqref{eq:log-theta-ratio} and  \eqref{eq:vs-v0-comparison} to \eqref{eq:G-ratio-log-expanded}, we obtain
\begin{align}
        \left|
        \log\frac{\mathfrak G(s)}{\mathfrak G(0)}
        \right|
        \le
        C
        \frac{|s|}{\log N_{\varepsilon,b}},
        \qquad |s|\le A.
\label{eq:Gaussian-ratio-bound-final}
\end{align}
 This along with  \eqref{eq:K-ratio-vs-G-ratio} yield that as $\varepsilon \rightarrow 0$ (uniformly in \(|s|\le A\) and $U$ satisfying \eqref{condit}),
\begin{align}
        \left|
        \log\frac{K_{U^c}(s)}{K_{U^c}(0)}
        \right|
        \le
        C
        \left(
        \frac{|s|}{\log N_{\varepsilon,b}}
        +
       (\log \varepsilon^{-1})^{-\eta/2 + o(1)}
        \right) .
\label{eq:K-ratio-final-with-log-error}
\end{align}

\medskip
\noindent
\textbf{Step 5. Local ratio estimate.}
From \eqref{eq:q-local-asym-correct},
\begin{align}
        \frac{q_{U^c}(s)}{q_{U^c}(0)}
        =
        \exp\left\{
        \mathcal I_{U^c}(t_{\eta,\alpha})
        -
        \mathcal I_{U^c}(t_{\eta,\alpha}-s)
        \right\}
        \frac{K_{U^c}(s)}{K_{U^c}(0)}.
\label{eq:q-ratio-step5}
\end{align}
By Lemma~\ref{lem:rate-difference-saddle-curve} below which is {another straightforward consequence of convex duality}
\[
        \left|
        \mathcal I_{U^c}(t_{\eta,\alpha})
        -
        \mathcal I_{U^c}(t_{\eta,\alpha}-s)
        \right|
        \le
        C|s|.
\]
Applying this and  \eqref{eq:K-ratio-final-with-log-error} to \eqref{eq:q-ratio-step5}, we obtain \eqref{eq:local-ratio-final-correct}.

 \qed

The proof of  Proposition \ref{cor:full-sum-upper-tail-from-local-ratio-relaxed} is now just based on computing $\mathcal I (t_{\eta,\alpha}).$

\subsection{Computing the rate function: Proof of Proposition \ref{cor:full-sum-upper-tail-from-local-ratio-relaxed}}\label{sharptail23}  

Set
\[
        S
        :=
        \sum_{j=1}^{N_{\varepsilon,b} - \ell_{\varepsilon,\eta}}Y_j,
        \qquad   \lambda (\theta)
        :=
        \log \mathbb E e^{\theta S} \ \  (\forall \theta>0),
        \qquad
        \mathcal I (x)
        :=
        \sup_{\theta>0}
        \{\theta x- \lambda (\theta)\}.
\]  
Then taking the logarithm, it suffices to estimate $ \mathbb P(S\ge t_{\eta,\alpha})$ (see \eqref{deft} for the definition of $    t_{\eta,\alpha}  $).
Taking \(U=\varnothing\) and \(s=0\) in Proposition~\ref{prop:local-ratio-BE-correct} gives
\begin{align} \label{555}
        \mathbb P(S\ge t_{\eta,\alpha})
        =
        \exp\{-\mathcal I (t_{\eta,\alpha})\}K (0),
\end{align}
with  $  
        K (0)
\asymp_\alpha
        \frac{1}{\sqrt{\log   N_{\varepsilon,b}}}.$ 
It remains to compute  
\(\mathcal I (t_{\eta,\alpha})\).   As in \eqref{defj}, set
\begin{align*}
      J:=[\theta_-,\theta_+],\qquad  \theta_-:=
        \frac{\alpha}{2(1-\eta)},
        \qquad
        \theta_+:=
        \frac{2\alpha}{1-\eta}.
\end{align*}
{By Lemma \ref{lem:Mi-expansion-full}, $\lambda(\theta)$ is $C^1$ in the interior of $J$ for sufficiently small $\varepsilon>0$.} For any small enough constant $b>0$, 
let \(\theta_0  \in J\) be the unique solution of $\lambda '(\theta )
        =
        t_{\eta,\alpha}$. This last claim is guaranteed by the  Lemma \ref{lem:compact} stated and proved below. By  \eqref{eq:Lambda1-local-correct} with  \(U=\varnothing\),
\[
         \lambda '(\theta_0)
        =
        \left(\theta_0-\frac12\right)V+O(1).
\]
The  definition and asymptotic of $V$ was recorded in \eqref{eq:V-asymptotic-with-b-gap}.
This implies that
\begin{align}
       \theta_0
        =
        \frac12+\frac{t_{\eta,\alpha}}{V}
        +
        O(V^{-1}).
\label{eq:theta0-before-b-gap}
\end{align}
In addition, by \eqref{eq:Mi0-full} in 
Lemma~\ref{lem:Mi-expansion-full} and using $  \mathbb E[W_j^2] \le C m_j^{-1} \rightarrow 0, $ as $\varepsilon \rightarrow
 0$ (uniformly in \(j \in \{1,2,\cdots,N_{\varepsilon,b} - \ell_{\varepsilon,\eta}\} \)) along with the fact that $\log(1+x)=x+O(x^2)$ for $|x|\le 1/2,$ we obtain 
\[
        \log M_j(\theta_0)
        =
        \frac{\theta_0(\theta_0-1)}2\mathbb E[W_j^2]
        +
        O(m_j^{-3/2}) .
\]
Summing over \(j \in \{1,2,\cdots,N_{\varepsilon,b} - \ell_{\varepsilon,\eta}\} \), noting that $\lambda(\theta_0)  = \sum_j \log M_j(\theta_0),$
\begin{align}
        \lambda (\theta_0)
        =
        \frac{\theta_0(\theta_0-1)}2V 
        +
        O(1) .
\label{eq:Lambda0-local-correct}
\end{align}
Thus  recalling $\lambda '(\theta_0 )
        =
        t_{\eta,\alpha}$, by Legendre duality (see Lemma \ref{lem:interior-legendre-duality} below),  
\begin{align}
        \mathcal I(t_{\eta,\alpha})
        &=
        \theta_0 t_{\eta,\alpha}
        -
        \lambda(\theta_0) \notag\\
        &\overset{\eqref{eq:Lambda0-local-correct}}{=}
        \theta_0 t_{\eta,\alpha}
        -
        \frac{\theta_0(\theta_0-1)}2V
        +
        O(1) \notag \\ 
        &\overset{\eqref{eq:theta0-before-b-gap}}{=}
        {\left(
        \frac12+\frac{t_{\eta,\alpha}}{V}
        \right)t_{\eta,\alpha}
        -
        \frac12
        \left[
        \left(
        \frac12+\frac{t_{\eta,\alpha}}{V}
        \right)
        \left(
        -\frac12+\frac{t_{\eta,\alpha}}{V}
        \right)
        \right]V
        +
        O(1)} \notag \\
        &=
        \frac{t_{\eta,\alpha}}2
        +
        \frac{t_{\eta,\alpha}^2}{V}
        -
        \frac12
        \left(
        \frac{t_{\eta,\alpha}^2}{V^2}
        -
        \frac14
        \right)V
        +
        O(1) =
        \frac{(t_{\eta,\alpha}+V/2)^2}{2V}
        +
        O(1).
\label{eq:I-before-completing-square}
\end{align}  
Recalling from \eqref{eq:V-asymptotic-with-b-gap}
that {$a_{\varepsilon,b}=\frac{V}{\log N_{\varepsilon,b}}$},\begin{align*}
        t_{\eta,\alpha}+\frac V2
        &\overset{\eqref{deft}}{=}
        \left(
        -\frac{1-\eta}{2}+\alpha
        \right)\log N_{\varepsilon,b}
        +
        \frac12a_{\varepsilon,b}\log N_{\varepsilon,b} =
        \left(
        \alpha+\frac12(a_{\varepsilon,b}-(1-\eta))
        \right)\log N_{\varepsilon,b}.
\end{align*}
Substituting this into \eqref{eq:I-before-completing-square}, we obtain
\begin{align}
        \frac{\mathcal I(t_{\eta,\alpha})}{\log N_{\varepsilon,b}} 
        &=
        \frac{
        \left(
        \alpha+\frac12(a_{\varepsilon,b}-(1-\eta))
        \right)^2
        }{
        2a_{\varepsilon,b}
        }
        +
        o(1).
\label{eq:I-coefficient-with-aepsb}
\end{align}
   Hence by \eqref{eq:V-asymptotic-with-b-gap},  there exists 
\(\widetilde \nu_b\)  (depending on $\alpha$)   with \(\widetilde \nu_b\to0\)  as \(b\downarrow0\)  such that  
\begin{align}
        \mathcal I(t_{\eta,\alpha})
        =
        \left(
        \frac{\alpha^2}{2(1-\eta)}
        +
      \widetilde \nu_b
        +
        o(1)
        \right)
        \log N_{\varepsilon,b}.
\label{eq:I-with-b-gap}
\end{align}
Applying this to \eqref{555} finishes the proof.
\qed

\vspace{.2in}

As a quick corollary we obtain Lemma \ref{gconcen} which was used in the proof of Proposition \ref{prop:Z-mass-concentration-from-Geta1-Geta2}.

\subsection{Size biased mass concentration}\label{sizebias23}
Let us begin by recalling the statement. 

\begin{lemma} \label{gconcen1}
There exists
\(\widetilde\nu_b \ge0\)  with \(\widetilde\nu_b\to0\) as \(b\downarrow0\)  such that the following holds.
For \(\eta \in (0,1/10)\) and \(\gamma\in(0,1/10)\), define the event
\[
        \mathcal V_{\eta,\gamma}
        :=
        \left\{
         N_{\varepsilon,b}^{1/2-\eta/2-\gamma}
        \le
        G_{\eta;1}
        \le
         N_{\varepsilon,b}^{1/2-\eta/2+\gamma}
        \right\}.
\]
Then for any constant  \(\xi>0\), the following holds for any small enough constant $b>0$. For all sufficiently small $\varepsilon>0,$
\begin{equation}
        \mathbb E\left[
        G_{\eta;1}\mathbf 1_{\mathcal V_{\eta,\gamma}^c}
        \right]
        \le
         N_{\varepsilon,b}^{-\frac{\gamma^2}{2(1-\eta)}+\xi}.
\label{eq:size-biased-Geta1-concentration123}
\end{equation}
\end{lemma}
This shows that the total (size-biased) mass of  \(G_{\eta;1}\) is concentrated on the scale $N_{\varepsilon,b}^{1/2-\eta/2}$ (recall that $\E G_{\eta;1}=1$). The proof will proceed by splitting the real line into small bins and using the probability bounds from \eqref{sharptail676}. 

\begin{proof} It will be convenient to first handle the case when $G_{\eta;1}$ is too large.
Let \(B >0\) be a large constant such that
\begin{equation}
        \mathbb E\left[
        G_{\eta;1}^2
        \mathbf 1_{\{G_{\eta;1}> N_{\varepsilon,b}^B\}}
        \right]
        \le
         N_{\varepsilon,b}^{-100}.
\label{eq:Geta1-high-tail-L22}
\end{equation}
We justify the existence of such a cutoff \(B\). By Lemma \ref{gmoment}, we have $\mathbb E\left[G_{\eta;1}^4\right]
        \le
        N_{\varepsilon,b}^{C }$ for some constant $C>0$. Hence,  for a large enough constant $B>0$,
\begin{align*}
        \mathbb E\left[
        G_{\eta;1}^2
        \mathbf 1_{\{G_{\eta;1}>N_{\varepsilon,b}^B\}}
        \right]
        &\le   N_{\varepsilon,b}^{-2B}
        \mathbb E\left[G_{\eta;1}^4  \mathbf 1_{\{G_{\eta;1}>N_{\varepsilon,b}^B\}}\right] \le 
        N_{\varepsilon,b}^{-2B}
        \mathbb E\left[G_{\eta;1}^4\right]  \le
        N_{\varepsilon,b}^{C  -2B}  \le
         N_{\varepsilon,b}^{-100}.
\end{align*} 
Now we fix a small mesh \(h_0>0\). First consider the upper side $  \{G_{\eta;1}> N_{\varepsilon,b}^{1/2-\eta/2+\gamma}\}.$ 
Let
\[
        \rho_k:=\frac12-\frac{\eta}{2}+\gamma+kh_0,
        \qquad k=0,1,\ldots,K,
\]
where \(K\) is chosen so that \(\rho_K\le B<\rho_{K+1}\). For the bin $   N_{\varepsilon,b}^{\rho_k}
        \le
        G_{\eta;1}
        \le
         N_{\varepsilon,b}^{\rho_k+h_0},$ by Corollary \ref{cor:full-sum-upper-tail-from-local-ratio-relaxed}, there exists
\(\widetilde \nu_b\ge0\) (depending on \(k=0,1,\ldots,K\)) with \(\widetilde \nu_b\to0\) as \(b\downarrow0\)  such that for any small enough constant $b>0$,
\begin{equation}
\begin{aligned}
\mathbb E\left[
G_{\eta;1}\mathbf 1_{\{
 N_{\varepsilon,b}^{\rho_k}
\le G_{\eta;1}\le
 N_{\varepsilon,b}^{\rho_k+h_0}
\}}
\right]
&\le
 N_{\varepsilon,b}^{\rho_k+h_0}
\mathbb P(G_{\eta;1}> N_{\varepsilon,b}^{\rho_k})         \le
        N_{\varepsilon,b}^{\rho_k+h_0} \cdot
        N_{\varepsilon,b}^{
        -\frac{(\rho_k+1/2-\eta/2)^2}{2(1-\eta)}
        + \widetilde \nu_b
        +o(1)
        } .
\end{aligned}
\label{eq:size-biased-upper-bin-eta}
\end{equation}  
Since there are only finitely many bins, we may choose \(\widetilde\nu_b\) large
enough so that the error bound is uniform in \(k=0,1,\ldots,K\); this does not
affect the property \(\widetilde\nu_b\to0\) as \(b\downarrow0\).
Since $$   \rho
        -
        \frac{(\rho+1/2-\eta/2)^2}{2(1-\eta)}
        =
        -\frac{(\rho-(1/2-\eta/2))^2}{2(1-\eta)}$$ 
and \(\rho_k\ge 1/2-\eta/2+\gamma\), the contribution from each upper-side bin is bounded by $
         N_{\varepsilon,b}^{-\frac{\gamma^2}{2(1-\eta)}+h_0+\widetilde\nu_b+o(1)}.$ 
The number of bins is fixed, and
the remaining part \(G_{\eta;1}> N_{\varepsilon,b}^B\) is negligible due to \eqref{eq:Geta1-high-tail-L22}. Hence the contribution from the upper 
side is at most
\begin{align} \label{upperside}
     N_{\varepsilon,b}^{-\frac{\gamma^2}{2(1-\eta)}+h_0+\widetilde\nu_b+o(1)}.
\end{align}
Next, consider the lower side $ \{G_{\eta;1}< N_{\varepsilon,b}^{1/2-\eta/2-\gamma}\}.$  Again we start from the case when $G_{\eta;1}$ is too small. 
Let \(a_0 \in (0,1/100)\) be  a small constant. Then
\[
        \mathbb E\left[
        G_{\eta;1}
        \mathbf 1_{\{G_{\eta;1}< N_{\varepsilon,b}^{-1/2+\eta/2+a_0}\}}
        \right]
        \le
         N_{\varepsilon,b}^{-1/2+\eta/2+a_0}.
\]
Let $h_0' \in (0,a_0 / 10)$ be a small mesh size similarly as before, and set
\[
        \rho'_k:=\frac12-\frac{\eta}{2}-\gamma-kh_0',
        \qquad k=0,1,\ldots,K',
\]
where \(K'\) is chosen so that 
$   \rho'_{K'+1}\le -\frac12+\frac{\eta}{2}+a_0<\rho'_{K'}.$  By Corollary \ref{cor:full-sum-upper-tail-from-local-ratio-relaxed} again, similarly as above,
\[
\begin{aligned}
\mathbb E\left[
G_{\eta;1}\mathbf 1_{\{
 N_{\varepsilon,b}^{\rho'_k-h_0'}
\le G_{\eta;1}\le
 N_{\varepsilon,b}^{\rho'_k}
\}}
\right]
&\le
 N_{\varepsilon,b}^{\rho'_k}
\mathbb P(G_{\eta;1}> N_{\varepsilon,b}^{\rho'_k-h_0'})   \le
 N_{\varepsilon,b}^{
\rho'_k
-
\frac{(\rho'_k-h_0'+1/2-\eta/2)^2}{2(1-\eta)} +\widetilde\nu_b
+o(1)
}.
\end{aligned}
\]
Using again $\rho
        -
        \frac{(\rho+1/2-\eta/2)^2}{2(1-\eta)}
        =
        -\frac{(\rho-(1/2-\eta/2))^2}{2(1-\eta)},$ 
the last exponent is at most $        -\frac{\gamma^2}{2(1-\eta)}+10 h_0'+\widetilde\nu_b+o(1)$ 
uniformly over the lower-side bins. Hence the contribution from the lower
side is at most
\begin{align}
         N_{\varepsilon,b}^{-\frac{\gamma^2}{2(1-\eta)}+10 h_0'+ \widetilde\nu_b+o(1)}
        +
         N_{\varepsilon,b}^{-1/2+\eta/2+a_0}.
\end{align}
Since   \(a_0,h_0,h_0'>0\)  are arbitrary small constants, recalling that $\widetilde\nu_b \rightarrow 0$ as $b\downarrow 0,$ this along with \eqref{upperside} conclude the proof.
\end{proof}

\nin
We end this section with the following lemma and its proof which was used in the proof of Proposition \ref{prop:local-ratio-BE-correct}.

\begin{lemma}\label{lem:compact}
Let \(\alpha,A>0\) and \(\eta\in(0,1/10)\) be constants.  Define
\begin{align}
        J:=[\theta_-,\theta_+],
        \qquad
        \theta_-:=
        \frac{\alpha}{2(1-\eta)},
        \qquad
        \theta_+:=
        \frac{2\alpha}{1-\eta}.
\label{eq:theta-pm}
\end{align}
Then there exists \(b_0=b_0(\alpha )>0\) such that the following holds for
every constant \(b\in(0,b_0)\).  Let $ U\subset\{1,2,\dots,M_{\varepsilon,\eta}\}$ 
be  such that $\sum_{j\in U}m_j^{-1}\le 100.$  
Then, for all sufficiently small \(\varepsilon>0\) (uniformly in
\(|s|\le A\) and uniformly over all such subsets \(U\)), \(\lambda_{U^c}'\) is strictly increasing on \(J\), and there exists a unique
\(\theta_s\in J\) satisfying
\begin{align}
         \lambda_{U^c}'(\theta_s)
        =
        t_{\eta,\alpha}-s.
\label{eq:saddle-eq11}
\end{align} 
\end{lemma}

\begin{proof}
Since \(\alpha>0\), we have \(J\subset(0,\infty)\).  Choose an open interval
\(O\subset(0,\infty)\) such that \(J\subset O\).  By
Lemma~\ref{lem:Mi-expansion-full}, the functions \(M_j\), and therefore
\(\lambda_{U^c}\), are \(C^2\) on \(O\) for all sufficiently small
\(\varepsilon>0\).  Consequently, \(\lambda_{U^c}'\) and
\(\lambda_{U^c}''\) are well-defined and continuous in \(J\).

By the
variance estimate \eqref{eq:V-asymptotic-with-b-gap} and our assumption $\sum_{j\in U}m_j^{-1}\le 100$,  
\begin{equation}
        V_{U^c}
        =
        \left(
        1-\eta+{\nu_b}+o(1)
        \right)
        \log N_{\varepsilon,b}.
\label{eq:V-Uc-asymptotic-b-gap}
\end{equation}  
By \eqref{eq:Lambda1-local-correct} and recalling the definition of $  t_{\eta,\alpha}$ in  \eqref{deft}, for any small enough constant $b>0$  (depending on $\alpha$),
\begin{align}
         \lambda _{U^c} '(\theta_-)
        &=
        \left(\theta_- -\frac12\right)
        V_{U^c}
        +
        O(1) \notag\\
        &\overset{\eqref{eq:V-Uc-asymptotic-b-gap}}{=}
        \left(
        \left(
        \frac{\alpha}{2(1-\eta)}-\frac12
        \right)
        (1-\eta+\nu_b)
        +
        o(1)
        \right)
        \log N_{\varepsilon,b}  <
        t_{\eta,\alpha}-s
\label{eq:saddle-lower-trap}
\end{align}
for all sufficiently small \(\varepsilon\), uniformly in \(|s|\le A\) and
uniformly over such \(U\).   
Similarly,
\begin{align}
         \lambda _{U^c} '(\theta_+)
        &=
        \left(\theta_+ -\frac12\right)
        V_{U^c}
        +
        O(1) \notag\\
        &\overset{\eqref{eq:V-Uc-asymptotic-b-gap}}{=}
        \left(
        \left(
        \frac{2\alpha}{1-\eta}-\frac12
        \right)
        (1-\eta+\nu_b)
        +
        o(1)
        \right)
        \log N_{\varepsilon,b} >
        t_{\eta,\alpha}-s
\label{eq:saddle-upper-trap}
\end{align}
for all sufficiently small \(\varepsilon\), uniformly in \(|s|\le A\) and
uniformly over such \(U\).  

Note that by \eqref{eq:Lambda2-local-correct}, for any small enough constant $b>0 $, 
\[
        \lambda_{U^c}''(\theta)
        =
        V_{U^c}+O(1) \overset{\eqref{eq:V-Uc-asymptotic-b-gap}}{>} 0.1 \log N_{\varepsilon,b}
\] 
for all sufficiently small \(\varepsilon\), uniformly  in
\(\theta\in J=[\theta_-,\theta_+]\) and such \(U\).   Thus \(\lambda_{U^c}'\) is
continuous and strictly increasing on \(J\).
Hence by \eqref{eq:saddle-lower-trap} and \eqref{eq:saddle-upper-trap}, the
intermediate value theorem gives a solution \(\theta_s\in J\) of
\eqref{eq:saddle-eq11}.  Since \(\lambda_{U^c}'\) is strictly increasing on
\(J\), this solution is unique in \(J\).
\end{proof}

\section{Radon-Nikodym derivative}\label{rnproof34}
In this section, relying on Proposition \ref{prop:local-ratio-BE-correct}, we prove the all-important stability estimate 
 for the conditional law under projection onto a small set of coordinates, presented earlier as Proposition \ref{prelimcond}. We will in fact state and prove the following more general result. 
  Recall that the upper-tail event $$\textup{UT}_{\varepsilon,\eta}(\alpha)=\left\{\sum _{i=1}^{M_{\varepsilon, \eta}} \log G_i
\ge
\bigl( -\frac{1-\eta}{2}+\alpha \bigr)\log N_{\varepsilon,b}
\right\}$$  is imposed on the full bulk sum
$\sum_{i=1}^{M_{\varepsilon, \eta}}\log G_i,$
whereas in our
applications {(proof of Proposition \ref{prop:second-moment-small-total-length-full}, see \eqref{rnrn})} we needed to analyze the conditional law projected on a small subset \(U\) of the
coordinates.
 The following proposition shows that, if $\sum_{j\in U}\frac1{m_j}$ is small enough then conditioning on
the upper-tail event has a negligible effect on the
marginal law of \((G_i)_{i\in U}\). The reason why such a statement is expected to be true was outlined in Section \ref{iop}.

\begin{proposition}
\label{thm:direct-Lp-no-split-corrected}
Let \(\alpha>0\). 
Then there exists \(b_0=b_0(\alpha )>0\) such that the following holds for
any constant \(b\in(0,b_0)\). Let $\eta \in (0,1/10)$ and $  U\subset\{1,2,\dots,M_{\varepsilon, \eta}\}$ 
be a subset such that $ \Delta_U:= \sum_{j\in U} m_j^{-1}$ satisfies
\begin{align} \label{verysmallcon}
      |U| \cdot   \Delta_U = o\left(\frac{1}{\log N_{\varepsilon,b}}  \right) ,\qquad \varepsilon \rightarrow 0.
\end{align}
(This is our smallness hypothesis on $|U|.$)
Define the Radon-Nikodym derivative
\begin{align} \label{rnderiv}
\mathsf{RN}_U:=
\frac{
d \mathrm{Law}\bigl((G_i)_{i\in U}\mid \textup{UT}_{\varepsilon,\eta}(\alpha) \bigr)
}{
d \mathrm{Law}((G_i)_{i\in U})
}.  
\end{align}
Then for any integer \(p\ge2\),   as \(\varepsilon \rightarrow 0\),
\begin{equation}
\mathbb E|\mathsf{RN}_U-1|^p
\le
C(\log N_{\varepsilon,b})^{p/2}|U|^{p/2}\Delta_U^{p/2} + 
       (\log \varepsilon^{-1})^{-p\eta/2 + o(1)}.
\label{eq:Rminus1-Lp-final}
\end{equation}
In particular, when $|U| = O(\log \log \varepsilon^{-1}),$ the assumption \eqref{verysmallcon} is automatically satisfied and
\begin{equation}
\mathbb E\bigl[|\mathsf{RN}_U-1|^p\bigr]
\le  
       (\log \varepsilon^{-1})^{-p\eta/2 + o(1)}.
\label{rnspecial}
\end{equation}
\end{proposition}
Thus, in words, the above statement is a quantitative assertion of the closeness of the conditional and unconditional distributions, but crucially only when projected onto the coordinates in $U$.

\nin
The proof of the above relies on the sharp upper-tail probability at levels $t_{\eta,\alpha}-s$ for the sum $S_{U^c}$ as $s$ varies, i.e., $q_{U^c}(s)=\P(S_{U^c}\ge t_{\eta, \alpha}-s)$ defined in \eqref{uppertailprob345}.
The latter exactly determines the Radon-Nikodym derivative of the conditional distribution projected on $U$. If $q_{U^c}(s)$ did not change with $s$ (which is obviously false), then the conditional distribution on $U$ would be exactly the unconditional distribution. Nonetheless, the same intuition applies whenever it is regular enough and this is indeed what was already shown in \eqref{eq:local-ratio-final-correct} and will be our key input.  

\begin{proof} 
Set
\[
          H_{U^c}(s):=\frac{q_{U^c}(s)}{q_{U^c}(0)},\qquad \forall s\in \mathbb R.
\]

\medskip
\noindent
\textbf{Step 1.}
Since \(S_{U^c}\) is independent of \((G_i)_{i\in U}\), for every bounded
measurable \(f\),
\[
\begin{aligned}
\mathbb E\bigl[f((G_i)_{i\in U})\mathbf 1_{\textup{UT}_{\varepsilon,\eta}(\alpha)}\bigr]
&=
\mathbb E \left[
f((G_i)_{i\in U})
\mathbb P(S_{U^c}\ge t_{\eta,\alpha}-S_U\mid (G_i)_{i\in U})
\right]  =
\mathbb E\bigl[f((G_i)_{i\in U}) q_{U^c}(S_U)\bigr].
\end{aligned}
\]
Hence
\[
        \mathsf{RN}_U=\frac{q_{U^c}(S_U)}{\mathbb P(\textup{UT}_{\varepsilon,\eta}(\alpha))}.
\]
Since  $ \mathbb P(\textup{UT}_{\varepsilon,\eta}(\alpha))
        =
        \mathbb E[q_{U^c}(S_U)]
        =
        q_{U^c}(0)\mathbb E[H_{U^c}(S_U)],$ 
we have
\begin{equation}
\mathsf{RN}_U
=
\frac{H_{U^c}(S_U)}
{\mathbb E[H_{U^c}(S_U)]}.
\label{eq:R-normalized-H-Lp-final}
\end{equation}

\medskip
\noindent
\textbf{Step 2. Pointwise bounds for \(H_{U^c}(s)-1\).}
We claim the following bounds:  For any small enough $b>0,$ as $\varepsilon \rightarrow 0,$
\begin{align}
|H_{U^c}(s)-1|^p
&\le C |s|^p+  
       (\log \varepsilon^{-1})^{-p\eta/2 + o(1)},
&& |s|\le1,
\label{eq:Hminus1-local-p-final}
\\
|H_{U^c}(s)-1|^p
&\le
C(\log N_{\varepsilon,b})^{p/2}|s|^p e^{p\theta_{+} s},
&& s\ge1,
\label{eq:Hminus1-pos-p-final}
\\
|H_{U^c}(s)-1|^p
&\le1,
&& s\le -1.
\label{eq:Hminus1-neg-p-final}
\end{align}
Indeed, on \(|s|\le1\), {and this is where all the action lies},  \eqref{eq:local-ratio-final-correct} in Proposition~\ref{prop:local-ratio-BE-correct} implies $$  |\log H_{U^c}(s)|\le C(|s|+ 
       (\log \varepsilon^{-1})^{-\eta/2 + o(1)}),$$ thus
\[
|H_{U^c}(s)-1|
=
|e^{\log H_{U^c}(s)}-1| 
\le
C(|s|+ 
       (\log \varepsilon^{-1})^{-\eta/2 + o(1)}).
\]

\medskip

\nin
However to compute moments, we still need to use some crude bounds when $|s|\ge 1.$\\

\nin
The case $s\le -1,$ i.e., \eqref{eq:Hminus1-neg-p-final} is a consequence of $0\le H_{U^c}(s)\le1  \ (s\le 0)$ which follows from monotonicity of \(q_{U^c}\).\\

\nin
For the case $s\ge 1$ we need a bit more preparation.
For $\theta>0$, define
\[
        \lambda_{U^c}(\theta)
        :=
        \log\mathbb E[e^{\theta S_{U^c}}],\qquad  \mathcal I_{U^c}(x)
        :=
        \sup_{\theta>0}
        \left\{
        \theta x-\lambda_{U^c}(\theta)
        \right\}.
\]
Let $J  = [\theta_-,\theta_+] \subset (0,\infty)$ be a compact interval defined in \eqref{eq:theta-pm}.
By Lemma \ref{lem:compact} (with $s=0$), there exists a unique solution \(\theta_0 \in J \) to $\lambda_{U^c}'(\theta )
        =
        t_{\eta,\alpha}.$ 
Then by Legendre duality (see Lemma \ref{lem:interior-legendre-duality} below), 
\begin{align}\label{legend}
  \mathcal I_{U^c}(t_{\eta,\alpha})+\lambda_{U^c}(\theta_0 ) =    \theta_0 t_{\eta,\alpha}.
\end{align}
By Chernoff's inequality with the tilt \(\theta_0 >0\), for all
\(s\in\mathbb R\),  
\[
\begin{aligned}
q_{U^c}(s)
&=
\mathbb P(S_{U^c}\ge t_{\eta,\alpha}-s)  \le
e^{-\theta_0 (t_{\eta,\alpha}-s)}
\mathbb E[e^{\theta_0 S_{U^c}}]  =
e^{-\theta_0 t_{\eta,\alpha}
+\lambda_{U^c}(\theta_0 )}
e^{\theta_0 s} \overset{\eqref{legend}}{=}
e^{-\mathcal I_{U^c}(t_{\eta,\alpha})}
e^{\theta_0 s}.
\end{aligned}
\] 
On the other hand, Proposition~\ref{prop:local-ratio-BE-correct}, applied at
\(s=0\), gives that for any small enough $b>0,$ as $\varepsilon \rightarrow 0,$
\[
q_{U^c}(0)
=
e^{-\mathcal I_{U^c}(t_{\eta,\alpha})}
K_{U^c}(0),
\qquad
K_{U^c}(0)
\asymp_\alpha
\frac{1}{\sqrt{\log N_{\varepsilon,b}}}.
\]
Therefore for all
\(s\in\mathbb R\), 
\begin{align}\label{eq:H-global-positive}
H_{U^c}(s)
=
\frac{q_{U^c}(s)}{q_{U^c}(0)}
\le
\frac{
e^{-\mathcal I_{U^c}(t_{\eta,\alpha})}
e^{\theta_0 s}
}{
e^{-\mathcal I_{U^c}(t_{\eta,\alpha})}
K_{U^c}(0)
}
\le
C\sqrt{\log N_{\varepsilon,b}}\, e^{\theta_0 s} .
\end{align}
Thus for \(s\ge1\), by \eqref{eq:H-global-positive} and recalling $\theta_0 \in  J=[\theta_{-}, \theta_{+}]$,
\[
|H_{U^c}(s)-1|
\le
H_{U^c}(s)+1
\le
C\sqrt{\log N_{\varepsilon,b}} e^{\theta_{+} s}.
\]

\medskip
\noindent
\textbf{Step 3. Tilted \(p\)-moment of \(S_U\).} To bound moments of $|H_{U^c}(s)-1|$ using expressions as in \eqref{eq:Hminus1-local-p-final} we must first control moments of $S_U.$ Towards this we claim that, for any   \(\theta_*>0\) and \(p\ge2\),
\begin{equation}
\mathbb E\bigl[|S_U|^p e^{\theta_* S_U}\bigr]
\le
C_{p,\theta_*} |U|^{p/2} \Delta_U^{p/2}.
\label{eq:fixed-tilt-p-moment-self-final}
\end{equation}
{Due to the condition \(\theta_*>0\),} the factor
\(e^{\theta_* S_U}\) damps the singularity of \(Y_i=\log G_i\) near
\(G_i=0\). {While we state this general bound, the singularity at $0$ is not critical for us, since in the corresponding regime
the observable \(H_{U^c}(S_U)\) remains close to \(1\).} 
%when $s<-1$, i.e., $0\le H_{U^c}(s)\le1$  

Define, for $\theta>0$,
\[
M_{U} (\theta)
:=
\mathbb E[e^{\theta S_U}],
\qquad
\frac{d\mathbb Q_{\theta ; U}}{d\mathbb P}
: = 
\frac{e^{\theta S_U}}{M_{U} (\theta)}.
\]
Then
\begin{equation}
\mathbb E\bigl[|S_U|^p e^{\theta_* S_U}\bigr]
=
M_{U} (\theta_*)
\mathbb E_{\mathbb Q_{\theta_* ; U}}|S_U|^p.
\label{eq:tilted-factorization-fixed-final}
\end{equation} Since $      M_i(\theta_*)=1+O(m_i^{-1})$ by \eqref{eq:Mi0-full} in Lemma \ref{lem:Mi-expansion-full},  we have $ M_{U} (\theta_*) = \prod_{i\in U} M_i (\theta_*) = \Theta(1)$ (recall the condition on  \(\Delta_U \) in \eqref{verysmallcon}).\\

\nin
Note that
 under \(\mathbb Q_{\theta_* ; U}\), the variables
\((Y_i)_{i\in U}\) remain independent, and the Radon-Nikodym derivative with respect to $\mathbb P$ is $\frac{e^{\theta_*Y_i}}{M_i(\theta_*)}
=
\frac{G_i^{\theta_*}}{M_i(\theta_*)}.$ 
By the bound \eqref{eq:tilted-Y-third-moment-bound} in the proof of Lemma \ref{lem:tilted-third-moment-full},  
\begin{equation}
\mathbb E_{Q_{\theta_* ; U}}|Y_i|^p
\le
C m_i^{-p/2}.
\label{eq:one-coordinate-positive-tilt-final}
\end{equation} 
Thus  by Minkowski's inequality in \(L^p(\mathbb Q_{\theta_*;U})\),
\begin{align}\label{minp}
        \left(
        \mathbb E_{\mathbb Q_{\theta_*;U}}|S_U|^p
        \right)^{1/p}
      =
        \left\|
        \sum_{i\in U}Y_i
        \right\|_{L^p(\mathbb Q_{\theta_*;U})}                  \le
        \sum_{i\in U}
        \|Y_i\|_{L^p(\mathbb Q_{\theta_*;U})}                  \le
        C 
        \sum_{i\in U}m_i^{-1/2}\le    C   |U|^{1/2}\Delta_U^{1/2},
\end{align}  
where we used Cauchy-Schwarz inequality in the last inequality.
Applying this to  \eqref{eq:tilted-factorization-fixed-final}, recalling $ M_{U} (\theta_*) = \Theta(1)$, this proves
\eqref{eq:fixed-tilt-p-moment-self-final}.\\

The above preparation allows us to quickly deduce \(L^p\)-bounds for \(H_{U^c}(S_U)-1\).

\medskip
\noindent
\textbf{Step 4. \(L^p\)-bound for \(H_{U^c}(S_U)-1\).}
  We decompose
\begin{align*}
\mathbb E|H_{U^c}(S_U)-1|^p
&=
\mathbb E\bigl[|H_{U^c}(S_U)-1|^p;\ |S_U|\le1\bigr]
\\
&\quad+
\mathbb E\bigl[|H_{U^c}(S_U)-1|^p;\ S_U\ge1\bigr]
+
\mathbb E\bigl[|H_{U^c}(S_U)-1|^p;\ S_U\le-1\bigr].
\end{align*}

\nin
$(1)$
On \(\{|S_U|\le1\}\) which is the typical behavior as apparent from, say, \eqref{minp},  using \eqref{eq:Hminus1-local-p-final}, 
\begin{align} 
\mathbb E \bigl[&|H_{U^c}(S_U)-1|^p ;\ |S_U|\le1\bigr]
 \notag \\
       &\le
C\mathbb E\bigl[|S_U|^p;\ |S_U|\le1\bigr]+   
       (\log \varepsilon^{-1})^{-p\eta/2 + o(1)} \notag \\
       &\le
C\mathbb E\bigl[|S_U|^p e^{S_U}\bigr] + 
       (\log \varepsilon^{-1})^{-p\eta/2 + o(1)}\overset{\eqref{eq:fixed-tilt-p-moment-self-final}}{\le}
C|U|^{p/2}\Delta_U^{p/2}+  
       (\log \varepsilon^{-1})^{-p\eta/2 + o(1)}. \label{eq:Hp-local-expectation-final}
\end{align}

\nin
$(2)$
On \(\{S_U\ge1\}\),  using \eqref{eq:Hminus1-pos-p-final},
\begin{equation}
\mathbb E\bigl[|H_{U^c}(S_U)-1|^p;\ S_U\ge1\bigr]
\le
C(\log N_{\varepsilon,b})^{p/2}
\mathbb E\bigl[|S_U|^p e^{ p\theta_{+}S_U}\bigr]
\overset{\eqref{eq:fixed-tilt-p-moment-self-final}}{\le}
C(\log N_{\varepsilon,b})^{p/2} |U|^{p/2} \Delta_U^{p/2}.
\label{eq:Hp-positive-expectation-final}
\end{equation}

\nin
$(3)$
On \(\{S_U\le -1\}\),  using \eqref{eq:Hminus1-neg-p-final},
\begin{align}
\mathbb E\bigl[|H_{U^c}(S_U)-1|^p;\ S_U\le-1\bigr]
\le
\mathbb P(S_U\le-1).
\label{-1part}
\end{align}

\nin
Therefore it remains to bound \(\mathbb P(S_U\le-1)\). We claim that for
{any \(r > 0 \) for which $2r$ is integer,}
\begin{align}
\mathbb P(S_U\le -1)\le C|U|^r\Delta_U^{r}.
\label{-1claim}
\end{align}
Define the (rare) event
\[
\mathcal E
:=
\left\{\exists i\in U:\ W_i\le -\frac12\right\}.
\]
Then
\begin{align}
\mathbb P(S_U\le-1)
\le
\mathbb P(
\mathcal E)
+
\mathbb P(S_U\le-1,\ 
\mathcal E^c).
\label{531}
\end{align}
By the union bound and Markov's inequality,
\begin{align}
\mathbb P(
\mathcal E)
&\le
\sum_{i\in U}\mathbb P\left(W_i\le-\frac12\right)  \le
\sum_{i\in U}2^{2r}\mathbb E|W_i|^{2r}  \overset{\eqref{eq:assump-allmom-direct-Lp-final}}{\le}
2^{2r}C\sum_{i\in U}m_i^{-r}
\le C \Big(\sum_{i\in U}m_i^{-1}\Big)^r = 
C\Delta_U^{r}.
\label{eq:EN-bound}
\end{align}
On \(\mathcal E^c\), we have \(W_i>-1/2\) for every \(i\in U\).
Since $ \log(1+w)\ge w-w^2$ for $w>-1/2,$
on \(\mathcal E^c\),
\[
S_U
=
\sum_{i\in U}\log(1+W_i)
\ge
\sum_{i\in U}W_i-\sum_{i\in U}W_i^2.
\]
Hence the second term in \eqref{531} is bounded by
\begin{equation}
\mathbb P(S_U\le -1,\ 
\mathcal E^c)
\le
\mathbb P\left(\sum_{i\in U}W_i\le-\frac12\right)
+
\mathbb P\left(\sum_{i\in U}W_i^2\ge\frac12\right).
\label{eq:split-on-ENc}
\end{equation}
For the first term, by Markov's inequality and Minkowski's inequality,
\begin{align*}
\mathbb P\left(\sum_{i\in U}W_i\le-\frac12\right)
&\le
2^{2r}
\mathbb E\left|
\sum_{i\in U}W_i
\right|^{2r} \le
C
\left(
\sum_{i\in U}
\|W_i\|_{L^{2r}}
\right)^{2r}   \overset{\eqref{eq:assump-allmom-direct-Lp-final}}{\le}
C
\left(
\sum_{i\in U}m_i^{-1/2}
\right)^{2r} \le
C |U|^r\Delta_U^r,
\end{align*}
where we used Cauchy-Schwarz inequality in the last inequality. For the second term, again by Markov's inequality and Minkowski's inequality,
\begin{align*}
\mathbb P\left(\sum_{i\in U}W_i^2\ge\frac12\right)
&\le
2^r
\mathbb E\left(\sum_{i\in U}W_i^2\right)^r            
\le
C
\left(
\sum_{i\in U}
\|W_i^2\|_{L^r}
\right)^r   \overset{\eqref{eq:assump-allmom-direct-Lp-final}}{\le}
C
\left(
\sum_{i\in U}m_i^{-1}
\right)^r
=
C\Delta_U^r.
 \end{align*}
Together with \eqref{eq:EN-bound}, this proves
\eqref{-1claim}. Using \eqref{-1part} and \eqref{-1claim} with $r=p/2$, we deduce that
\begin{equation}
\mathbb E\bigl[|H_{U^c}(S_U)-1|^p;\ S_U\le-1\bigr]
\le
C|U|^{p/2}\Delta_U^{p/2}.
\label{eq:Hp-negative-expectation-final}
\end{equation}
Combining this with \eqref{eq:Hp-local-expectation-final} and
\eqref{eq:Hp-positive-expectation-final}, we conclude that
\begin{equation}
\mathbb E|H_{U^c}(S_U)-1|^p
\le
C(\log N_{\varepsilon,b})^{p/2}|U|^{p/2}\Delta_U^{p/2} +   
       (\log \varepsilon^{-1})^{-p\eta/2 + o(1)}. 
\label{eq:Hminus1-Lp-self-final}
\end{equation}

\medskip
\noindent
\textbf{Step 5: Conclusion.}
By Jensen's inequality,  as   $\log N_{\varepsilon,b}\cdot |U| \cdot \Delta_U = o(1)$  (by the assumption in \eqref{verysmallcon}), as $\varepsilon \rightarrow 0,$
\begin{align}
\bigl|\mathbb E[H_{U^c}(S_U)]-1\bigr|
&\le
\mathbb E|H_{U^c}(S_U)-1|  \le
\Bigl(\mathbb E|H_{U^c}(S_U)-1|^p\Bigr)^{1/p}
\overset{\eqref{eq:Hminus1-Lp-self-final}}{\longrightarrow}0.
\label{conc}
\end{align}
Hence, for all sufficiently small $\varepsilon>0$, we have $  \mathbb E[H_{U^c}(S_U)]\ge\frac12.$ 
Thus from \eqref{eq:R-normalized-H-Lp-final},
\[
  |\mathsf{RN}_U-1|^p  
=
\frac{
|H_{U^c}(S_U)-\mathbb E[H_{U^c}(S_U)]|^p
}{
(\mathbb E[H_{U^c}(S_U)])^p
}
\le
2^p
|H_{U^c}(S_U)-\mathbb E[H_{U^c}(S_U)]|^p.
\]
Also, by the inequality $ |x-y|^p\le2^{p-1}(|x-1|^p+|y-1|^p),$ 
we have
\[
\begin{aligned}
&|H_{U^c}(S_U)-\mathbb E[H_{U^c}(S_U)]|^p  \le
2^{p-1}
\Bigl(
|H_{U^c}(S_U)-1|^p
+
|\mathbb E[H_{U^c}(S_U)]-1|^p
\Bigr).
\end{aligned}
\]
Taking expectations and using \eqref{conc},
\[
\mathbb E
|H_{U^c}(S_U)-\mathbb E[H_{U^c}(S_U)]|^p
\le
C
\mathbb E|H_{U^c}(S_U)-1|^p.
\]
Thus combining the above inequalities,
\begin{align*}
    \mathbb E |\mathsf{RN}_U-1|^p \le   C
\mathbb E|H_{U^c}(S_U)-1|^p.
\end{align*}
Using \eqref{eq:Hminus1-Lp-self-final}, we establish \eqref{eq:Rminus1-Lp-final}. Finally,  recalling  \eqref{mimin}, the assumption \eqref{verysmallcon} is satisfied whenever $|U| = O(\log \log \varepsilon^{-1})$, because $\Delta_U \le |U|\cdot (\log \varepsilon^{-1})^{-\eta}$ and $N_{\varepsilon,b} \asymp_b  \log \varepsilon^{-1}$. Thus we  establish \eqref{rnspecial}  as a consequence of  \eqref{eq:Rminus1-Lp-final}.
\end{proof}

\begin{remark} \label{remarkvary5}
All the results in this section  remain
valid for \(\varepsilon\)-dependent parameters \(b=b_\varepsilon\) and
\(\vartheta=\vartheta_\varepsilon\), provided that for   
some \(\vartheta^\star\in\mathbb R\) and small enough \(b^\star>0\),
\[
        b_\varepsilon\to b^\star,
        \qquad
        \vartheta_\varepsilon\to\vartheta^\star
        \qquad\text{as }\varepsilon\rightarrow 0.
\]
Indeed the second-moment estimate \eqref{eq:V-asymptotic-with-b-gap} holds for such parameters as well: There exists 
\(\nu_{b^\star}\ge0\) (depending on $\vartheta^\star$) with $\nu_{b^\star} \rightarrow0$ as  $b^\star\downarrow0$ such that  as $\varepsilon \rightarrow 0,$
\begin{equation}
\sum_{j=1}^{N_{\varepsilon,b}-\ell_{\varepsilon,\eta}}
        \mathbb E\left[
        W_j^2
        \right]
        =
        a_\varepsilon\log N_{\varepsilon,b},
        \qquad
        \left|
        a_\varepsilon-(1-\eta)
        \right|
        \le
        \nu_{b^\star}+o(1),
\label{eq:V-asymptotic-varying-b-theta}
\end{equation}
We briefly explain why this holds.  By Lemma \ref{vary}, the above sum is written as
\begin{align} \sum_{m=\ell_{\varepsilon,\eta}}^{N_{\varepsilon,b}-1}
        \frac1m
        +
        O\left(
        \omega_{[\vartheta^\star-1,\vartheta^\star+1]}(2b^\star)
        \sum_{m=\ell_{\varepsilon,\eta}}^{N_{\varepsilon,b}-1}
        \frac1m
        \right).
\label{eq:V-varying-reduced-to-harmonic}
\end{align} 
It remains to estimate the harmonic sum.  Since \(b_\varepsilon\to b^\star\),  we have $N_{\varepsilon,b}
        \asymp_{b^\star}
        \log\varepsilon^{-1}.$ Thus
\begin{align}
        \sum_{m=\ell_{\varepsilon,\eta}}^{N_{\varepsilon,b}-1}
        \frac1m
        &=
        \log N_{\varepsilon,b}
        -
        \log \ell_{\varepsilon,\eta}
        +
        o(1) =
        (1-\eta+o(1))\log N_{\varepsilon,b} .
\label{eq:harmonic-varying-b}
\end{align}
Substituting this into
\eqref{eq:V-varying-reduced-to-harmonic},  as $\omega_{[\vartheta^\star-1,\vartheta^\star+1]}(2b^\star) \rightarrow 0$ as $b^\star \downarrow 0$, we obtain \eqref{eq:V-asymptotic-varying-b-theta}.
Given the variance-sum estimate
\eqref{eq:V-asymptotic-varying-b-theta} and the 
higher-moment bounds \eqref{eq:centered-G-moment-B3}, the arguments in this section work without any modifications.

\end{remark}

\appendix

\section{Monotonicity of SHF in $\vartheta$}

In this appendix, we prove Lemma \ref{lem:moment-monotonicity-theta} establishing the monotonicity of the moments of SHF in the coupling parameter $\vartheta$.

\begin{proof}[Proof of Lemma \ref{lem:moment-monotonicity-theta}]
We use the collision-diagram moment formula \cite{shfuppermoment}. For a nonnegative test function
\(\phi\), the \(h\)-th moment  of SHF  can be written as
\[
\mathbb E\left[
\left(
Z_t^\vartheta(\phi)
\right)^h
\right]
=
\sum_{m=0}^\infty (2\pi)^m
\sum_{\substack{
\{i_1,j_1\},\ldots,\{i_m,j_m\}\subset\{1,\ldots,h\}^2\\
\{i_k,j_k\}\neq\{i_{k+1},j_{k+1}\}
}}
\mathcal I_m^\vartheta
\bigl(
\{i_1,j_1\},\ldots,\{i_m,j_m\}
\bigr),
\]
where, for every integer \(m\ge1\), {writing $\textbf{x}  = (x^1,\ldots,x^h)$,}
\[
\begin{aligned}
\mathcal I_m^\vartheta
&=
\int_{(\mathbb R^2)^h} d\textbf{x}
\phi^{\otimes h}(\mathbf{x})
\int_{0\le a_1<b_1<\cdots<a_m<b_m\le t}
\int_{(\mathbb R^2)^{2m}}
g_{a_1/2}(x_1-x^{i_1})g_{a_1/2}(x_1-x^{j_1})            \\
&\qquad \times
\prod_{r=1}^{m}
\left[
G_\vartheta(b_r-a_r)
g_{(b_r-a_r)/4}(y_r-x_r)
\mathbf 1_{\mathsf{S}_{i_r,j_r}}
\right]                                                   \\
&\qquad \times
\prod_{1\le r\le m-1}
g_{(a_{r+1}-b_{\mathsf{p}(i_{r+1})})/2}
\bigl(x_{r+1}-y_{\mathsf{p}(i_{r+1})}\bigr)
g_{(a_{r+1}-b_{\mathsf{p}(j_{r+1})})/2}
\bigl(x_{r+1}-y_{\mathsf{p}(j_{r+1})}\bigr)
\,d\vec x\,d\vec y\,d\vec a\,d\vec b .
\end{aligned}
\] 
{Here, \(\mathsf{S}_{i_r,j_r}\) denotes the event that Brownian motions
indexed by \(i_r\) and \(j_r\) collide, and \(\mathsf{p}(i_r)\) denotes
the most recent time prior to \(r\) at which the Brownian motion indexed by
\(i_r\) participated in a collision (see \cite[Section 2]{shrinking} for details).}
The important point is that every factor in the integrand is nonnegative whenever
\(\phi\ge0\), and the parameter \(\vartheta\) appears only through the
 factors $G_\vartheta(b_r-a_r).$ 
By the explicit representation
\[
        G_\vartheta(t)
        =
        \int_0^\infty
        e^{(\vartheta-\gamma)s}
        \frac{t^{s-1}}{\Gamma(s)}\,ds,
        \qquad t>0
\]
($\gamma$ is the Euler constant and $\Gamma(s)$ is the Gamma function),
we have, for \(\vartheta_1\le\vartheta_2\),
\[
        G_{\vartheta_1}(t)
        \le
        G_{\vartheta_2}(t),
        \qquad t>0.
\] 
Since all other factors in the diagram integral are nonnegative and independent
of \(\vartheta\), we obtain 
\[
        \mathcal I_m^{\vartheta_1}
        \bigl(
        \{i_1,j_1\},\ldots,\{i_m,j_m\}
        \bigr)
        \le
        \mathcal I_m^{\vartheta_2}
        \bigl(
        \{i_1,j_1\},\ldots,\{i_m,j_m\}
        \bigr).
\]
Summing over \(m\) and over all diagrams proves monotonicity.

\end{proof}

Next, as a consequence of Lemma \ref{lem:moment-monotonicity-theta}, we record a uniform variance asymptotic estimate that applies when \(b\) and
\(\vartheta\) are allowed to depend on \(\varepsilon\), provided that both parameters
converge as \(\varepsilon\downarrow0\).
We write $\mathscr W_{i}^{\vartheta}
        :=
        \mathscr Z_{t_{i-1},t_i}^{\vartheta}
        -
        P_{t_{i-1},t_i}$  for the centered SHF. Note that
\[
        p_{i-1}         \blacktriangleleft         \mathscr W_i^{\vartheta}         \blacktriangleright1=G_i^{\vartheta}-1.
\]
For the moment, we keep the dependence on the coupling parameter $\vartheta$ explicit in the notation
\(G_i^\vartheta\).

\begin{lemma} 
\label{vary}   

For every compact interval
\(K=[\vartheta_-,\vartheta_+]\subset\mathbb R\), define
\[
        \omega_K(\bar b)
        :=
        \sup_{0<b\le \bar b}
        \sup_{\varepsilon\le1/2}
        \sup_{1\le i\le N_{\varepsilon,b}-1}
        \sup_{\vartheta\in K}
        \left|
        (N_{\varepsilon,b}-i)
        \mathbb E\left[
        \left(
        p_{i-1}         \blacktriangleleft         \mathscr W_i^{\vartheta}         \blacktriangleright1
        \right)^2
        \right]
        -1
        \right|.
\]
Then
\begin{equation}
        \omega_K(\bar b)\longrightarrow0
        \qquad\text{as }\bar b\downarrow0.
\label{eq:variance-envelope-small-b}
\end{equation}
In particular, if \(b_\varepsilon\to b^\star > 0 \) and
\(\vartheta_\varepsilon\to\vartheta^\star\), then,  for any compact
interval \(K\) containing \(\vartheta^\star\), the following holds: For all sufficiently
small \(\varepsilon>0\),
\begin{equation}
        \sup_{1\le i\le N_{\varepsilon,b_\varepsilon}-1}
        \left|
        (N_{\varepsilon,b_\varepsilon}-i)
        \mathbb E\left[
        \left(
       p_{i-1}         \blacktriangleleft         \mathscr W_i^{\vartheta_\varepsilon}         \blacktriangleright1 
        \right)^2
        \right]
        -1
        \right|
        \le
        \omega_K(2b^\star).
\label{eq:variance-varying-b-small-fixed}
\end{equation}
Note that the bound in RHS above tends to zero as  $b^\star\downarrow0.$ 
\end{lemma}

\begin{proof}
For fixed \(\varepsilon,b,i\), we set
\[
        A_{\varepsilon,b,i}(\vartheta)
        :=
        (N_{\varepsilon,b}-i)
        \mathbb E\left[
        \left(
        p_{i-1}         \blacktriangleleft         \mathscr W_i^{\vartheta}         \blacktriangleright1
        \right)^2
        \right].
\]
Note that
since \(\mathbb E G_i^{\vartheta}=1\),  
\[
        \mathbb E\left[
        \left(
        p_{i-1}         \blacktriangleleft         \mathscr W_i^{\vartheta}         \blacktriangleright1
        \right)^2
        \right]
        =
        \mathbb E\left[
        \left(
        G_i^{\vartheta}-1
        \right)^2
        \right]
        =
        \mathbb E\left[
        \left(
        G_i^{\vartheta}
        \right)^2
        \right]-1.
\] 
Hence by Lemma \ref{lem:moment-monotonicity-theta}, \(A_{\varepsilon,b,i}(\vartheta)\) is nondecreasing in \(\vartheta\).
This implies that  for every \(\vartheta\in[\vartheta_-,\vartheta_+]\), 
\[
\begin{aligned}
        \sup_{\vartheta\in K}
        |A_{\varepsilon,b,i}(\vartheta)-1|
        &\le
        \max\left\{
        |A_{\varepsilon,b,i}(\vartheta_-)-1|,
        |A_{\varepsilon,b,i}(\vartheta_+)-1|
        \right\}.
\end{aligned}
\]
Taking the supremum over
\(0<b\le\bar b\), \(\varepsilon\le1/2\), and
\(1\le i\le N_{\varepsilon,b}-1\), we obtain  
\[
\begin{aligned}
        \omega_K(\bar b)
        &\le
        \max_{\sigma\in\{-,+\}}
        \sup_{0<b\le\bar b}
        \sup_{\varepsilon\le1/2}
        \sup_{1\le i\le N_{\varepsilon,b}-1}
        \left|
        (N_{\varepsilon,b}-i)
        \mathbb E\left[
        \left(
        p_{i-1}         \blacktriangleleft         \mathscr W_i^{\vartheta_\sigma}         \blacktriangleright1
        \right)^2
        \right]
        -1
        \right|.
\end{aligned}
\]
For the two fixed endpoint parameters \(\vartheta_-\) and
\(\vartheta_+\), the variance estimate \eqref{2moment} gives that the
RHS tends to \(0\) as \(\bar b\downarrow0\).  This proves
\eqref{eq:variance-envelope-small-b}.

Now suppose \(b_\varepsilon\to b^\star > 0\) and
\(\vartheta_\varepsilon\to\vartheta^\star\). Take a compact interval
\(K\subset\mathbb R\) such that $\vartheta_\varepsilon\in K$ 
for all sufficiently small \(\varepsilon\). As $ b_\varepsilon\le 2b^\star$  for all sufficiently small \(\varepsilon\), 
 we obtain \eqref{eq:variance-varying-b-small-fixed}.  
\end{proof}

\section{Radon--Nikodym identity for projected conditioning}

The next lemma was a key input in Section \ref{centralsec}. 
In summary, it shows that when an observable under consideration depends only on part of the
underlying randomness and on some independent auxiliary randomness, the effect of a global conditioning can be represented by a Radon--Nikodym derivative factor involving only the variables participating in the observable.

\begin{lemma} 
\label{lem:projected-RN-coarse-conditioning}
Let \(X_1,\ldots,X_m\) be independent random elements, possibly taking values
in abstract measurable spaces.  Let $ Y=(Y_1,\ldots,Y_{m'})$ 
be a random vector independent of \((X_1,\ldots,X_m)\).  For each
\(i=1,\ldots,m\), let
\[
        \widetilde X_i:=\psi_i(X_i)
\]
be a measurable real-valued function of \(X_i\).  Let $A\in\sigma(\widetilde X_1,\ldots,\widetilde X_m)$
be an event with \(\mathbb P(A)>0\).  For
\(K\subset\{1,\ldots,m\}\), write $ X_K:=(X_i)_{i\in K}$ and $\widetilde X_K:=(\widetilde X_i)_{i\in K}.$ 
Define the projected Radon--Nikodym derivative
\begin{equation}
        \mathsf{RN}_{K}^{A}
        :=
        \frac{
        d\textup{Law}(\widetilde X_K\mid A)
        }{
        d\textup{Law}(\widetilde X_K)
        } .
\label{eq:projected-RN-coarse-definition}
\end{equation}
 We use the convention $\mathsf{RN}_{\varnothing}^{A}\equiv1$ when \(K=\varnothing\).
Then, for every integrable random variable \(\Phi\) which is measurable with
respect to $\sigma(X_K)\vee\sigma(Y),$
we have
\begin{equation}
        \mathbb E[\Phi\mid A]
        =
        \mathbb E\left[
        \Phi\cdot \mathsf{RN}_{K}^{A}(\widetilde X_K)
        \right].
\label{eq:projected-RN-coarse-identity}
\end{equation}
Moreover,
\begin{equation}
        \mathsf{RN}_{K}^{A}(\widetilde X_K)
        =
        \mathbb E\left[
        \frac{\mathbf 1_A}{\mathbb P(A)}
        \,\middle|\,
        \sigma (\widetilde X_K)
        \right]
        =
        \mathbb E\left[
        \frac{\mathbf 1_A}{\mathbb P(A)}
        \,\middle|\,
        \sigma(X_K)\vee\sigma(Y)
        \right]
        \qquad\textup{a.s.}
\label{eq:projected-RN-coarse-cond-exp}
\end{equation}
\end{lemma}

\begin{proof}

 We prove the first equality in \eqref{eq:projected-RN-coarse-cond-exp}.
For every bounded measurable function \(f:\mathbb R^{|K|} \rightarrow \mathbb R\),
\[
\begin{aligned}
        \mathbb E\left[
        f(\widetilde X_K) \cdot 
        \mathsf{RN}_{K}^{A}(\widetilde X_K)
        \right]
        &=
        \mathbb E\left[
        f(\widetilde X_K)
        \,\middle|\,
        A
        \right]     =
        \mathbb E\left[
        f(\widetilde X_K) \cdot
        \frac{\mathbf 1_A}{\mathbb P(A)}
        \right].
\end{aligned}
\] 
The second equality  in \eqref{eq:projected-RN-coarse-cond-exp} follows from the fact that  \(X_1,\ldots,X_m\) are independent and \(Y\) is
independent of all of them: Recalling that \(A\in
\sigma(\widetilde X_1,\ldots,\widetilde X_m) \subseteq \sigma(X_1,\ldots, X_m)\), 
\begin{align}\label{pilam}
        \mathbb E\left[
        \mathbf 1_A
        \,\middle|\,
        \sigma(X_K)\vee\sigma(Y)
        \right]
        =\mathbb E\left[
        \mathbf 1_A
        \,\middle|\,
        \sigma(X_K)
        \right]=
        \mathbb E\left[
        \mathbf 1_A
        \,\middle|\,
        \sigma (\widetilde X_K)
        \right].    
\end{align}
{Note that the first identity is immediate from the independence of $\{X_1,X_2,\ldots, X_m\}$ and $Y$, and we quickly verify the second one.
Write \(K^c=\{1,\ldots,m\}\setminus K\).
Consider first an event of the form
\[
    A=B\cap C,
    \qquad
    B\in\sigma(\widetilde{X}_K),
    \quad
    C\in\sigma(\widetilde{X}_{K^c}).
\]
Using \(B\in\sigma(\widetilde{X}_K)\subseteq\sigma(X_K)\) and the
independence of \(\sigma(X_K)\) and
\(\sigma(\widetilde{X}_{K^c})\), we obtain
\[
    \mathbb E\left[
        \mathbf 1_{B\cap C}
        \,\middle|\,
        \sigma(X_K)
    \right]
    =
    \mathbf 1_B\,\mathbb P(C).
\]
Similarly, since   \(\sigma(\widetilde{X}_K)\) and
\(\sigma(\widetilde{X}_{K^c})\) are independent, 
\[
    \mathbb E\left[
        \mathbf 1_{B\cap C}
        \,\middle|\,
        \sigma(\widetilde{X}_K)
    \right]
    =
    \mathbf 1_B\,\mathbb P(C).
\]
Thus, the second identity in \eqref{pilam} holds for all events
\(B\cap C\) of the above form.  These events form a \(\pi\)-system
generating
\[
    \sigma\bigl(
        \sigma(\widetilde{X}_K),
        \sigma(\widetilde{X}_{K^c})
    \bigr)
    =
    \sigma(\widetilde X_1,\ldots,\widetilde X_m).
\]
Since the collection of events for which the second identity in
\eqref{pilam} holds is a Dynkin system,
the \(\pi\)-\(\lambda\) theorem proves the claim.}

Finally, we prove \eqref{eq:projected-RN-coarse-identity}. Using the second equality in
\eqref{eq:projected-RN-coarse-cond-exp},  recalling that \(\Phi\)  is measurable with
respect to $\sigma(X_K)\vee\sigma(Y),$
\[
\begin{aligned}
        \mathbb E\left[
        \Phi\cdot
        \mathsf{RN}_{K}^{A}(\widetilde X_K)
        \right]
        &=
        \mathbb E\left[
        \Phi \cdot 
        \mathbb E\left[
        \frac{\mathbf 1_A}{\mathbb P(A)}
        \,\middle|\,
        \sigma(X_K)\vee\sigma(Y)
        \right]
        \right]                                                  \\
        &=
        \mathbb E\left[
        \Phi \cdot \frac{\mathbf 1_A}{\mathbb P(A)}
        \right]  =
        \mathbb E[\Phi\mid A].
\end{aligned}
\]

\end{proof}

In our applications, the random elements  are the SHF measures $Z_i=\mathscr Z^\vartheta_{t_{i-1},t_i}$. Since the objects are not standard and are quite singular, for completeness we provide the details needed to verify the measurability hypothesis for the above lemma to be applicable. 
Let us recall from Section \ref{sec2.1} the measure space supporting the SHF.  We have $E:=\mathbb R^2\times\mathbb R^2,$ 
and \(\mathcal M_+ = \mathcal M_+(E)\) the space of positive locally finite Borel
measures on \(E\), equipped with the Borel \(\sigma\)-field induced by the
vague topology:
\[
        \mathcal B _{\mathrm{vag}}(\mathcal M_+)
        :=
        \sigma\left(
        \mu\mapsto \int_E \varphi\,d\mu:
        \varphi\in C_c(E)
        \right).
\]
Then each SHF increment $Z_k=\mathscr Z^\vartheta_{t_{k-1},t_k}$ 
is viewed as a random element in $(\mathcal M_+,\mathcal B _{\mathrm{vag}}(\mathcal M_+)).$ 

{Recall that the proof of Proposition \ref{prop:second-moment-small-total-length-full}   uses the above lemma in \eqref{rnrn}, with the variables $D_i$ playing the role of the measurable function $\Phi$  in  \eqref{eq:projected-RN-coarse-identity}.}
Thus we must verify the measurability  of the random variable $D_I$ (with $I = [i_I,j_I]$) with respect to \(\sigma(Z_{i_I},\ldots,Z_{j_I})\),  where we write 
\[
        \sigma(Z_i,\ldots,Z_j)
        :=
 { (Z_i,\ldots,Z_j)^{-1}
        \left(
        \mathcal B_{\mathrm{vag}}(\mathcal M_+)^{\otimes(j-i+1)}
        \right)}.
\]
{where implicitly above we assume that all the random objects \(  Z_i\) are defined on some common probability space \((\Omega,\mathcal F,\mathbb P)\).}
Recall from \eqref{mainobject} that    $$ 
D_I
        =
        p_{i_I-1}
        \blacktriangleleft
        D_{i_I}\bullet D_{i_I+1}\bullet\cdots\bullet D_{j_I-1}
        \bullet Z_{j_I}
        \blacktriangleright1 .$$ where for any index $k,$ $D_k=Z_k-(Z_k\blacktriangleright1)p_k .$
Thus, it is enough to examine the measurability of scalar convolution functionals of the form
\[
  p_{i-1}
        \blacktriangleleft
        Z_i\bullet Z_{i+1}\bullet\cdots\bullet Z_j
        \blacktriangleright1,
\]
since 
\(D_I\) is a finite sum of several products of such objects. This can be seen by expanding $  D_k=Z_k-(Z_k\blacktriangleright1)p_k .$ 
The following lemma verifies that such quantities are  \(\sigma(Z_i,\ldots,Z_j)\)-measurable random variables.  {Recall from \eqref{conv34} that $\bullet$  denotes the limit of regularized convolutions.}

 \begin{lemma}
\label{lem:scalar-singular-convolution-measurable-version}  
Let $1\le i\le j\le N_{\varepsilon,b}$ be any integers. The singular convolution
\[
        \Phi
        :=
        p_{i-1}
        \blacktriangleleft
        Z_i\bullet Z_{i+1}\bullet\cdots\bullet Z_j
        \blacktriangleright1
\]
is \(\sigma(Z_i,\ldots,Z_j)\)-measurable.
\end{lemma}

The definition of $Z_i\bullet Z_{i+1}\bullet\cdots\bullet Z_j$  in \eqref{conv34} defined as a limit in \cite{clark2024continuum} essentially already implies the above lemma. A slight amount of detail is still involved since the space of measures where $Z_i\bullet Z_{i+1}\bullet\cdots\bullet Z_j$ lives is equipped with the vague topology which works particularly well with compactly supported functions which the functions $p_{i-1}$ or $1$ are not. 
The proof thus involves a cutoff approximation.  Namely, after multiplying
\(p(t_{i-1},x)\) and \(1\) by increasing compactly supported cutoffs, each
resulting pairing is measurable by the definition of the vague topology which then monotonically converges to the above object.

\begin{proof}

Let
\(\mathcal F_{i,j}:=\sigma(Z_i,\ldots,Z_j)\), and choose a sequence
\((\chi_R)_{R\ge1}\subset C_c(\mathbb R^2)\) such that
\[
        0\le\chi_R\le1,
        \qquad
        \chi_R\uparrow1
        \quad\text{pointwise on }\mathbb R^2.
\]
For \(R\ge1\), define a continuous compactly support function
\[
        \varphi_R(x,y)
        :=
        \chi_R(x)\,p(t_{i-1},x)\,\chi_R(y),
        \qquad
        (x,y)\in\mathbb R^2\times\mathbb R^2.
\]
For every \(R\ge1\), the evaluation map
\[
        \mathcal M_+\ni\mu
        \longmapsto
        \int_{\mathbb R^2\times\mathbb R^2}
        \varphi_R\,d\mu
\]
is continuous with respect to the vague topology, and is therefore
\(\mathcal B_{\mathrm{vag}}(\mathcal M_+)\)-measurable.   Indeed, by definition,
\(\mu_n\to\mu\) vaguely precisely when
\[
        \int \varphi\,d\mu_n\longrightarrow\int\varphi\,d\mu
        \qquad\text{for every }\varphi\in C_c(\mathbb R^2\times\mathbb R^2).
\]
As \(\varphi_R\in C_c(\mathbb R^2\times\mathbb R^2)\), the claimed
continuity follows immediately.

Since
$\mu_{i,j}:=   Z_i\bullet Z_{i+1}\bullet\cdots\bullet Z_j
$ is an
\(\mathcal F_{i,j}\)-measurable \(\mathcal M_+\)-valued random element (see  \cite[Proposition 2.6]{clark2024continuum}), it
follows, by the above stated continuity, that
\[
        \Phi_R
        :=
        \int_{\mathbb R^2\times\mathbb R^2}
        \varphi_R(x,y)\,
        \mu_{i,j}(dx,dy)
\]
is \(\mathcal F_{i,j}\)-measurable for every \(R\).
Then by the monotone convergence theorem,   we have
\begin{align}
               \lim_{R\to\infty}\Phi_R = \Phi.
\label{eq:scalar-singular-convolution-cutoff-limit}
\end{align}
Thus \(\Phi\) is \(\mathcal F_{i,j}\)-measurable.
Finally,  as $  \mathbb E[\Phi]
=1$ and \(\Phi\ge0\),  we have \(\Phi<\infty\) almost surely.

\end{proof}

As explained before Lemma \ref{lem:scalar-singular-convolution-measurable-version}, we immediately have the following corollary.

\begin{corollary} \label{cormeasure}
 For $I = [i_I,j_I]$, $D_I$  is measurable with respect to \(\sigma(Z_{i_I},\ldots,Z_{j_I})\).    
\end{corollary}

\section{Convex analysis}
In this section, we state and prove some basic facts in convex analysis which featured in Section \ref{sec5}.
\begin{lemma} 
\label{lem:tangent-line-convex}
Let \(I\subset\mathbb R\) be an open interval, and let
\(\lambda:I\to\mathbb R\) be convex in the sense that
\[
        \lambda((1-a)\theta_1+a\theta_2)
        \le
        (1-a)\lambda(\theta_1)+a\lambda(\theta_2),
        \qquad \theta_1,\theta_2\in I,\quad a\in[0,1].
\]
Suppose that \(\lambda\) is differentiable at \(\theta_0\in I\).  Then, for every
\(\theta\in I\),
\begin{equation}
        \lambda(\theta)
        \ge
        \lambda(\theta_0)+\lambda'(\theta_0)(\theta-\theta_0).
\label{eq:tangent-line-convex}
\end{equation}
\end{lemma}

\begin{proof}
The case \(\theta=\theta_0\) is trivial. 
First suppose \(\theta>\theta_0\).  For \(0<h<\theta-\theta_0\), write
\[
        \theta_0+h
        =
        \left(1-\frac{h}{\theta-\theta_0}\right)\theta_0
        +
        \frac{h}{\theta-\theta_0}\theta.
\]
By convexity,
\[
        \lambda(\theta_0+h)
        \le
        \left(1-\frac{h}{\theta-\theta_0}\right)\lambda(\theta_0)
        +
        \frac{h}{\theta-\theta_0}\lambda(\theta).
\]
Rearranging gives
\[
        \frac{\lambda(\theta_0+h)-\lambda(\theta_0)}{h}
        \le
        \frac{\lambda(\theta)-\lambda(\theta_0)}{\theta-\theta_0}.
\]
Letting \(h\downarrow0\), and using differentiability of \(\lambda\) at
\(\theta_0\), we obtain
\[
        \lambda'(\theta_0)
        \le
        \frac{\lambda(\theta)-\lambda(\theta_0)}{\theta-\theta_0}.
\]
Since \(\theta-\theta_0>0\), this is equivalent to
 \eqref{eq:tangent-line-convex}. The case \(\theta<\theta_0\) can be done similarly.
 
\end{proof}

\begin{lemma} 
\label{lem:interior-legendre-duality}
Let \(\lambda:(0,\infty)\to\mathbb R\) be  convex   in the sense that
\begin{align} \label{convexdef}
        \lambda((1-a)\theta_1+a\theta_2)
        \le
        (1-a)\lambda(\theta_1)+a\lambda(\theta_2),
        \qquad
        \theta_1,\theta_2>0,\quad a\in[0,1].
\end{align}
Define
\begin{align} \label{ldual}
        \mathcal I(x)
        :=
        \sup_{\theta>0}
        \{\theta x-\lambda(\theta)\}.
\end{align}
Let $x\in \mathbb R$.
Suppose that  \(\lambda\) is differentiable at some
\(\theta_x\in(0,\infty)\) and that $\lambda'(\theta_x)=x.$ 
Then
\[
        \mathcal I(x)
        =
        \theta_xx-\lambda(\theta_x).
\]
\end{lemma}

\begin{proof}
By Lemma~\ref{lem:tangent-line-convex}, for every \(\theta>0\),
\[
        \lambda(\theta)
        \ge
        \lambda(\theta_x)+\lambda'(\theta_x)(\theta-\theta_x)=
        \lambda(\theta_x)+x(\theta-\theta_x).
\]
Rearranging,
\[
        \theta x-\lambda(\theta)
        \le
        \theta_xx-\lambda(\theta_x),
        \qquad \forall \theta>0.
\]
Applying this to \eqref{ldual}, we have $  \mathcal I(x)
        =
        \theta_xx-\lambda(\theta_x).$ 
\end{proof}
 
The next lemma gives a convenient integral representation for the difference
between nearby values of the rate function, expressed through the unique
parameter solving the variational equation.
\begin{lemma} 
\label{lem:rate-difference-saddle-curve}
 Let $  \lambda:(0,\infty)\to\mathbb R$ be convex in the sense of \eqref{convexdef}.  Define
\[
        \mathcal I(x)
        :=
        \sup_{\theta>0}
        \{\theta x-\lambda(\theta)\}.
\]
Let  \(t\in\mathbb R\)  and
\(\mathcal U\subset\mathbb R\) be an open interval {containing 0}. 
Assume that there exists an interval $   J=[\theta_-,\theta_+]\subset(0,\infty)$  and an open interval \(O\subset(0,\infty)\) with \(J\subset O\), such that
\(\lambda\in C^1(O)\), and such that for every \(u\in\mathcal U\) there is a
unique solution \(\theta_u \in J\) to the equation   $ \lambda'(\theta_u)=t-u$. Then  for every \(s\in \mathcal U\),
\[
        \mathcal I(t)-\mathcal I(t-s)
        =
        \int_0^s\theta_u\,du.
\]
Consequently,
\begin{align} \label{541}
        \left|
        \mathcal I(t)-\mathcal I(t-s)
        \right|
        \le
        \theta_+|s|.
\end{align}
\end{lemma}

\begin{proof}
We first prove continuity of \(u\mapsto\theta_u\) {in $U$}.  Let \(u_n\to u\).  Since
\(\theta_{u_n}\in J\), every subsequence has a further subsequence converging
to some \(\theta_*\in J\).  Along such a subsequence, $    \lambda'(\theta_{u_n})=t-u_n.$ 
Because \(\lambda'\) is continuous on \(J\), passing to the limit gives $        \lambda'(\theta_*)=t-u.$ 
By uniqueness of the solution, \(\theta_*=\theta_u\).  Hence every
subsequential limit is \(\theta_u\), and therefore \(\theta_{u_n}\to\theta_u\).

 For $u\in \mathcal U,$
define
\[
        F(u):=\mathcal I(t-u) =\theta_u(t-u)-\lambda(\theta_u),
\]
where the last identity follows from Lemma~\ref{lem:interior-legendre-duality}.

We claim that \(F'(u)=-\theta_u\). For  \(u\in\mathcal U\),  set \(x_u:=t-u\).  Let \(h\) be small enough that
\(u+h\in\mathcal U\).  
Using \(\theta_u\) as a competitor for the variational problem \(\mathcal I(x_{u+h})\),
\[
       F(u+h)= \mathcal I(x_{u+h})
        \ge
        \theta_u x_{u+h}-\lambda(\theta_u).
\]
Since $  F(u)=  \mathcal I(x_u)=\theta_u x_u-\lambda(\theta_u),$ 
we obtain
\begin{align} \label{543}
        F(u+h)-F(u)
        \ge
        \theta_u(x_{u+h}-x_u)
        =
        -h\theta_u.
\end{align}
Similarly, using \(\theta_{u+h}\) as a competitor for the variational problem \(\mathcal I(x_u)\),
\[
     F(u)=   \mathcal I(x_u)
        \ge
        \theta_{u+h}x_u-\lambda(\theta_{u+h}).
\]
Since $
       F(u+h) =  \mathcal I(x_{u+h})
        =
        \theta_{u+h}x_{u+h}-\lambda(\theta_{u+h}),$ 
we get
\begin{align} \label{544}
        F(u+h)-F(u)
        \le
        \theta_{u+h}(x_{u+h}-x_u)
        =
        -h\theta_{u+h}.
\end{align}
 Since \(\theta_{u+h}\to\theta_u\) as $h\rightarrow 0$ due to the continuity of $u \mapsto \theta_u$,  by  \eqref{543} and \eqref{544},
 \[
        \frac{F(u+h)-F(u)}{h}
        \longrightarrow
        -\theta_u.
\]
Hence \(F\) is differentiable and \(F'(u)=-\theta_u\).  Since
\(u\mapsto\theta_u\) is continuous, \(F'\) is continuous.  Therefore, by the
fundamental theorem of calculus,
\[
        \mathcal I(t)-\mathcal I(t-s)
        =
        F(0)-F(s)
        =
        -\int_0^sF'(u)\,du
        =
        \int_0^s\theta_u\,du.
\]
Finally, \eqref{541} follows from triangle inequality along with the fact that $ |\theta_u| \le \theta_+.$ 
\end{proof}

\section{Bounded-variation test functions}
 We state  the following elementary comparison inequality for bounded-variation test functions.
\begin{lemma}\label{tv}
Assume that \(g\) has bounded variation whose total variation is $ \|g\|_{\mathrm{TV}}$.
Let  \(X\) and \(Y\) be (real-valued) random variables with
distribution functions \(F_X\) and \(F_Y\).  Then,
\[
        \left|
        \mathbb E[g(X)]-\mathbb E[g(Y)]
        \right|
        \le
        \|g\|_{\mathrm{TV}}
        \sup_{x\in\mathbb R}|F_X(x)-F_Y(x)|.
\]
\end{lemma}
\begin{proof}
    
Writing the expectations as Lebesgue--Stieltjes integrals and
integrating by parts gives
\[
        \mathbb E[g(X)]-\mathbb E[g(Y)]
        =
        \int_{\mathbb R} g(x)\,d(F_X-F_Y)(x)
        =
        -\int_{\mathbb R}(F_X(x)-F_Y(x))\,dg(x),
\]
and therefore
\[
        \left|
        \mathbb E[g(X)]-\mathbb E[g(Y)]
        \right|
        \le
        \sup_{x\in\mathbb R}|F_X(x)-F_Y(x)|\, |dg|(\mathbb R)
        =
        \|g\|_{\mathrm{TV}}
        \sup_{x\in\mathbb R}|F_X(x)-F_Y(x)|.
\]
\end{proof}

\bibliographystyle{plain}
\bibliography{SHFLDP}

\end{document}